\documentclass{amsart}

\usepackage{amssymb}
\usepackage{enumerate}  %% Get better looking enumerations

 \numberwithin{equation}{section}

\usepackage{latexsym,amssymb,amsthm,amsmath,amscd}
\theoremstyle{plain}

\newtheorem{theorem}{Theorem}[section]
\newtheorem*{theorem2.3}{Theorem 2.3}
\newtheorem*{theorem2.4}{Theorem 2.4}
\newtheorem*{theorem2.5}{Theorem 2.5}

\newtheorem*{theorem5.1}{Theorem 5.1}
\newtheorem*{theorem5.2}{Theorem 5.2}
\newtheorem*{theorem5.3}{Theorem 5.3}
\newtheorem*{theorem5.4}{Theorem 5.4}

\newtheorem*{theorem7.3}{Theorem 7.3}
\newtheorem*{theorem7.4}{Theorem 7.4}
\newtheorem*{theorem7.7}{Theorem 7.7}
\newtheorem*{theorem8.1}{Theorem 8.1}
\newtheorem*{theorem9.1}{Theorem 9.1}
\newtheorem*{theorem9.2}{Theorem 9.2}
\newtheorem*{theorem10.1}{Theorem 10.1}
\newtheorem*{theorem10.2}{Theorem 10.2}
\newtheorem*{theorem11.2}{Theorem 11.2}
\newtheorem*{theorem12.1}{Theorem 12.1}
\newtheorem*{theorem12.2}{Theorem 12.2}
\newtheorem*{theorem2.7}{Theorem 2.7}
\newtheorem*{theorem2.8}{Theorem 2.8}

\newtheorem*{theorem3.3}{Theorem 3.3}
\newtheorem*{theorem3.4}{Theorem 3.4}

\newtheorem*{theorem3.5}{Theorem 3.5}
\newtheorem*{theorem6.2}{Theorem 6.2}
\newtheorem*{theoremA}{Theorem A}

\newtheorem*{theoremC}{Theorem C}
\newtheorem*{theoremD}{Theorem D}
\newtheorem*{theoremE}{Theorem E}
\newtheorem*{theoremF}{Theorem F}
\newtheorem*{remark1}{Remark 3.1}
\newtheorem*{remark7.1}{Remark 7.1}
\newtheorem*{remark7.2}{Remark 7.2}
\newtheorem*{corollary3.1}{Corollary 3.1}
\newtheorem*{corollary3.5}{Corollary 3.5}
\newtheorem*{corollary3.7}{Corollary 3.7}
\newtheorem*{corollary5.1}{Corollary 5.1}
\newtheorem*{corollary6.1}{Corollary 6.1}
\newtheorem*{lemma7.1}{Lemma 7.1}
\newtheorem*{lemma7.2}{Lemma 7.2}
\newtheorem*{lemma7.3}{Lemma 7.3}
\newtheorem*{corollary8.1}{Corollary 8.1}
\newtheorem*{lemma8.1}{Lemma 8.1}
\newtheorem*{lemA}{Lemma A}

\newtheorem{corollary}{Corollary}[section]
\newtheorem*{definition1}{Definition 2.1}
\newtheorem*{definition5.1}{Definition 5.1}
\newtheorem*{definition5.2}{Definition 5.2}
\newtheorem*{definition5.3}{Definition 5.3}
\newtheorem*{definition6.1}{Definition 6.1}
\newtheorem*{definition6.2}{Definition 6.2}
\newtheorem*{definition7.1}{Definition 7.1}
\newtheorem*{definition8.1}{Definition 8.1}
\newtheorem*{definition8.2}{Definition 8.2}
\newtheorem*{definition11.1}{Definition 11.1}

\numberwithin{equation}{section}

\newtheorem*{Theorem A}{{\bf Theorem A}}
\newtheorem*{Theorem B}{{\bf Theorem B}}
\newtheorem*{Theorem C}{Theorem C}
\newtheorem*{Theorem 5.1}{Theorem 5.1}
\newtheorem*{Theorem 5.2}{Theorem 5.2}
\newtheorem*{remark2.1}{Remark 2.1}

\theoremstyle{definition}

\numberwithin{equation}{section}

\newcommand{\E}{\epsilon}

\begin{document}

\title[Dualities in Comparison Theoryems and Bundle-Valued Generalized Harmonic Forms  
] {Dualities in Comparison Theorems and Bundle-Valued Generalized Harmonic Forms on Noncompact Manifolds}

\author[S.W. Wei]{Shihshu Walter Wei$^{\ast }$}
\address{Department of Mathematics\\
University of Oklahoma\\ Norman, Oklahoma 73019-0315\\ U.S.A.}
\email{wwei@ou.edu}

\begin{abstract} We observe, utilize dualities in differential equations and differential inequalities (see Theorem \ref{T: 2.1}), dualities between comparison theorems in differential equations (see Theorems E and \ref{T: 2.2}), and obtain dualities in ``swapping" 
comparison theorems in differential equations. These dualities generate comparison theorems on differential equations of mixed types I and II (see Theorems \ref{T: 2.3} and \ref{T: 2.4}) and lead to comparison theorems in Riemannian 
geometry (see Theorems \ref{T: 2.5} and \ref{T: 2.8}) with analytic, geometric, P.D.E.'s and physical applications. In particular, we prove Hessian comparison theorems (see Theorems \ref{T: 3.1} - \ref{T: 3.5}) and Laplacian comparison theorems (see Theorems \ref{T: 2.6}, \ref{T: 2.7}, \ref{T: 3.1} - \ref{T: 3.5}) under varied radial Ricci curvature, radial 
curvature, Ricci curvature and sectional curvature assumptions, generalizing and extending the work of Han-Li-Ren-Wei (\cite {HLRW}), and 
Wei (\cite {W3}). We also extend the notion of function or differential form growth to bundle-valued differential form growth of various types and discuss their interrelationship (see Theorem 5.4). These provide tools in extending the notion, integrability and decomposition 
of generalized harmonic forms to those of bundle-valued generalized harmonic forms, introducing Condition W for bundle-valued differential forms, and proving duality theorem and unity theorem, generalizing the work of  Andreotti and Vesentini \cite {AV} and Wei  \cite {W4}.    
We then apply Hessian and Laplacian comparison theorems to obtain comparison theorems in mean curvature, generalized sharp Caffarelli-Kohn-Nirenberg type inequalities on
Riemannian manifolds, embedding theorem  for weighted Sobolev spaces of functions on manifolds, geometric differential-integral inequalities, generalized sharp Hardy type inequalities on
Riemannian manifolds, monotonicity formulas and vanishing theorems for differential forms of degree $k$ with values in vector bundles, such as   $F$-Yang Mills fields (when $F$ is the identity map, they are Yang-Mills fields), generalized Yang-Mills-Born-Infeld
fields on manifolds, Liouville type theorems for $F$-harmonic maps (when $F(t) = \frac 1p (2t)^{\frac{p}{2}}\, , p > 1$, they become $p$-harmonic maps or harmonic maps if $p=2$), and Dirichlet problems on starlike domains for vector bundle valued differential $1$-forms and $F$-harminic maps (see Theorems \ref{T: 4.1}, \ref{T: 7.3} - \ref{T: 7.7}, \ref{T: 8.1}, \ref{T: 9.1} - \ref{T: 9.3}, \ref {T: 10.1}, \ref{T: 11.2}, \ref{T: 12.1}, \ref{T: 12.2}),
generalizing the work of Caffarelli-Kohn-Nirenberg (\cite {CKN}) and Costa (\cite {C}), in which $M=\mathbb{R}^n\, $ and its radial curvature $K(r) = 0\, ,$ 
the work of Wei and Li \cite {WL}, Chen-Li-Wei \cite {CLW1, CLW2}, Dong and Wei \cite {DW}, Wei \cite {W3}, Karcher and Wood \cite {KW}, etc.
The boundary value problem for bundle-valued differential $1$-forms
is in contrast to the Dirichlet problem for $p$-harmonic maps to which the solution
is due to Hamilton \cite {H1} for the case $p=2$ and $Riem^{N}\leq 0\, ,$ and  Wei \cite {W2} for $1 < p < \infty\, .$\end{abstract}


\thanks{\\
$^*$ Research supported in part by NSF (DMS-1447008).}
\date{}
\maketitle
\section{Introduction}
Duality is a special type of symmetry that involves with “polar opposites” and their dynamical interplays. The two (``Yin" and ``Yang", studied by a legendary sage Lao Tzu in his book Tao Te Ching and in the Book of Changes or I-Ching) are not merely opposites. The more we learn about human striving, the more we see they are supplementary, complementary, integrative and inextricably bound together. Duality is very elegant, yet powerful and has a long and distinguished history going back thousands of years. It is a natural and precious phenomenon that permeates or occurs in practically all branches of mathematics, physics, engineering, logic, psychology, real life, food science, social sciences, natural sciences, medical sciences such as alternative or energy medicine, acupuncture, meditation, qigong, physical therapy, nutrition therapy, immunotherapy, etc. \smallskip

Fundamentally, duality gives {\it two} different points of view of looking at
the  {\it same} object. In the study of comparison theorems (in {\it differential equations, differential geometry, differential - integral inequalities}), harmonic forms, mean curvature, etc, we find many objects that have two different points
of view and in principle they are all dualities. Whereas, in {\it physics} electricity and  magnetism are dual objects, {\it Hodge theory of harmonic forms} 
is motivated in part by {\it Maxwell's equations of unifying}
magnetism with electricity in a {\it physics} world. In a {\it complex line bundle} $\mathbb L_{\mathbb C}$ with structure group $U(1)$ over a $4$-dimensional Lorentzian manifold, Maxwell's equations are the precise equations for the curvature $2$-form $\omega$ of a connection on $\mathbb L_{\mathbb C}$ being both closed $(d \omega = 0)$ and co-closed $(d^{\star} \omega = 0)$, and hence the curvature form $\omega$ is harmonic $\big (\Delta \omega := - (d d^{\star} + d^{\star} d) \omega = 0\big )$. Whereas in {\it topology}, duality can broadly distinguish between {\it contravariant} functors such as comology, $K$-theory, or more general bundle theory and {\it covariant} functor such as homology or homotopy. A de Rham
cohomology class involves with interplaying two objects - closed forms (coming from the {\it source} of one exterior differential operator $d$) and exact forms (coming from the {\it target} of its preceding operator $d$). Whereas in {\it operator theory}, Hodge Laplacian involves with utilizing dual objects - exterior differential operator $d$ and codifferential operator $d^{\star}$, harmonic forms are privileged representatives in a de Rham
cohomology class picked out by the Hodge Laplacian. From the view point of {\it calculus of variations}, harmonic forms are precisely the minimal
$L^2$ forms within their respective topological (cohomology)
classes. Whereas, many things embrace dualities {\it Yin} and {\it Yang} according to Lao Tzu, their dynamical interplay is viewed as the essence of all {\it natural} phenomena, {\it human} affairs,  and Chinese {\it medicine}, and has strong impacts, interactions and connections with unity.
Harmonic forms, motivated in part by Maxwell's equations, coming out of physics, generalize
harmonic functions in the study of {\it complex analysis, partial differential equations, 
potential theory}, and have
important connections with {\it several complex variables, Lie group representation theory and algebraic geometry} through the use of
K\"ahler metric and Bochner technique.
\smallskip

In the {\bf first} part of this paper, 
we begin with studying comparison theorems in differential equations and in differential geometry and the transitions between these two fields from the viewpoint of dualities.

We observe and describe the duality between two types of differential equations, the second-order, linear Jacobi type equation \eqref{2.1} and the first-order, nonlinear Riccati type equation \eqref{2.2}, and the duality between the initial condition in \eqref{2.1} and the asymptotic condition in \eqref{2.2} (see Theorem 2.1).  Moreover, to each type of differential equation, there corresponds a comparison theorem on its supersolutions and subsolutions
with appropriate initial or asymptotic condition (see Theorems E and Theorem \ref{T: 2.2}). The duality between these two types of differential equations \eqref{2.1} and \eqref{2.2} leads to the duality between two corresponding comparison theorems in differential equations of the same type. This, in turn gives rise to new duality in {\it swapping} differential equations of dual type and hence generates the following Comparison Theorems in Differential Equations of Mixed Types I and II with appropriate initial 
and asymptotic conditions: 

Denote $AC(0,t)$ the set of absolutely continuous real-valued function on an open interval $(0, t)$, i.e.,
\[ AC(0,t) = \{ f : (0, t) \to \mathbb R\, |\,  f \operatorname{is}\,  \operatorname{absolutely}\,  \operatorname{continuous}\, , \operatorname{where}\, (0, t) \subset (0, \infty)\}.   
\]

\begin{theorem2.3}$($Comparing Differential Equations of Mixed Type $\operatorname{I})$
Let $G_i$ be real-valued functions defined on $(0, t_i) \subset (0, \infty)\, , i =1,2$ satisfying \begin{equation}G_2 \le G_1\tag{2.15}
\end{equation}on $(0, t_1) \cap (0, t_2) $. Let $g_1\in AC(0, t_1)\, $  
be a solution of 
\begin{equation}\tag{2.24}
\qquad \, \, \begin{cases}
{g_1}^{\, \prime} + \frac {{g_1}^2}{\kappa_1 } + \kappa_1 G_1 \le\, 0\quad  \operatorname{a.}\, \operatorname{e.}\quad  \operatorname{in}\quad (0, t_1) \\
g_1(t) = \frac {\kappa_1 }{t} + O(1)\quad \operatorname{as}\quad t \to 0^+\, 
\end{cases}\end{equation}
and $f_2 \in C([0, t_2]) \cap C^1(0,t_2)$ with 
$f_2^{\prime}\in AC(0,t_2)$ be a positive solution of 
\begin{equation}
\begin{cases}
f_2^{\, \prime\prime} + G_2 f_2 \,   \ge\,  0\quad  \operatorname{a.}\, \operatorname{e.}\quad  \operatorname{in}\quad (0, t_2)   \\
f_2(0) = 0\, , f_2^{\, \prime}(0) = \kappa_2. 
\end{cases}\tag{2.14}
\end{equation}
with constants $\kappa _i$ satisfying
\begin{equation}  0 < \kappa_1 \le  \kappa_2\, .\tag{2.16}
\end{equation}
Then $\quad t_1 \le t_2$ and \begin{equation}g_1 \le \frac {\kappa_2 f_2^{\, \prime}}{f_2}\, \tag{2.30}
\end{equation} 
on $(0, t_1)\, .$
\end{theorem2.3}

\begin{theorem2.4}$($Comparing Differential Equations of Mixed Type $\operatorname{II})\, $  
Let functions $G_i : (0, t_i) \subset (0, \infty) \to \mathbb R$, $i = 1, 2$ satisfy \eqref{2.15}. Let $f_1 \in C([0, t_1]) \cap C^1(0,t_1)$ with 
$f_1^{\prime}\in AC(0,t_1)$ be a positive solution of 
\begin{equation}
\begin{cases}
f_1^{\, \prime\prime} + G_1 f_1 \,   \le\,  0\quad  \operatorname{a.}\, \operatorname{e.}\quad  \operatorname{in}\quad (0, t_1)   \\
f_1(0) = 0\, , f_1^{\, \prime}(0) = \kappa_1 
\end{cases}\tag{2.13}
\end{equation}
and $g_2\in AC(0, t_2)\, $  
be a solution of 
\begin{equation}\tag{2.25}
\qquad \, \quad \begin{cases}
{g_2}^{\, \prime} + \frac {{g_2}^2}{\kappa_2 } + \kappa_2 G_2 \ge\, 0\quad  \operatorname{a.}\, \operatorname{e.}\quad  \operatorname{in}\quad (0, t_2). \\
g_2(t) = \frac {\kappa_2 }{t} + O(1)\quad \operatorname{as}\quad t \to 0^+\, ,
\end{cases}\end{equation}
Assume \eqref{2.16}. Then  $t_1 \le t_2$ and \begin{equation}\frac {\kappa_1 f_1^{\, \prime}}{f_1} \le g_2\, \tag{2.32}
\end{equation} on $(0, t_1)\, .$
\end{theorem2.4}

When the radial Ricci curvature of a manifold is bounded below by a function $(n-1)\, G_1$ $\big (\operatorname{cf}.\, \eqref{2.34} \big )$, we show via Weitzenb\"ock formula (Theorem F), this is a {\it disguised supersolution} of a Riccati type equation \eqref{2.38}. Similarly, when the radial curvature of a manifold is bounded above by a function $\tilde {G_2}$$\big (\operatorname{cf}.\, \eqref{2.60} \big )$ $(\operatorname{resp.}\, \operatorname{bounded}\, \operatorname{below}\, \operatorname{by}\, \operatorname{a}\, \operatorname{function}\, G_1$$\big (\operatorname{cf}.\, \eqref{2.59} \big )$, we show this is a {\it disguised subsolution} $\big (\operatorname{resp.}\, \operatorname{supersolution}\, \big )$ 
of a Riccati type equation $\eqref{2.28}$ $\big (\operatorname{resp.}\, \eqref{2.27} \big )$ in which $G_2 = \tilde {G_2}$ and $t_2 = \tilde {t_2}$ derived from \eqref{2.66} 
$\big (\operatorname{resp.}\,  \eqref{2.65} \big )$. Thus, we are ready to utilize dualities in comparison theorems in differential equations of mixed types to
 generate comparison theorems in Riemannian geometry and provide simple and direct proofs.
In particular, we obtain Laplacian Comparison Theorem \ref{T: 2.5} when the radial Ricci curvature is bounded below as in \eqref{2.34}, and Hessian and Laplacian Comparison Theorem \ref{T: 2.8} when the radial curvature is bounded above as in \eqref{2.60} (resp. bounded below as in \eqref{2.59}).
Throughout this paper we fix a point $x_0$ in an $n$-dimensional manifold $M\, .$  Let $r$ be the distance function on $M\, $ relative to $x_0\, ,$ $D(x_0) = M \backslash (\operatorname{Cut}(x_0) \cup \{ x_0 \})\, ,$ 
  $B_{t}(x_0) = \{ x \in M: r(x) <  t \}\, ,$ and a punctured geodesic ball $\overset {\circ}B_{t}(x_0) = B_{t}(x_0) \backslash \{x_0\}.$ 

\begin{theorem2.5}$($Laplacian Comparison Theorem$)$ Let functions $G_i : (0, t_i) \subset (0, \infty) \to \mathbb R$, $i = 1, 2$ satisfy \eqref{2.15} on $(0, t_1) \cap (0, t_2)$. 
Assume 

\noindent
$\eqref{2.34}\hskip0.8in \qquad (n-1)\, G_1(r)\le \, \operatorname{Ric}^{\operatorname{M}}_{\operatorname{rad}}(r)$

\noindent
$\operatorname{on}\quad \overset {\circ} B_{t_1}(x_0)\subset D(x_0),$ and let 
$f_2 \in C([0, t_2]) \cap C^1(0,t_2)$ with 
$f_2^{\prime}\in AC(0,t_2)$
be a positive solution of 
\begin{equation}
\begin{cases}
f_2^{\prime\prime} + G_2 f_2 \ge 0\quad  \operatorname{a.}\, \operatorname{e.}\quad \operatorname{in}\quad (0, t_2) \\
f_2(0) = 0\, , f_2^{\prime}(0) = n-1\, .
 \end{cases}\tag{2.35}
\end{equation}
Then $\quad t_1 \le t_2\quad \, $ and
\begin{equation}
\Delta r\le (n-1)\frac{f_2^{\prime}}{f_2}(r)\tag{2.36}
\end{equation}
holds in $\overset {\circ} B_{t_1}(x_0)\, .$ If in addition, 
\eqref{2.34}  occurs in $D(x_0)\, ,$
then  
\eqref{2.36} holds pointwise on $D(x_0)\, $ and weakly on $M.$
\end{theorem2.5}

\begin{theorem2.8}$($Hessian and Laplacian Comparison Theorems$)$
Let 

\noindent
$\, \eqref{2.59} \hskip1.0in G_1(r) \le K(r)$ 

\noindent
on $\overset {\circ} B_{t_1}(x_0)\subset D(x_0) $ $\big ( \operatorname{resp.}\, $

\noindent
$\eqref{2.60}  \hskip1.0in K(r) \le \widetilde{G_2}(r)$ 

\noindent
on $\overset {\circ} B_{\tilde{t_2}}(x_0)\subset D(x_0)\, \big ),$
and let 
$f_2 \in C([0, t_2]) \cap C^1(0,t_2)$ with 
$f_2^{\prime} \in AC(0,t_2)$ be a 
positive solution of 
\begin{equation}
\quad \quad \begin{cases}
f_2^{\, \prime\prime} + G_2 f_2 \,   \ge\,  0\quad  \operatorname{a.}\, \operatorname{e.}\quad  \operatorname{in}\quad (0, t_2)   \\
f_2(0) = 0\, , f_2^{\, \prime}(0) = \kappa_2,
\end{cases}\tag{2.14}
\end{equation}
where $G_i : (0, t_i) \to \mathbb R\,  $ satisfy 

\noindent
$\eqref{2.15}\hskip1.0in G_2 \le G_1$ 

\noindent 
on $(0, t_1) \cap (0, t_2)$ and $1 \le \kappa _2\, .$

\noindent
$\big ( \operatorname{resp.}\,  f_1 \in C([0, \tilde{t_1}]) \cap C^1(0,\tilde{t_1})$ with
$f_1^{\prime} \in AC(0,\tilde{t_1}) $ be a positive solution of 
\begin{equation}
\begin{cases}
f_1^{\prime\prime} + \widetilde{G_1}f_1 \le 0\quad \operatorname{on}\quad (0, \tilde{t_1}) \\
f_1(0) = 0\, , f_1^{\prime}(0) = \kappa_1\, ,
 \end{cases} \tag{2.61}\end{equation}
where $\widetilde{G_i} : (0, \tilde{t_i}) \to \mathbb R\, $ satisfy 

\noindent
$\eqref{2.62} \hskip1.0in \widetilde{G_2} \le \widetilde{G_1}$

\noindent
on $(0, \tilde{t_1}) \cap (0, \tilde{t_2})$  
and $0 < \kappa_1 \le 1
\big )\, .$

\noindent
Then $\quad t_1 \le t_2\, ,$

\noindent
$\eqref{2.63}\hskip1.0in \operatorname{Hess} r \le \frac{\kappa_2 f_2^{\prime}}{f_2}(g-dr\otimes dr)\quad \operatorname{and}\quad  \Delta r \le (n-1)\frac{\kappa_2 f_2^{\prime}}{f_2}(r)$

\noindent
on $\overset {\circ} B_{t_1}(x_0)$
$\big ( \operatorname{resp.}\quad  \tilde{t_1} \le \tilde{t_2}\, ,$ 

\noindent
$\eqref{2.64}\hskip0.55in \frac{\kappa_1 f_1^{\prime}}{f_1}\big( g-dr\otimes dr \big) \le \operatorname{Hess} r \quad \operatorname{and}\quad  (n-1)\frac{\kappa_1 f_1^{\prime}}{f_1}(r)\le \Delta r $

\noindent
on $\overset {\circ} B_{\tilde{t_1}}(x_0)\big )$, in the sense of quadratic forms. If in addition \eqref{2.59} occurs on $D(x_0)$, then the second part of \eqref{2.63}  holds pointwise on $D(x_0)$ and weakly in $M\, .$
\end{theorem2.8}

This strengthens main theorems in  \cite [Theorem 4.1] {HLRW}, and \cite [Theorem D] {W3} by weakening the hypotheses on the domain and regularity of 
$G_i \, $ $(\operatorname{resp}.   \tilde{G_i}\, )$ from a smooth function  on $\mathbb {R}^{+} \cup \{0\}\, $ to an arbitrary real-valued function  on  $(0, t_i) \big ( \operatorname{resp.} (0, \tilde{t_i}) \big )$, (not necessarily continuous) and the hypotheses on the domain of
$K(r)$ from $D(x_0)$ to $\overset {\circ} B_{t_1}(x_0) \big (\operatorname{resp}.\, \overset {\circ} B_{\tilde{t_2}}(x_0)\big )\, .$ We also give direct and simple proofs with new applications.

In applying Theorem \ref{T: 2.5} under the radial Ricci curvature assumption,  \eqref{2.39} $\quad G_1=-(n-1)\frac {A(A-1)}{r^2}\leq \text{Ric}^{\rm{M}}_{\rm{rad}}\, , A \ge 1$, or equivalently \eqref{2.41}, we need \eqref{2.35} with appropriate $G_2$ for comparison. But in Theorem \ref{T: 2.6} we do not need to {\it assume} \eqref{2.35} as in Theorem \ref{T: 2.5}. Instead, utilizing $G_1$, we {\it find} \eqref{2.44} as a companion system for comparison by duality, and estimate $\frac {f_2^{\prime}}{f_2}$ (see Theorem \ref{T: 2.6} and its proof). Similarly, without assuming \eqref{2.35}, we have

\begin{theorem2.7}
$(1)$ If  

\noindent
\eqref{2.51} $\hskip1.0in (n-1)\frac {B_1(1-B_1)}{r^2}\leq \operatorname{Ric}^{\rm{M}}_{\rm{rad}}(r),\quad \operatorname{where} \quad  0 \le B_1 \le 1\,$ 

\noindent
on $\overset {\circ} B_{t_1}(x_0)\subset D(x_0)$, then

\noindent
\eqref{2.52}$\hskip1.0in \Delta r  \leq (n-1)\frac{1+\sqrt{1+4B_1(1-B_1)}}{2r}$ 
 
 \noindent
in $\overset {\circ} B_{t_1}(x_0)$. If in addition, \eqref{2.51} occurs on $D(x_0)\, ,$  then \eqref{2.52} holds pointwise on $D(x_0)$ and weakly on $M\, .$
%\begin{equation}
%\Delta r \leq (n-1)\frac{1+\sqrt{1+4B_1(1-B_1)}}{2r}\quad \operatorname{pointwise}\, \operatorname{on}\quad  D(x_0)\quad \operatorname{and}\quad \operatorname{weakly}\, \operatorname{on}\, M.\label{2.45}
%\end{equation}

\noindent
$(2)$ Equivalently, if 

\noindent 
$\eqref{2.53}\hskip1.0in   (n-1) \frac {B_1(1-B_1)}{(c+r)^2}\leq \operatorname{Ric}^{\rm{M}}_{\rm{rad}}(r) \, ,  \quad 0 \le B_1 \le 1\,$ 

\noindent
on $\overset {\circ} B_{t_1}(x_0) \subset D(x_0)\, $, $\operatorname{where}\, \, c \ge 0\, ,$ then \eqref{2.52} holds in $\overset {\circ} B_{t_1}(x_0)$. If in addition \eqref{2.53} occurs on $D(x_0)$, then \eqref{2.52} holds pointwise on $D(x_0)$ and weakly on $M\, .$
\end{theorem2.7}

As applications of Theorem \ref{T: 2.8}, we have Theorems \ref{T: 3.1} and \ref{T: 3.2} under the negative lower bound of the radial curvature assumption \eqref{3.1} $\, -\frac {A(A-1)}{r^2}\leq K(r)\, , A \ge 1\, \big ($or equivalently \eqref{3.4}$\big )$, and
under the negative upper bound of the radial curvature assumption \eqref{3.5} $\, K(r) \le -\frac {A_1(A_1-1)}{r^2}\, , A_1 \ge 1$ $\big ($or equivalently \eqref{3.7}$\big )$ respectively.

\begin{corollary3.1} $(1)$
If the
radial curvature $K$   satisfies  
\begin{equation}
- \frac {A(A-1)}{r^2}\leq K(r)\le - \frac {A_1(A_1-1)}{r^2}\quad \operatorname{on} \quad M\backslash \{x_0\}\quad \operatorname{where} \quad   A \ge  A_1 \ge 1\, ,\tag{3.14}\end{equation}

\noindent
then we have 

\noindent
 $\eqref{3.15} \quad \frac{A_1}{r}\bigg(
g-dr\otimes dr\bigg) \le \operatorname{Hess} r  \leq \frac{A}{r}\bigg(
g-dr\otimes dr\bigg)\, \operatorname{in}\, \operatorname{the}\, \operatorname{sense}\, \operatorname{of}\, \operatorname{quadratic}\, \operatorname{forms},$

\noindent
$\qquad \qquad \qquad (n-1)\frac{A_1}{r} \le  \Delta r  \leq (n-1)\frac{A}{r}\quad \operatorname{pointwise}\quad \operatorname{on}\, M\backslash \{x_0\} \quad \operatorname{and}$

\noindent
$\qquad \qquad \qquad  \Delta r \leq (n-1)\frac{A}{r}\quad \operatorname{weakly}\, \operatorname{on}\, M.$

\noindent
$(2)$ Equivalently, if $K$ satisfies  

\noindent
$\eqref{3.16} \hskip1.0in - \frac {A(A-1)}{(c+r)^2}\leq K(r) \le - \frac {A_1(A_1-1)}{(c+r)^2}\, , \quad  A \ge A_1 \ge 1\, ,$

\noindent
$\operatorname{on} \, M \backslash \{x_0\}\, \operatorname{where} \, c \ge 0\, ,$
then \eqref{3.15} holds. 
\end{corollary3.1}

Corollary \ref{C: 3.1} is equivalent to the following Theorem A in which $(1)$ is due to Han-Li-Ren-Wei (\cite {HLRW}), and extends the work of Greene and Wu (\cite [p.38-39] {GW}) from asymptotical estimates $\big ($off a geodesic ball$\,  B_{(a-1)}(x_0)\big )\, $ near infinity to pointwise estimates on $ M\backslash \{x_0\}\, $ 

\begin{theoremA}$\, (1$$)$$($\cite {HLRW}$)\, $  Let the
radial curvature $K$  satisfy
\begin{equation}
-\frac {A}{r^2}\leq K(r)\leq -\frac {A_1}{r^2}\quad \operatorname{where}\quad 0 \le A_1 \le A\, \tag{3.17}\end{equation}
$\operatorname{on}\, M\backslash \{x_0\}.$ Then 
\begin{equation}
\frac{1+\sqrt{1+4A_1}}{2r}\bigg(g-dr\otimes dr\bigg) \le  \operatorname{Hess} (r) \le \frac{1+\sqrt{1+4A}}{2r}\bigg(
g-dr\otimes dr\bigg)\quad \operatorname{on}\quad M\backslash \{x_0\}\, ,\tag{3.18}\end{equation}
\begin{equation}
(n-1)\frac{1+\sqrt{1+4A_1}}{2r} \le  \Delta r  \le (n-1)\frac{1+\sqrt{1+4A}}{2r}\, \operatorname{pointwise}\, \operatorname{on}\, M\backslash \{x_0\}\,  ,\operatorname{and}\,\tag{3.19}\end{equation}

\begin{equation}\Delta r \le (n-1)\frac{1+\sqrt{1+4A}}{2r} \quad \operatorname{weakly}\, \operatorname{on}\, M.\tag{3.20}\end{equation}

\noindent
$(2)$ Equivalently, if 
\begin{equation}
-\frac {A}{(c+r)^2}\leq K(r)\leq -\frac {A_1}{(c+r)^2}\quad \operatorname{where}\quad 0 \le A_1 \le A\, \tag{3.21}\end{equation}
$\operatorname{on}\, M\backslash \{x_0\},$ then \eqref{3.18}-\eqref{3.20} hold.

\end{theoremA}

As further applications of Theorem \ref{T: 2.8}, we also obtain
the following Theorems \ref{T: 3.3} and \ref{T: 3.4} under positive lower and upper bound on radial curvature assumptions respectively.

\begin{theorem3.3}
$(1)$ If 

\noindent
$\eqref{3.25}\hskip1.0in \frac {B_1(1-B_1)}{r^2}\leq K(r)\quad \operatorname{where} \quad  0 \le B_1 \le 1\, $

\noindent
on $\overset {\circ} B_{t_1}(x_0) \subset D(x_0)$, then
$\operatorname{Hess} r $ and $\Delta r$ satisfy 
 
 \noindent
 $\eqref{3.26} \hskip1.0in \operatorname{Hess} r  \le \frac{1+\sqrt{1+4B_1(1-B_1)}}{2r}\bigg(
g-dr\otimes dr\bigg)\quad \operatorname{and}$

\noindent
$\eqref{3.27}\hskip1.0in \Delta r  \le (n-1) \frac{1+\sqrt{1+4B_1(1-B_1)}}{2r}$

\noindent
on $\overset {\circ} B_{t_1}(x_0)$
respectively. If in addition \eqref{3.25} occurs on $D(x_0)\, ,$ then \eqref{3.26} holds pointwise on $D(x_0)\, $ and \eqref{3.27} holds weakly on $M$.

\noindent
$(2)$ Equivalently, if 

\noindent
$\eqref{3.28}\hskip1.0in \frac {B_1(1-B_1)}{(c+r)^2}\leq K(r)\, , \quad  0 \le B_1 \le 1\,$

\noindent
on $\overset {\circ} B_{t_1}(x_0) \subset D(x_0)$, where, $c \ge  0\, ,$ then \eqref{3.26} and \eqref{3.27} hold on $\overset {\circ} B_{t_1}(x_0)$. If in addition \eqref{3.28} occurs on $D(x_0)$, then \eqref{3.26} holds pointwise on $D(x_0)\, $ and 
\eqref{3.27} holds weakly on $M$.\end{theorem3.3}

\begin{theorem3.4}
$(1)$ If  

\noindent
$\eqref{3.35}\hskip1.0in K(r) \le \frac {B(1-B)}{r^2}\quad \operatorname{where}\quad  0 \le B \le 1 $

\noindent
on $\overset {\circ} B_{t_2}(x_0) \subset D(x_0)$, then $\operatorname{Hess} r $ and $\Delta r$ satisfy

\noindent
$\eqref{3.36}\hskip1.0in \frac{|B - \frac 12| + \frac 12}{r} \bigg(g-dr\otimes dr\bigg) \le \operatorname{Hess} r \quad \operatorname{and}$

\noindent
$\hskip1.4in (n-1)\frac{|B - \frac 12| + \frac 12}{r} \le  \, \, \Delta r $ 

\noindent
on $\overset {\circ} B_{t_2}(x_0)$ respectively. 

\noindent
$(2)$ Equivalently, $\operatorname{if}$ 

\noindent
$\eqref{3.37}\hskip1.0in K(r) \le \frac {B(1-B)}{(c+r)^2}\, ,\quad   0 \le B \le 1$

\noindent
on $\overset {\circ} B_{t_2}(x_0) \subset D(x_0)\, , \operatorname{where} \, c \ge 0\, ,$ then \eqref{3.36} holds on $\overset {\circ} B_{t_2}(x_0)$.
\label{T: 3.4}\end{theorem3.4}

\begin{theorem3.5}  $(1)$
If the
radial curvature $K$ satisfies 
\begin{equation} 
\frac {B_1}{r^2} \le K(r) \le \frac {B}{r^2} \quad \operatorname{where} \quad  0 \le B_1 \le B \le \frac 14\, \tag{3.45}\end{equation}
on $\overset {\circ} B_{\tau}(x_0) \subset D(x_0)\, ,$ then

\noindent
$\eqref{3.46}\hskip0.2in \frac{1+\sqrt{1-4B}}{2r} \bigg(
g-dr\otimes dr\bigg) \le  \operatorname{Hess} r  \le \frac{1+\sqrt{1+4B_1}}{2r}\bigg(g-dr\otimes dr\bigg)\quad \operatorname{and}$

\noindent
$\hskip0.6in (n-1)  \frac{1+\sqrt{1-4B}}{2r} \le  \,  \, \Delta r \le (n-1)\frac{1+\sqrt{1+4B_1}}{2r}
$

\noindent
hold $\operatorname{pointwise}\, \operatorname{on}\, \overset {\circ} B_{\tau}(x_0)$. If in addition \eqref{3.33} occurs on $D(x_0)\, ,$ then
\begin{equation}
\Delta r \leq  (n-1) \frac{1+\sqrt{1+4B_1}}{2r}\tag{3.47}\end{equation}
holds weakly $\operatorname{on}\, M\, .$

\noindent
$(2)$ Equivalently, if $K$ satisfies  
\begin{equation}
\frac {B_1}{(c+r)^2}\leq K(r) \le \frac {B}{(c+r)^2}\, ,\quad 0 \le B_1 \le B \le \frac 14\,
\tag{3.48}\end{equation}
\noindent
$\operatorname{on} \,  \overset {\circ} B_{\tau}(x_0) \subset D(x_0)\, , \operatorname{where} \,  c \ge 0\, ,$ 
then \eqref{3.46} holds on $\overset {\circ} B_{\tau}(x_0)$. If in addition, \eqref{3.33} occurs on $D(x_0)\, ,$ then \eqref{3.47} holds weakly $\operatorname{on}\, M\,.$
 \end{theorem3.5}
 
 The case $K(r) \le \frac {B}{r^2},\, B \le \frac 14$ is due to Han-Li-Ren-Wei (\cite{HLRW}). Theorem \ref{T: 3.5} is equivalent to the following  

\begin{corollary3.5} $(1)$
Let the
radial curvature $K$
  satisfy  
\begin{equation} \frac {B_1(1-B_1)}{r^2}\leq K(r)\le \frac {B(1-B)}{r^2}\, \quad  \operatorname{where}\quad 0 \le B \, , B_1 \le 1\, \tag{3.49}\end{equation} 
$\operatorname{on} \, \overset {\circ} B_{\tau}(x_0) \subset D(x_0)\, .$ Then 

\noindent
$\eqref{3.50}\hskip0.1in \frac {|B - \frac 12| + \frac 12}{r} \bigg(g-dr\otimes dr\bigg)    \le \operatorname{Hess} r  \le \frac{1+\sqrt{1+4B_1(1-B_1)}}{2r}\bigg(
g-dr\otimes dr\bigg)\quad \operatorname{and}$

\noindent
$\hskip0.5in (n-1)\frac{|B - \frac 12| + \frac 12}{r} \le  \Delta r  \le (n-1) \frac{1+\sqrt{1+4B_1(1-B_1)}}{2r}$

\noindent
hold $\operatorname{pointwise}\, \operatorname{on}\, \overset {\circ} B_{\tau}(x_0)$. If in addition \eqref{3.25} occurs on $D(x_0)\, ,$ then

\noindent
$\eqref{3.27}\hskip1.0in \Delta r \leq  (n-1) \frac{1+\sqrt{1+4B_1(1-B_1)}}{2r}$

\noindent 
holds weakly $\operatorname{on}\, M\, .$

\noindent
$(2)$ Equivalently, if $K$ satisfies  

\noindent
$\eqref{3.51}\hskip1.0in  \frac {B_1(1-B_1)}{(c+r)^2}\leq K(r) \le  \frac {B(1-B)}{(c+r)^2}\, ,\quad 0 \le B \, , B_1  \le 1\,$

\noindent
$\operatorname{on} \,  \overset {\circ} B_{\tau}(x_0) \subset D(x_0)\, , \operatorname{where} \,  c \ge 0\, ,$ 
then \eqref{3.50} holds on $\overset {\circ} B_{\tau}(x_0)$. If in addition, \eqref{3.28} occurs on $D(x_0)\, ,$ then \eqref{3.27} holds weakly $\operatorname{on}\, M\,.$
\end{corollary3.5}

 \begin{corollary3.7} 
If the
radial curvature $K$   satisfies 
 \begin{equation} \begin{aligned}-\frac {A}{r^2}\leq & K(r)\leq \frac {B}{r^2}\\
  (\operatorname{resp.}\, -\frac {A}{(c+r)^2}\leq & K(r)\leq \frac {B}{(c+r)^2})\, , \quad 0 \le A\, , \, 0 \le B \le \frac 14 \end{aligned}\tag{3.54}\end{equation}
$\operatorname{on}\, M\backslash \{x_0\},$ where $c > 0\, ,$ then  
\begin{equation}
\frac{1+\sqrt{1-4B}}{2r}\bigg(g-dr\otimes dr\bigg) \le  \operatorname{Hess} (r) \le \frac{1+\sqrt{1+4A}}{2r}\bigg(
g-dr\otimes dr\bigg)\quad \operatorname{and}\tag{3.55}\end{equation}
\begin{equation}
(n-1)\frac{1+\sqrt{1-4B}}{2r} \le  \Delta r  \le (n-1)\frac{1+\sqrt{1+4A}}{2r}\, \tag{3.56}\end{equation} 

\noindent
hold $\operatorname{on}\, M\backslash \{x_0\}\,  .$
\end{corollary3.7}

As immediate applications of Comparison Theorems in Differential Geometry (Theorems \ref{T: 3.1}-\ref{T: 3.5}), we obtain {\it Mean Curvature Comparison Theorems} (see Theorems \ref{T: 4.1}).  

We then utilize dualities to study harmonic forms and their decompostion, integrability and growth. It is
well-known that on a compact Riemannian manifold, a smooth
differential form $\omega$ is harmonic if and only if it is closed and co-closed. That is,  
\begin{equation}\tag{6.1} \Delta \omega
= 0 \qquad {\rm if}\quad {\rm and}\quad {\rm only}\quad {\rm
if}\qquad d\, \omega = 0\quad {\rm and}\quad d^{\star}\omega = 0.
\end{equation}

On {\it complete noncompact} Riemannian manifolds, although \eqref{6.1} holds for smooth $\omega$ with compact support, it does not hold in general.
Simple examples include, in $\mathbb{R}^n\, ,$ a closed, non-co-closed, harmonic form $\omega_1 = x_1 d x_1\, ,$
a non-closed, co-closed, harmonic form $\omega_2 = x_n d x_1\, ,$ and a non-closed, non-co-closed, harmonic form $\omega_3 = (x_1 + x_n)d x_1\, ,$ or $\omega_4 = x_1 x_n d x_1 + x_n d x_n\, $ (see \cite {W4}). However, it is proved in \cite {AV}.
\begin{Theorem B}$($A. Andreotti and E. Vesentini \cite {AV}$)$ On a complete noncompact Riemannian manifold $M$, \eqref{6.1} holds for every smooth $L^2$ differential form $\omega\, .$  \end{Theorem B}
It is interesting to explore any possible generalizations of Theorem B, in particular, to discuss whether or not  \eqref{6.1} holds for $L^q$ differential form $\omega\, ,$ where $q\ne 2\, .$ Some study of this generalization can be found in \cite  [p.663, Proposition 1]{Y1}, its counter-examples are given by D. Alexandru-Rugina (see \cite [p. 81, Remarque 3] {AR}) on the hyperbolic space $H^m_{-1}, m\geq 3\, ,$ and relevant remarks of  Pigola-Rigoli-Setti are discussed in \cite [p.260, Remark B.8] {PRS}.
\smallskip

In \cite {W4}, we introduce and add Condition $\operatorname{W}$ $\big ($see \eqref{6.5}, where $\Omega = \omega$ and $ A^k(\xi) = A^k\big )$ to the above work, so that the counter-examples provided in \cite {AR} cannot happen, the conclusion of the Proposition in \cite {Y1} still holds, and works in a more general setting with geometric and physical applications.
We extend the differential form in $L^2$ space in Theorem B in several ways. To this end, recall in extending functions in $L^2$ space (resp. in $L^q$ space), we introduce and study the notion of
function growth: ``\emph{$p$-finite},
\emph{$p$-mild}, \emph{$p$-obtuse}, \emph{$p$-moderate}, and
\emph{$p$-small}" growth $\big ($ for $q=2$ (resp. for the same value of $q$)$\big )$, and their counter-parts
``\emph{$p$-infinite}, \emph{$p$-severe}, \emph{$p$-acute},
\emph{$p$-immoderate}, and \emph{$p$-large}" growth \cite {WLW} (see
Definition \ref {5.1}).
We also introduce the notion of {\it $p$-balanced} and {\it
$p$-imbalanced} growth for functions and differential forms on
complete noncompact Riemannian manifold $M$ in \cite {W1}. Namely,
a function or a differential form $f$ has \emph{$p$-balanced
growth} $(or, simply, \emph{is $p$-balanced})$ if $f$ has one of
the following: \emph{$p$-finite}, \emph{$p$-mild},
\emph{$p$-obtuse}, \emph{$p$-moderate}, or \emph{$p$-small}
growth on $M$, and has \emph{$p$-imbalanced growth} $( $ or, simply,
is $p$-\emph{imbalanced}$)$ otherwise (see
Definition \ref{5.2}).
Furthermore, extending our techniques for function growth (see \cite {CW1, CW2}) to differential form growth, we use direct simple new methods.
Let $A^k$ be the space of smooth differential $k$-forms on $M\, ,$ and $\omega$ be a smooth differential $k$-form, $k \ge 0$ on $M\, .$ 
Denote by $\langle \cdot , \cdot \rangle\, $
and $\star : A^k  \to A^{n-k}\, ,$ the {\it inner product} 
and the {\it linear operator} which assigns to each $k$-form on $M$ an $(n-k)$-form and which satisfies $\star\, \star = (-1)^{nk+k+1}$ respectively.  In particular, we have the following theorem, generalizing \cite {Y2} from functions to differential forms (See also \cite {WLW} for generalizations to various types of functions).

\begin{theoremC} $($Unity Theorem$)$ $($\cite [Theorem 4.1] {W4}$)$
If a differential form $\omega \in A^k$ has $2$-balanced growth, for $q=2\, ,$ or for $1 < q(\ne2) < 3$ with $\omega$ satisfying $\operatorname{Condition} \operatorname{W}$, then the following six statements:
$($i$)$ $\langle\omega, \Delta \omega\rangle \ge 0\, .$
$($ii$)$ $\Delta \omega = 0\, .$ 
$($iii$)$ $d\, \omega = d^{\star}\omega = 0\, .$ 
$($iv$)$ $\langle \star\, \omega, \Delta \star\, \omega\rangle \ge 0\, .$ 
$($v$)$ $\Delta \star\, \omega = 0\, .$ 
$($vi$)$ $d\star\, \omega = d^{\star} \star\, \omega = 0\, .$ 
 are equivalent.
\end{theoremC}

In the {\bf second} part of the paper, we extend the notion of function growth and differential form growth $(\operatorname{see}$ \cite{W1,W4}$)$ to bundle-valued differential form growth of various types, and prove their interrelationship. In particular, we show
\begin{theorem5.4} 
For a given $q \in \mathbb R\, ,$ a function, or differential form or bundle-valued differential form $f$ is
\[
\begin{aligned}
& p-moderate\, \eqref{5.4}\quad  \Leftrightarrow \quad p-small\, \eqref{5.6}\quad \Rightarrow \quad p-
mild\, \eqref{5.2}\quad \Rightarrow \quad p-obtuse\, \eqref{5.3} \\
& 
\operatorname{or}\quad \operatorname{equiavalently},\\
& p-acute \quad   \Rightarrow \quad p-severe \quad \Rightarrow \quad p-large \quad \Leftrightarrow \quad p-immoderate.
\end{aligned}
\]
Hence, for a given $q \in \mathbb R\, ,$ $f$ is 
\[
\begin{aligned}
\quad p-\operatorname{balanced}\quad
& \Rightarrow \quad \operatorname{either} \quad p-\operatorname{finite}\, \eqref{5.1}\quad \operatorname{or}\quad p-\operatorname{obtuse}\, \eqref{5.3}\\
\quad p-\operatorname{imbalanced}\quad
& \Rightarrow \quad \operatorname{both} \quad p-\operatorname{infinite}\quad \operatorname{and}\quad p-\operatorname{immoderate}.
\end{aligned}\]
If in addition, $\int_{B(x_0;r)}|f|^{q}dv$
 is convex in $r$, then the following four types of growth are all equivalent: $f$ is \emph{$p$-mild}, \emph{$p$-obtuse}, \emph{$p$-moderate}, and
\emph{$p$-small} $(\operatorname{resp.}$ 
\emph{$p$-severe}, \emph{$p$-acute}, \emph{$p$-immoderate}, and 
\emph{$p$-large}$)$, i.e.,  $f$ is

$\qquad \eqref{5.2}\quad  \Leftrightarrow \quad \eqref{5.3}\quad \Leftrightarrow \quad \eqref{5.4}\quad \Leftrightarrow \quad \eqref{5.6}\quad $ for the same value of $q \in \mathbb R\, .$
\end{theorem5.4}
In particular, we have
\begin{corollary5.1}
 Every $L^q$ function or differential form or bundle-valued differential form $f$ on $M$ has \emph{$p$-balanced} growth, $p \ge 0\, ,$ and in fact, has \emph{$p$-finite},
\emph{$p$-mild}, \emph{$p$-obtuse}, \emph{$p$-moderate}, and
\emph{$p$-small} growth, $p \ge 0\, ,$  for the same value of $q$
\end{corollary5.1}
We also introduce $\operatorname{Condition} \operatorname{W}$ for bundle-valued differential forms in $A^k(\xi)$ and extend the
unity theorem from generalized harmonic forms in $A^k$ to generalized harmonic forms in $A^k(\xi)$ with values in a vector bundle, where 
$\xi :E\rightarrow M$ be a smooth Riemannian vector bundle over $(M,g)\, ,$ i.e. a vector bundle such that at each fiber is equipped with a positive inner product $\langle \quad , \quad \rangle_E\, .$
 Set $A^k(\xi )=\Gamma (\Lambda
^kT^{*}M\otimes E)$ the space of smooth $k$-forms on $M$ with
values in the vector bundle $\xi :E\rightarrow M$. 
Let $d^\nabla : A^k(\xi )\rightarrow
A^{k+1}(\xi )$ relative to the connection $\nabla ^E$ be the {\it exterior
differential operator}, ${\delta}^\nabla :A^{k+1}(\xi )\rightarrow
A^{k}(\xi )$ relative to the connection $\nabla ^E$ be the {\it codifferential operator}, and $\Delta ^\nabla = - (d^\nabla {\delta}^\nabla + {\delta}^\nabla d^\nabla) : A^k(\xi )\rightarrow
A^{k}(\xi) $.  Note that if $E$ is the trivial bundle $M \times \mathbb R$ equipped with the canonical metric, then $A^{k}$ is isometric to $A^{k}(\xi)\, ,  \, d = d^\nabla \big (\operatorname{as}\, \operatorname{in}\, \eqref{6.2} \big )\, ,$  $d^{\star} = \delta^\nabla\, \big (\operatorname{as}\, \operatorname{in}\, \eqref{6.3} \big )\, ,$ and $\Delta ^\nabla : A^k(\xi )\rightarrow
A^{k}(\xi) $ coincides with $\Delta = - (d d^{\star} + d^{\star} d) : A^k \rightarrow
A^{k}$. Denote by the same notations $\langle \cdot , \cdot \rangle\, ,$ $| \cdot | \, ,$ and $\star : A^k (\xi )\rightarrow  A^{n-k}(\xi )\, ,$ the {\it inner product}, the {\it norm} induced in fibers of various tensor bundles by the metric of $M\, ,$ and the {\it linear operator} which assigns to each $k$-form in $A^k(\xi)$ on $M$ an $(n-k)$-form in $A^{n-k}(\xi)$ and which satisfies $\star\, \star = (-1)^{nk+k+1}$ respectively. 

\begin{theorem6.2}$($Unity Theorem$)$
If a bundle-valued differential $k$-form $\Omega \in A^k(\xi )$ has $2$-balanced growth, for $q=2\, ,$ or for $1 < q(\ne2) < 3$ with $\Omega$ satisfying $\operatorname{Condition} \operatorname{W}\, \, \eqref{6.5}$, then the following six statements are equivalent.

\noindent
$($i$)$ $\langle\Omega, \Delta ^\nabla \Omega\rangle \ge 0$ $($ i.e. $\Omega$ is a generalized bundle-valued harmonic form$)$.\\
$($ii$)$ $\Delta ^\nabla \Omega = 0$ $($ i.e. $\Omega$ is a bundle-valued harmonic form$)$.\\
$($iii$)$ $d^\nabla\, \Omega = \delta^{^\nabla}\Omega = 0$ $($ i.e. $\Omega$ is closed and co-closed$)$.\\
$($iv$)$ $\langle \star\, \Omega, \Delta ^\nabla \star\, \Omega\rangle \ge 0$ $($ i.e. $\star\, \Omega$ is a generalized bundle-valued harmonic form$)$.\\
$($v$)$ $\Delta ^\nabla \star\, \Omega = 0\, $ $($ i.e. $\star\, \Omega$ is a bundle-valued harmonic form$)$.\\
$($vi$)$ $d^\nabla\,\star\, \Omega = \delta^{\nabla\,} \star\, \Omega = 0$ $($ i.e. $\star\, \Omega$ is closed and co-closed$)$.\label{T: 6.3}
\end{theorem6.2}

This generalizes the work of Andreotti and Vesentini \cite {AV} for the case $\Omega \in A^k\, $ and $\Omega$ is $L^2\, ,$ Theorem C $($\cite [Theorem 4.1] {W4}$)\, ,$ and 
leads to its immediate extension, 

\begin{corollary6.1} Let $\Omega \in A^{k}(\xi)$ be in $L^2$, or in $L^q, 1 < q(\ne 2) < 3$ satisfying $\operatorname{Condition} \operatorname{W}$$\big ($see \eqref{6.5}$\big ).$ Then
$\Omega$  is harmonic if and only if $\Omega$ is closed and co-closed. 
\end{corollary6.1}

In \cite {WL}, using an analog of Bochner's method or $``B^2 - 4AC \le 0"$ method, Wei and Li derive a geometric differential-integral inequality on a manifold (see Theorem \ref{T: 7.1}) and give the first application of the Hessian comparison theorems to generalized sharp Caffarelli-Kohn-Nirenberg type inequalities on
Riemannian manifolds(see Theorem \ref{T: 7.2}) . The case $M=\mathbb{R}^n$ is due to L. Caffarelli et al. \cite {CKN} and is sharp due to Costa \cite{C}. Furthermore, Wei and  Li \cite [Theorems 1 and 2, Corollaries 1.1-1.5] {WL} also gave the first application of the Hessian comparison theorems to generalized Hardy type inequalities on
Riemannian manifolds. The case $M=\mathbb{R}^n$ is sharp (see \cite {Mi, HLP}). 
\smallskip

In the {\bf third} part of the paper, using and extending the work in \cite {WL} and \cite {W3}, we apply Laplacian Comparison Theorems \ref{T: 2.6} and \ref{T: 2.7}, based on Theorem \ref{T: 2.5} to study geometric differential-integral inequalities on manifolds. In particular, we obtain 

\begin{theorem7.3}$($Generalized sharp Caffarelli-Kohn-Nirenberg type inequalities under radial Ricci curvature assumption$)$
 Let $M$ be an $n$-manifold with a pole such that the radial Ricci curvature $\text{Ric}^{\rm{M}}_{\rm{rad}} (r) $ of $M$ satisfies one of the
five conditions in \eqref{7.5}
where  $0 \le B_1 \le 1 \le A\, $ are constants. Then 
 for every $u\in W_{0}^{1,2}(M\backslash \left\{ x_{0}\right\} )$
and $a,b\in
\mathbb{R}\, ,$ \eqref{7.3} holds in which the constant $C$ is given by \eqref{7.6}.
\end{theorem7.3}

When $A=1\, ,$ the result is sharp and we recapture theorems of Wei and Li (see \cite [Theorems 3 and 4] {WL} or Theorem \ref{T: 7.2}).
Analogously, we apply Hessian Comparison Theorems \ref{T: 3.1}, \ref{T: 3.2}, \ref{T: 3.3}, and \ref{T: 3.4}, based on Theorem \ref{T: 2.8} to study geometric differential-integral inequalities. In particular we obtain
the following Theorem \ref{T: 7.4}. When $B_1 = 0\, ,$ we recapture theorems of Wei and Li (see \cite [Theorems 3 and 4] {WL} or Theorem \ref{T: 7.2}). The result is sharp when $M=\mathbb{R}^n\, .$   

\begin{theorem7.4}$($Generalized sharp Caffarelli-Kohn-Nirenberg type inequalities under radial curvature assumption$)$
 Let $M$ be an $n$-manifold with a pole and the radial curvature $K(r)$ of $M$ satisfy one of the
twelve conditions in \eqref{7.7} 
where  $0 \le B_1 \le 1 \le A\, $ are constants. Then 
 for every $u\in W_{0}^{1,2}(M\backslash \left\{ x_{0}\right\} )$
and $a,b\in
\mathbb{R}\, ,$ \eqref{7.3} holds in which the constant $C$ is given by \eqref{7.8}.
\end{theorem7.4}

Hessian comparison theorems lead to 
embedding theorems for weighted Sobolev spaces of functions on Riemannian manifolds (see Theorem \ref{T: 7.5}) and various geometric integral inequalities on manifolds $\big ($see \eqref{7.13}, \eqref{7.15}, \eqref{7.17}, \eqref{7.19}, \eqref{7.21}, \eqref{7.23}, \eqref{7.25} in Theorem \ref{T: 7.6}$\big )$. Concurrently, Laplacian comparison Theorems lead to the following theorem. \smallskip 

\begin{theorem7.7}$($Generalized sharp Hardy type inequality$)$
\label{T: 7.7} Let $M$ be an $n$-manifold with a pole satisfying 

\noindent
$\eqref{2.39}\hskip1.0in \operatorname{Ric}^{\rm{M}}_{\rm{rad}}(r) \ge - (n-1)\frac {A(A-1)}{r^2}\quad  \operatorname{where}\quad A \ge 1\, .$

\noindent
Then for every $u\in
W_{0}^{1,p}(M)$, $\frac {u}{r} \in L^p(M)$ with $p > (n-1) A + 1$, we have 

\noindent
$\eqref{7.27} \hskip1.0in 
%\begin{equation}  
\big (\frac{p-1-(n-1)A}{p}\big ) ^{p}\int_{M}\frac{|u|^p}{r^p}dv\leq \int_{M}\left\vert \nabla u\right\vert ^{p}dv\, . $
%\end{equation}
\end{theorem7.7}

The case $A=1$ is sharp and is due to Chen-Li-Wei $($see \cite[Theorem 5]{CLW2}$)$. Furthermore, the assumption $\frac {u}{r} \in L^p(M)$ cannot be dropped, or a counter-example is constructed in \cite [in Section 5] {CLW2}.
Moreover, this theorem
does not require $u\in W_0^{1,p}(M\backslash \{x_0\})$ and hence even for the case $p=2\, ,$ the result is
stronger than Hardy's inequality \eqref{7.25}, obtained by setting $a=1$ and $b=0$
in Theorems      
\ref{T: 7.3} and \ref{T: 7.4} $\big (\operatorname{see}\, \operatorname{Theorem}\, \ref{T: 7.6} (\text{vii}) \big ) $. More recently, through the study of comparison theorems in Finsler geometry, some generalized Hardy type inequalities and generalized Caffarelli-Kohn-Nirenberg type inequalities on
Riemannian manifolds have been extended to Finsler manifolds by Wei and Wu $($see \cite {WW}$)$. 
\smallskip

In the {\bf forth} part of the paper, following the framework in \cite {W3}, employing and augmenting the work of Dong and Wei \cite {DW}, Wei \cite {W3}, we apply comparison theorems to the study of differential forms of degree $k$ with values in vector bundles. In particular, we obtain a monotonicity formulae for vector bundle valued differential $k$-forms, if curvature satisfies one of the seven conditions in \eqref{8.12} (see Theorem \ref{T: 8.1}).

Whereas a ``microscopic" approach to monotonicity formulae leads to celebrated blow-up techniques 
due to de-Giorgi (\cite{Gi}) and Fleming (\cite{Fl}), 
and regularity theory in geometric measure theory (see \cite{FF,A,SU,PS,HL,Lu}). Examining a macroscopic viewpoint of the above monotonicity formula, we derive vanishing theorems for vector bundle valued $k$-forms (see Theorem \ref{T: 9.1}). This macroscopic viewpoint also leads to the study of global rigidity phenomena in locally conformal flat manifolds (see Dong-Lin-Wei \cite {DLW}).  

In \cite {DW}, Dong and Wei introduced $F$-Yang-Mills fields (When $F$ is the identity map, they are Yang-Mills fields) (see also \cite {G}). 
As an application of Theorem \ref{T: 9.1}, we obtain Vanishing Theorem for $F$-Yang-Mills fields (see Theorem \ref{T: 9.2}).

A natural link to unity (see \cite {W}) leads to Liouville theorems for $F$-harmonic maps (When $F(t) = \frac 1p (2t)^{\frac{p}{2}}\, , \, p > 1\, ,$ they become $p$-harmonic maps, or harmonic maps if $p=2\, .$) (see Theorem \ref{T: 10.1}).

In contrast to the work of Chang-Chen-Wei (\cite {CCW}) on Liouville properties for a $p$-harmonic morphism or a  $p$-harmonic function on a manifold that supports a weighted Poincar\'{e} inequality, we have Liouville theorems for $p$-harmonic maps   (see Theorem \ref{T: 10.2}).

The techniques can be further applied to generalized Yang-Mills-Born-Infeld
fields on manifolds. In particular we have vanishing theorems for generalized Yang-Mills-Born-Infeld
fields (with the plus sign) on Riemannian manifolds (see Theorem \ref{T: 11.2}).

In \cite {W2},
 we solve the Dirichlet problem for $p$-harmonic maps to which the solution
is due to Hamilton \cite {H1} in the case $p=2$ and $Riem^{N}\leq 0$:

\begin{theoremD} $($\cite {W2}$)$ Let $M$ be a compact Riemannian
$n$-manifold with boundary $\partial M$ and $N$ be a compact Riemannian manifold
with a contractible universal cover $\widetilde N$.
Assume that $N$ has no non-trivial $p$-minimizing tangent map
of $R^\ell$ for $\ell\leq n$.  Then any $u\in Lip(\partial M,N)\cap C^0(M,N)$
 of finite $p$-energy can be deformed to a
 $p$-harmonic map $u_0\in C^{1,\alpha}(M \backslash \partial M,N)\cap C^{\alpha}(M,N)$ minimizing $p$-energy
in the homotopic class with $u_0|_{\partial M}=u|_{\partial M}$
, where $1<p<\infty$.  In particular, every
$u \in C^1(M,N)$ can be deformed to a
$C^{1,\alpha}$ $p$-harmonic map $u_0$ in $M \backslash \partial M$ minimizing $p$-energy
in the homotopic class with H$\ddot o$lder continuous
$u_0|_{\partial M}=u|_{\partial M}$.
\end{theoremD}

Whereas the \emph {domain} of the solution $u$ of the above Dirichlet problem is an arbitrary compact Riemannian manifold (with boundary), we discuss a Dirichlet problem for which the \emph {target} of its solution is an arbitrary Riemannian manifold.
\smallskip

In the {\bf fifth} part of the paper, we apply the comparison theorems to the study of Dirichlet problems on starlike domains for vector bundle valued differential $1$-forms, augmenting the work of Dong and Wei $($\cite {DW}$)\, ,$ Wei \cite {W3}  and generalizing and refining the work of Karcher and Wood (\cite {KW}) on
harmonic maps on disc domains with constant boundary value (see Theorems \ref{T: 12.1} and \ref{T: 12.2}).
\smallskip

The author wishes to thank the referees for their comments and suggestions which helped the author prepared the final version of this paper. 
\section{Dualities in Comparison Theory}

We begin with a fundamental duality in comparison theory - the duality between the second-order linear Jacobi type equation with the initial conditions at $0$ and the first-order nonlinear Riccati type equation with the asymptotic condition at $0\, :$\smallskip

\noindent
\subsection{\bf Dualities between Differential Equations and Differential Inequalities}

\begin{theorem} Let $G (t)$ be a real-valued function defined on  $(0, \tau) \subset (0, \infty)\, ,$ $f \in C([0, \tau]) \cap C^1(0, \tau)\, ,$ $f > 0$ in $(0, \tau)\, ,$ $f^{\prime} \in AC(0, \tau)\, ,$ $\kappa > 0$ be a constant and $g \in AC(0, \tau)$.  Then 
\begin{equation}
{\begin{cases}
f^{ \, \prime \prime} (t)+ G (t)\, f (t) = 0\, (\operatorname{resp.}\, \le 0\, ,\, \ge 0 )\quad \operatorname{a.}\, \operatorname{e.}\quad \operatorname{in}\quad (0, \tau) \\
f (0) = 0\, , f^{\, \prime}(0) = \kappa
\end{cases}}\label{2.1}
\end{equation}

\noindent
is dual to $($or if and only if $)$

\begin{equation}{\begin{cases}
g^{\, \prime} (t) + \frac {g(t)^2}{\kappa} + \kappa G(t) = 0\, (\operatorname{resp.}\, \le 0\, ,\, \ge 0 )\quad \operatorname{a.}\, \operatorname{e.}\quad \operatorname{in}\quad (0, \tau)  \\
g(t)  = \frac{\kappa}{t} + O(1)
\quad \operatorname{as}\quad t \to 0^+\, .
\end{cases}}\label{2.2}
\end{equation}
Furthermore, \eqref{2.1} is transformed into \eqref{2.2}, 
via the transformer
\begin{equation} g = \frac {\kappa f^{\, \prime}}{f}
\label{2.3}
\end{equation}
in $(0, \tau)\, .$ Conversely, \eqref{2.1} is reversed from \eqref{2.2} and solution $($resp. supersolution, subsolution$)$ $f(t)$ of \eqref{2.1} enjoys 
\begin{equation}\frac{ f^{\, \prime}(t)}{f(t)} = \frac{g(t)}{\kappa}\, \label{2.4}
\end{equation}
in $(0, \tau)\, ,$ via the reverser
\begin{equation} f(t) = \kappa t \exp \bigg( \int_0^t \frac {g(s)}{\kappa}  - \frac 1s\, ds\bigg)\, 
\label{2.5}
\end{equation}
in $(0, \tau)$.   
\label{T: 2.1}
\end{theorem}

\begin{proof}

$(\Rightarrow)$ Since $f > 0$ in $(0, \tau)$, \eqref{2.3} is well-defined in $(0, \tau)\, $ with $g \in AC(0, \tau)$.\smallskip

It follows from \eqref{2.3}, the quotient law and $(2.1)$ that 
\begin{equation}\begin{aligned}
g^{\, \prime} & 
%= \bigg (\frac {f^{\, \prime}}{f} \bigg )^{ \prime}
 = \frac {\kappa f^{\, \prime\prime} f - \kappa ({f^{\, \prime}})^2}{f^{\, 2}} \\
& = \, (\operatorname{resp.}\, \le \, ,\, \ge  ) \frac {-\kappa Gf^{\, 2}}{f^{\, 2}} - \frac {\big (\frac {\kappa f^{\, \prime}}{f}\big )^2}{\kappa} \\
& = \, (\operatorname{resp.}\, \le \, ,\, \ge  ) - \kappa G - \frac {g^2}{\kappa} \quad \operatorname{a.} \operatorname{e.}\quad \operatorname{in}\quad (0, \tau). 
\end{aligned}
\label{2.6}
\end{equation} 
 Thus, we obtain the first-order nonlinear differential equation $( \operatorname{resp.}$ differential inequalities $)$
\begin{equation}
 g^{\, \prime} + \frac {g^2}{\kappa} + \kappa G = 0\, (\operatorname{resp.}\, \le 0\, ,\, \ge 0 )
 %\quad \operatorname{a.}\, \operatorname{e.}\quad \operatorname{in}\quad (0, \tau) 
\label{2.7}
\end{equation}  
a.e. in $(0, \tau)\, ,$ with the initial condition at $0$ in \eqref{2.1} being converted to the asymptotic condition at $0$ in \eqref{2.2}. Indeed,
 
\begin{equation}\begin{aligned}
g(t) & = \kappa \frac{f^{\, \prime}(t)}{f(t)} = \kappa \frac{\kappa +O(t)}{\kappa t+O(t^2)}
= \frac{\kappa }{t} \cdot \frac{1+O(t)}{1+O(t)} \\ 
& = \frac{\kappa }{t} \bigg (\frac{1}{1+O(t)} +O(t) \bigg )  = \frac{\kappa }{t} + O(1)
\quad \operatorname{as}\quad t \to 0^+
\end{aligned}\label{2.8}
\end{equation}
$(\Leftarrow)$ Observe $\frac {g(s)}{\kappa} - \frac {1}{s}$ is locally bounded in a neighborhood of $0$ and integrable in $[0, \tau)$.
In view of \eqref{2.5}, $f \in C([0, \tau]) \cap C^1(0,\tau)\, , f(0)=0\, ,\, f > 0$ on $(0, \tau)\, $ and $f^{\, \prime} \in AC(0, \tau)\, .$ Indeed, for every $t \in (0, \tau)\, ,$
\begin{equation}
\begin{aligned}
f^{\, \prime}(t) & = \kappa \exp \bigg( \int_0^t \frac {g(s)}{\kappa}  - \frac 1s\, ds\bigg) + \kappa t \exp \bigg( \int_0^t \frac {g(s)}{\kappa}  - \frac 1s\, ds\bigg) \cdot \bigg (\frac {g(t)}{\kappa}  - \frac 1t \bigg )\\ 
& =  g(t)  \cdot t \exp \bigg( \int_0^t \frac {g(s)}{\kappa}  - \frac 1s\, ds\bigg)  \\ 
& = g(t)  \cdot  \frac {f(t)}{\kappa}.
\end{aligned}\label{2.9}
\end{equation}
Thus, \eqref{2.4} holds and as $t \to 0^{+}\, ,$ $f^{\, \prime}(t) = \bigg ( \frac{\kappa }{t} + O(1) \bigg ) \cdot t \exp \bigg( \int_0^t \frac {g(s)}{\kappa}  - \frac 1s\, ds \bigg )\to \kappa\, .$
Hence, $f^{\, \prime} (0)  =  \kappa\, ,$ and it follows from \eqref{2.9} and \eqref{2.2} that
$$f^{\, \prime\prime} (t)  =  g^{\, \prime} (t) \cdot  \frac {f(t)}{\kappa} + \frac {g(t)^{\, 2}}{\kappa}  \cdot  \frac {f (t)}{\kappa} =  \, (\operatorname{resp.}\, \le \, ,\, \ge  ) - G (t) f (t)\, \operatorname{a.} \operatorname{e.}\, \operatorname{in}\, (0, \tau).$$
Consequently, we obtain the second-order linear differential equation $( \operatorname{resp.}$ differential inequalities $)$ 
\begin{equation}
f^{ \, \prime \prime} (t)+ G (t)\, f (t) = 0\, (\operatorname{resp.}\, \le 0\, ,\, \ge 0 )
\label{2.10}
\end{equation}
a.e. in $(0, \tau)$, with $f (0) = 0$ and $f^{\, \prime}(0) = \kappa$.
\end{proof}

\begin{corollary}
$(1)$  Let $G: (0, \tau) \subset (0, \infty) \to \mathbb R$, $f \in C^1(0, \tau)\, ,$ $f > 0$ in $(0, \tau)\, ,$ $f^{\prime} \in AC(0, \tau)\, ,$ and $g \in AC(0, \tau)$.  Then  
\begin{equation}
{\begin{cases}
f^{ \, \prime \prime} (t)+ G (t)\, f (t) = 0\, (\operatorname{resp.}\, \le 0\, ,\, \ge 0 )\quad \operatorname{a.} \operatorname{e.}\quad \operatorname{in}\quad (0, \tau) \\
f (0) = 0\, , f^{\, \prime}(0) = 1
\end{cases}}\label{2.11}
\end{equation}

\noindent
is dual to $($or if and only if $)$

\begin{equation}{\begin{cases}
g^{\, \prime} (t) + g(t)^2 + G(t) = 0\, (\operatorname{resp.}\, \le 0\, ,\, \ge 0 )\quad \operatorname{a.} \operatorname{e.}\quad \operatorname{in}\quad (0, \tau)  \\
g(t)  = \frac{1}{t} + O(1)
\quad \operatorname{as}\quad t \to 0^+\, .
\end{cases}}\label{2.12}
\end{equation}

\noindent
Furthermore, \eqref{2.11} is transformed into \eqref{2.12}, 
via the transformer
\eqref{2.3} in which $\kappa =1$.  
Conversely, \eqref{2.11}  is reversed from \eqref{2.12} via the reverser
\eqref{2.5} in which $\kappa =1$, and $f(t)$ in \eqref{2.11} enjoys
\eqref{2.4} in which $\kappa =1$, 

\noindent
$(2)$ Theorem \ref{T: 2.1} is equivalent to $(1)$.

\noindent
$(3)$ Let $\tilde f = \kappa f$ and $\tilde g = {\kappa }{g}$. Then the following four statements are equivalent:  
$$f\, \operatorname{satisfies}\, \eqref{2.11}. \Leftrightarrow \tilde f\, \operatorname{satisfies}\, \eqref{2.1}. \Leftrightarrow \tilde g \, \operatorname{satisfies}\,  \eqref{2.2}. \Leftrightarrow g \, \operatorname{satisfies}\, \eqref{2.12}. $$
\label{C: 2.1}\end{corollary}
\begin{proof} $(1)$ Choose $\kappa = 1$ in Theorem \ref{T: 2.1}. 
$(2)$ and $(3)$ Let $\tilde f = \kappa f$. Then $f$ satisfies \eqref{2.11} if and only if $\tilde f$ satisfies \eqref{2.1}. Let $g = \frac {\tilde g}{\kappa }$. Then $g$ satisfies \eqref{2.12} if and only if $\tilde g$ satisfies \eqref{2.2}.
Furthermore, by Theorem \ref{T: 2.1}, 
$\tilde f$ satisfies \eqref{2.1} if and only if 
$\tilde g$ satisfies \eqref{2.2}. Consequently, ``\eqref{2.11} is dual to \eqref{2.12}" is equivalent to ``\eqref{2.1} is dual to \eqref{2.2}" and we have shown $(3)$.
\end{proof}

\subsection{\bf Dualities between Comparison Theorems in Differential Equations}
\medskip

To each differential equation $(2.1)\, \big (\operatorname{resp.}\, (2.2)\big )$ , there is associated Comparison Theorem E $\, (\operatorname{resp.}\, \operatorname{Theorem }\, \ref{T: 2.2})$ between its supersolution and subsolution with appropriate initial condition $\, (\rm{resp.}\, \rm{asymptotic}\, \rm{condition})$.
We state and prove these ``opposites" as well as ``unity" in a general form.

\begin{theoremE}$($Sturn-Liouville Type Comparison Theorem$)\, $
Let $G_i : (0, t_i) \subset (0, \infty)$ $\to \mathbb R$ be functions satisfying \eqref{2.15} on $(0, t_1) \cap (0, t_2)$, and $f_i \in C([0, t_i]) \cap C^1(0,t_i)$ with 
$f_i^{\prime}\in AC(0,t_i)$ for $i=1, 2\, $  
be solutions of the problems
\begin{equation}
\begin{cases}
 f_1^{\, \prime\prime} + G_1f_1\,   \le\,  0 \quad \operatorname{a}.\, \operatorname{e}.\quad \operatorname{in}\quad (0, t_1)  \\
f_1(0) = 0\, , f_1^{\, \prime}(0) = \kappa_1
\end{cases}\label{2.13}
\end{equation}
\begin{equation}
\begin{cases}
f_2^{\, \prime\prime} + G_2 f_2 \,  \ge\,  0 \quad \operatorname{a}.\, \operatorname{e}.\quad \operatorname{in}\quad (0, t_2)  \\
f_2(0) = 0\, , f_2^{\, \prime}(0) = \kappa_2
\end{cases}\label{2.14}
\end{equation}with   \begin{equation}G_2 \le G_1\label{2.15}
\end{equation}on $(0, t_1) \cap (0, t_2)$ and\begin{equation}  0 < \kappa_1 \le  \kappa_2\, ,\label{2.16}
\end{equation}
where $f_i > 0\, $ on $(0, t_i)\, , i=1, 2\, .$  Then 

\begin{equation}
f_2 > 0\, \operatorname{on}\, (0, t_1), \quad \frac{f_1^{\, \prime}}{f_1} \le \frac{f_2^{\, \prime}}{f_2}\, \operatorname{in}\, (0, t_1)\quad  \operatorname{and}\quad f_1 \le f_2 \, \operatorname{in}\, (0,t_1). \label{2.17}
\end{equation}
\end{theoremE}

\begin{proof}
By \eqref{2.13}, \eqref{2.14} and \eqref{2.15}, \begin{equation}(f_2^{\, \prime}f_1 - f_1^{\, \prime}f_2)(0) = 0\label{2.18}
\end{equation}and
\begin{equation}(f_2^{\, \prime}f_1 - f_1^{\, \prime}f_2)^{\, \prime} = f_2^{\, \prime\prime}f_1 - f_1^{\, \prime\prime}f_2\, \ge (G_1 - G_2) f_1 f_2\, \ge 0 \label{2.19}
\end{equation}a. e. in $(0, t_1) \cap (0, t_2)\, $ hold. Whence $f_2^{\, \prime}f_1 - f_1^{\, \prime}f_2 \ge 0\, $ and
\begin{equation}\frac {f_1^{\, \prime}}{f_1} \le \frac {f_2^{\, \prime}}{f_2} \label{2.20}
\end{equation}
in $(0, t_1) \cap (0, t_2)\, .$ Integrating \eqref{2.20} from $\epsilon (> 0)$ to $r (< \min \{t_1, t_2\})\, ,$ and passing $\epsilon$ to $0$ from the right, one has
\[
\ln f_1(t) \big {|}_{t=\epsilon}^{r} \le \ln f_2(t) \big {|}_{t=\epsilon}^{r} \quad \text{and} 
\]
$$f_1(r) = \lim_{\epsilon \to 0^+}f_1(r)\le \lim_{\epsilon \to 0^+}\frac{f_1(\epsilon)}{f_2(\epsilon)}f_2(r)= \frac{\kappa_1}{\kappa_2}f_2(r) \le f_2(r)\, \operatorname{in} \, [0, t_1)\cap [0, t_2)\, .$$
By the continuity of $f_1$ and $f_2$ ,
\begin{equation}f_1 \le f_2 \label{2.21}
\end{equation}
on $[0, t_1] \cap [0, t_2]\, .$ Let $t_2 = \sup \{ t: f_2 > 0\quad \operatorname{on}\quad (0, t)\}\, .$
We claim $t_1 \le t_2$. Otherwise, $t_2 < t_1$ would lead to, via \eqref{2.21} $0 < f_1(t_2) \le f_2(t_2)$, contradicting by continuity $f_2(t_2) = 0\, .$  
\end{proof}

\begin{corollary}
$(1)$  Let $G_i : (0, t_i) \subset (0, \infty) \to \mathbb R$ be functions satisfying \eqref{2.15} on $(0, t_1) \cap (0, t_2)$, and let $f_i \in C([0, t_i]) \cap C^1(0, t_i)\, ,$ with $f_i^{\prime} \in AC(0, t_i)\, $
be positive solutions of the problems
\begin{equation}
\begin{cases}
 f_1^{\, \prime\prime} + G_1f_1\,   \le\,  0 \quad \operatorname{a}.\, \operatorname{e}.\quad \operatorname{in}\quad (0, t_1)  \\
f_1(0) = 0\, , f_1^{\, \prime}(0) = 1
\end{cases}\label{2.22}
\end{equation}
\begin{equation}
\begin{cases}
f_2^{\, \prime\prime} + G_2 f_2 \,  \ge\,  0 \quad \operatorname{a}.\, \operatorname{e}.\quad \operatorname{in}\quad (0, t_2)  \\
f_2(0) = 0\, , f_2^{\, \prime}(0) = 1
\end{cases}\label{2.23}
\end{equation}
Then \eqref{2.17} holds.

\noindent
$(2)$ Theorem $\operatorname{E}$ is equivalent to $(1)$.
\label{C: 2.2}\end{corollary}
\begin{proof} $(1)$ Choose $\kappa _1 = \kappa _2 = 1$ in Theorem $\operatorname{ E}$. 
$(2)$ Let $\tilde f_i = \kappa_i f_i$. Then $f_1$ satisfies \eqref{2.22} if and only if $\tilde f_1$ satisfies \eqref{2.13},  $f_2$ satisfies \eqref{2.23} if and only if $\tilde f_2$ satisfies \eqref{2.14}. Furthermore, \eqref{2.20} holds in $(0, t_1) \cap (0, t_2)\, $ if and only if $\frac{\tilde f_1^{\, \prime}}{\tilde f_1} \le \frac{\tilde f_2^{\, \prime}}{\tilde f_2}$ in $(0, t_1) \cap (0, t_2)\, .$  These imply that \eqref{2.21} holds on $[0, t_1] \cap [0, t_2]\, $ if and only if $\tilde f_1 \le \tilde f_2 $ in on $[0, t_1] \cap [0, t_2]\, .$ It follows that \eqref{2.17} holds. Consequently, Theorem $\operatorname{E}$ is equivalent to $(1)$.
\end{proof}

\begin{theorem}$($Comparison Theorem for supersolutions and subsolutions of Riccati type equations$)$
Let functions $G_i : (0, t_i) \subset (0, \infty) \to \mathbb R$ satisfy \eqref{2.15} on $(0, t_1) \cap (0, t_2)$, and let $g_i \in \operatorname{AC}(0, t_i)\, , i = 1, 2$ be solutions of
\begin{equation}\begin{cases}
{g_1}^{\, \prime} + \frac {{g_1}^2}{\kappa_1} + \kappa_1 G_1 \le\, 0\quad  \operatorname{a.}\, \operatorname{e.}\quad \operatorname{in}\quad (0, t_1) \\
g_1(t) = \frac {\kappa_1}{t} + O(1)\quad \operatorname{as}\quad t \to 0^{+}\, 
\end{cases}\label{2.24}
\end{equation}
\begin{equation}\begin{cases}
{g_2}^{\, \prime} + \frac {{g_2}^2}{\kappa_2} + \kappa_2 G_2 \ge\, 0\quad  \operatorname{a.}\, \operatorname{e.}\quad  \operatorname{in}\quad (0, t_2)  \\
g_2(t) = \frac {\kappa_2}{t} + O(1)\quad \operatorname{as}\quad t \to 0^{+}\,  
\end{cases}\label{2.25}
\end{equation}
with constants $\kappa _i$ satisfying \eqref{2.16}.  
 Then in $(0, t_1)\, ,$ \begin{equation}  g_1 \le g_2\, .\label{2.26}
\end{equation}\label{T: 2.2}
\end{theorem}
\begin{proof}
Proceeding as in the proof of Theorem \ref{T: 2.1}, we use
\eqref{2.5} in which $f = f_i$, $g = g_i$ and
$\kappa= \kappa_i, i =1, 2\,$, then the systems \eqref{2.24} and \eqref{2.25} are transformed into 
\eqref{2.13} and \eqref{2.14} respectively, satisfying \eqref{2.15} on $(0, t_1) \cap (0, t_2)$, \eqref{2.16}, and  
enjoying  \eqref{2.4} (in which $f = f_i$, $g = g_i$ and
$\kappa= \kappa_i $),
$f_i > 0$ on $(0, t_i)\, ,$ $f_i \in C^0([0, t_i]) \cap C^1(0, t_i)\, $ and $f_i^{\prime} \in AC(0, t_i)$, $i = 1, 2.$ 
It follows from Theorem E that \eqref{2.17} holds. By \eqref{2.4}, $\frac {g_1}{\kappa_1} \le \frac {g_2}{\kappa_2}$. Consequently, via \eqref{2.16} we obtain the desired \eqref{2.26}.
\end{proof}

\begin{corollary}
$(1)$  Let functions $G_i : (0, t_i) \subset (0, \infty) \to \mathbb R$ satisfy \eqref{2.15} on $(0, t_1) \cap (0, t_2)$,  and  $g_i \in \operatorname{AC}(0, t_i)\, , i = 1, 2$ be solutions of
\begin{equation}\begin{cases}
{g_1}^{\, \prime} + {g_1}^2 + G_1 \le\, 0\quad  \operatorname{a.}\, \operatorname{e.}\quad  \operatorname{in}\quad (0, t_1) \\
g_1(t) = \frac {1}{t} + O(1)\quad \operatorname{as}\quad t \to 0^{+}\, 
\end{cases}\label{2.27}
\end{equation}
\begin{equation}\begin{cases}
{g_2}^{\, \prime} + {g_2}^2 + G_2 \ge\, 0\quad  \operatorname{a.}\, \operatorname{e.}\quad  \operatorname{in}\quad (0, t_2)  \\
g_2(t) = \frac {1}{t} + O(1)\quad \operatorname{as}\quad t \to 0^{+}\, , 
\end{cases}\label{2.28}
\end{equation}Then \eqref{2.26} holds on $(0, t_1)$.
\smallskip

\noindent
$(2)$ Theorem \ref{T: 2.2} is equivalent to $(1)$.
\label{C: 2.3}\end{corollary}
\begin{proof} $(1)$ Choose $\kappa _1 = \kappa _2 = 1$ in Theorem \ref{T: 2.2}. 
$(2)$ Enough to show that $(1) \Rightarrow$ Theorem \ref{T: 2.2}.  Let $\tilde g_i = \frac {g_i}{\kappa_i}$. Then $\tilde g_1$ satisfies \eqref{2.27} if and only if $g_1$ satisfies \eqref{2.24} and $\tilde g_2$ satisfies \eqref{2.28} if and only if $g_2$ satisfies \eqref{2.25}. 
It follows from $(1)$ that $\tilde g_1 \le \tilde g_2$. Hence, 
\begin{equation}\frac {g_1}{\kappa_1} = \tilde g_1  \le \tilde g_2 = \frac {g_2}{\kappa_2}\label{2.29}
\end{equation}
in $(0, t_1)$. Consequently, using \eqref{2.16} yields the desired \eqref{2.26} on $(0, t_1)$.
\end{proof}

\noindent
\subsection{\bf Dualities in ``Swapping" Comparison Theorems in Differential Equations}
\smallskip

We note the duality between Jacobi type equation $(2.1)$ and Riccati type equation $(2.2)$ leads to  
the duality between 
the above two types of corresponding comparison theorems in differential equations - Theorem E and Theorem \ref{T: 2.2}.
This duality between two comparison theorems in differential equations of the same type, in turn gives rise to the duality in swapping differential equations of different types and hence generates
Comparison Theorems on Differential Equations of Mixed Types I and II. Whereas, mixed type I concerns with comparing supersolutions of a Riccati type equation with subsolutions of a Jacobi type equation, mixed type II concerns with comparing supersolutions of a Jacobi type equation with subsolutions of a Riccati type equation.   
\medskip   

\noindent
\subsection{\bf Comparison Theorems on Differential Equations of Mixed Types}\smallskip

\begin{theorem}$($Comparing Differential Equations of Mixed Type $\operatorname{I})\, $
Let functions $G_i : (0, t_i) \subset (0, \infty) \to \mathbb R$ satisfy \eqref{2.15} on $(0, t_1) \cap (0, t_2)$.  Let $g_1\in AC(0, t_1)\, $  
be a solution of \eqref{2.24} and $f_2 \in C([0, t_2]) \cap C^1(0,t_2)$ with 
$f_2^{\prime}\in AC(0,t_2)$ be a positive solution of \eqref{2.14}
with constants $\kappa _i$ satisfying \eqref{2.16}. Then $t_1 \le t_2$ and \begin{equation}g_1 \le \frac {\kappa_2 f_2^{\, \prime}}{f_2}\, \label{2.30}
\end{equation} on $(0, t_1)\, .$\label{T: 2.3}
\end{theorem}

\begin{proof}
Proceed as in the proof of Theorem \ref{T: 2.1}, via \eqref{2.3} in which $g=g_2\, , f = f_2\, ,$ and $\kappa = \kappa_2$, the system \eqref{2.14} is transformed into \eqref{2.25}
Applying Theorem \ref{T: 2.2} to \eqref{2.24} and \eqref{2.25}, we have \eqref{2.26}, and hence \eqref{2.30}, via \eqref{2.3} on $(0, t_1)\, .$
\end{proof}

\begin{corollary}
$(1)$ Let $G_i : (0, t_i) \subset (0, \infty) \to \mathbb R$, $i = 1, 2$ satisfy \eqref{2.15} on $(0, t_1) \cap (0, t_2)$, $g_1\in AC(0, t_1)\, $  
be a solution of \eqref{2.27},
and $f_2 \in C([0, t_2]) \cap C^1(0,t_2)$ with 
$f_2^{\prime}\in AC(0,t_2)$ be a positive solution of 
\eqref{2.23}. Then $t_1 \le t_2$ and \begin{equation}g_1 \le \frac {f_2^{\, \prime}}{f_2}\, \label{2.31}
\end{equation} 
on $(0, t_1)\, .$
\smallskip

\noindent
$(2)$ Theorem \ref{T: 2.3} is equivalent to $(1)$.\label{C: 2.4}\end{corollary}

\begin{proof} $(1)$ Choose $\kappa _1 = \kappa _2 = 1$ in Theorem \ref{T: 2.3}. 
$(2)$ Enough to show that $(1) \Rightarrow$ Theorem \ref{T: 2.3}.  Let $\tilde g_1 = \frac {g_1}{\kappa_1}$ and $\tilde f_2 = \kappa _2 f_2$. Then $\tilde g_1$ satisfies \eqref{2.27} if and only if $g_1$ satisfies \eqref{2.24} and $\tilde f_2$ satisfies \eqref{2.23} if and only if $f_2$ satisfies \eqref{2.14}. 
It follows from $(1)$ that $\tilde g_1 \le \frac {\tilde {f_2}^{ \prime}}{\tilde  f_2}$. Hence, on $(0, t_1)$,
\[\frac {g_1}{\kappa_1} = \tilde g_1  \le \frac {\tilde  f_2^{ \prime}}{\tilde  f_2} = \frac  {f_2^{ \prime}}{ f_2} \]
Consequently, using \eqref{2.16} yields $g_1\le \kappa_1 \frac  {f_2^{ \prime}}{ f_2} \le \kappa _2 \frac  {f_2^{ \prime}}{ f_2}$, the desired \eqref{2.30} on $(0, t_1)$.
\end{proof}

\begin{theorem}$($Comparing Differential Equations of Mixed Type $\operatorname{II})\, $
Let functions $G_i : (0, t_i) \subset (0, \infty) \to \mathbb R$, $i = 1, 2$ satisfy \eqref{2.15} on $(0, t_1) \cap (0, t_2)$.  Let $f_1 \in C([0, t_i]) \cap C^1(0,t_1)$ with 
$f_1^{\prime}\in AC(0,t_1)$ be a positive solution of 
\eqref{2.13}
and $g_2\in AC(0, t_2)\, $  
be a solution of 
\eqref{2.25}
with constants $\kappa _i$ satisfying \eqref{2.16}. Then $t_1 \le t_2$ and \begin{equation}\frac {\kappa_1 f_1^{\, \prime}}{f_1} \le g_2\, \label{2.32}
\end{equation} on $(0, t_1)\, .$\label{T: 2.4}
\end{theorem}

\begin{proof} 
Proceed as in the proof of Theorem \ref{T: 2.1}, via \eqref{2.3} in which $g=g_1\, , f = f_1\, ,$ and $\kappa = \kappa_1$, the system \eqref{2.13} is transformed into \eqref{2.24}.
Applying Theorem \ref{T: 2.2} to \eqref{2.24} and \eqref{2.25}, we have \eqref{2.26}, and hence \eqref{2.32}, via \eqref{2.3} on $(0, t_1)\, .$
\end{proof}

Similarly, we have

\begin{corollary}
$(1)$ Let $G_i : (0, t_i) \subset (0, \infty) \to \mathbb R$, $1 \le i \le 2$ satisfy \eqref{2.15} on $\quad (0, t_1) \cap (0, t_2)$.   Let $f_1 \in C([0, t_2]) \cap C^1(0,t_1)$ with 
$f_1^{\prime}\in AC(0,t_1)$ be a positive solution of 
\eqref{2.22}
and $g_2\in AC(0, t_2)\, $  
be a solution of \eqref{2.28}.
Then \begin{equation}\frac {f_1^{\, \prime}}{f_1} \le g_2\, \label{2.33}
\end{equation} on $(0, t_1)\, .$
\smallskip

\noindent
$(2)$ Theorem \ref{T: 2.4} is equivalent to $(1)$.\label{C: 2.5}\end{corollary}

\noindent
\subsection{\bf Applications in Comparison Theorems in Riemannian Geometry}\smallskip

When the radial Ricci curvature $\operatorname{Ric}^{\operatorname{M}}_{\operatorname{rad}}$ of a manifold $M$ has a lower bound, it can be viewed as a supersolution of Riccati type 
equation in disguise. Similarly, when the radial curvature $K$ of a manifold has an upper bound $(\operatorname{resp.}\, \operatorname{a}\, \operatorname{lower}\, \operatorname{bound})\, ,$ it can be viewed as a subsolution $(\operatorname{resp.}\, \operatorname{supersolution})$ of Riccati type equation  in disguise. Thus, utilizing dualities in comparison theorems in differential equations of mixed types, we generate comparison theorems in Riemannian geometry
and provide simple and direct proofs.\smallskip

Let $M$ be an $n$-dimensional Riemannian manifold. 
$M$ is said to be a {\it manifold with a pole} $x_0\, ,$ if
$D(x_0) = M \backslash \{x_0\}\, .$  
We recall the radial vector field $\partial $ on $D(x_0)$ is the unit
vector field such that for any $x\in D(x_0)$, $\partial
\left( x\right) $ is the unit vector tangent to the unique geodesic $\gamma$ joining $%
x_{0}$ to $x$ and pointing away from $x_{0}.$ A radial plane is a plane $\pi
$ which contains $\partial (x)$\ in the tangent space $T_{x}M.$ By the
{\it radial curvature} $K $ of a manifold, we mean the
restriction of the sectional curvature function to all the radial planes.
We define $K(t)$ to be the {\it radial curvature} of $M$ at $x$
such that $r(x) = t$. We also study
\begin{definition1}
The {\it radial Ricci curvature} of a manifold is the
restriction of the Ricci curvature function to all the radial vector fields. 
Denoted $\operatorname{Ric}^{\operatorname{M}}_{\operatorname{rad}}(r)$, the {\it radial Ricci curvature} of $M$ at $x$
 such that $\operatorname{dist}(x_0, x) = r$.
\end{definition1}  
Let a tensor $g - dr \otimes dr = 0$  on the radial direction, and be just the
metric tensor $g$ on the orthogonal complement $\partial ^{\bot}$. At $x \in M\, ,$ the {\it  Hessian} of $r\, ,$ denoted by $\operatorname{Hess} r$ is a quadratic form on $T_x M$ given by
$\operatorname{Hess} r\, (v,w) = (\nabla _v dr)w = g(\nabla _v \nabla r, w)\, $ for $v,w \in T_x M\, .$ Here $\nabla _v$ is the covariant derivative in $v$ direction, $\nabla r$ is the gradient vector field of $r\, ,$ and hence is dual to the differential $dr$ of $r\, .$ Thus, $\operatorname{Hess} r (\nabla r,\nabla r) = 0\, .$ The {\it Laplacian} of $r\, ,$ is defined to be
$\Delta r = \operatorname{trace}\big(\operatorname{Hess} r \big)\, .$ We say $\Delta r \le f(r)\, $ holds \emph{weakly} on $M\, ,$ if for every $0 \le \varphi(r) \in C_0^{\infty}(M)\, ,$
$\int _M \varphi (r)\Delta r\, dv \le \int _M \varphi (r) f(r)\, dv\, . $ We recall
\begin{lemA} \cite [Lemma 9.1] {HLRW} $($see \cite {PRS}, \cite {WL}$)$   If $\Delta r \le f(r)\, $ holds pointwise in $D(x_0)\, ,$ where $f\in C^0(0, \infty)\, ,$ then $\Delta r \le f(r)\, $ holds weakly on $M\, .$
\end{lemA}
\smallskip

\noindent
\subsection{\bf Under radial Ricci curvature assumptions}\smallskip

\begin{theorem}$($Laplacian Comparison Theorem$)$  Let functions $G_i : (0, t_i) \subset (0, \infty) \to \mathbb R$, $i = 1, 2$ satisfy \eqref{2.15} on $(0, t_1) \cap (0, t_2)$. 
Assume 
\begin{equation}(n-1)\, G_1(r)\le \, \operatorname{Ric}^{\operatorname{M}}_{\operatorname{rad}}(r) \label{2.34}
\end{equation} 
$\operatorname{on}\, \overset {\circ} B_{t_1}(x_0)\subset D(x_0)\, ,$ and let 
$f_2 \in C([0, t_2]) \cap C^1(0,t_2)$ with 
$f_2^{\prime}\in AC(0,t_2)$
be a positive solution of 
\begin{equation}
\begin{cases}
f_2^{\prime\prime} + G_2 f_2 \ge 0\quad  \operatorname{a.}\, \operatorname{e.}\quad \operatorname{in}\quad (0, t_2) \\
f_2(0) = 0\, , f_2^{\prime}(0) = n-1\, .
\end{cases}\label{2.35}
\end{equation}
Then $\quad t_1 \le t_2\quad \, $ and
\begin{equation}
\Delta r\le (n-1)\frac{f_2^{\prime}}{f_2}(r)\label{2.36}
\end{equation}
in $\overset {\circ} B_{t_1}(x_0)\,.$ If in addition, 
\eqref{2.34} occurs in $D(x_0)\, ,$
then  
\eqref{2.27}  holds pointwise on $D(x_0)\,$ and weakly on $M\, .$
\label{T: 2.5}
\end{theorem}
\begin{proof} 
For a bilinear form $\mathcal A$, we write $|\mathcal A|^2 = \text{trace}(\mathcal A {\mathcal A}^t)\, .$ 
Recall 

\begin{theoremF}$($Weitzenb\"ock formula$)$ For every function $f \in C^3(M)\, ,$ 
\begin{equation} \frac 12\Delta |\nabla f|^2 = | \operatorname{Hess}\, f|^2 + \langle \nabla \Delta f, \nabla f \rangle+ \operatorname{Ric}\big (\nabla f,\nabla f\big ) \label{2.37}\end{equation}\end{theoremF}
Substituting the distance function $r(x)$ for $f(x)$ into Weitzenb\"ock formula \eqref{2.37}, we have inside the cut locus of $x_0$ $(|\nabla r| = 1)$, via Gauss lemma   
\[ 0 = |\text{Hess}\, r|^2 + \frac {\partial }{\partial r} \Delta r + \text{Ric}(\frac {\partial }{\partial r}, \frac {\partial }{\partial r})  \]

Since $\operatorname{Hess} (\frac {\partial }{\partial r}, \frac {\partial }{\partial r}) = 0\, ,|\text{Hess}\, r|^2 \ge \frac {(\Delta r)^2}{n-1}\, ,$ $\text{Ric}(\frac {\partial }{\partial r}, \frac {\partial }{\partial r}) = \operatorname{Ric}^{\operatorname{M}}_{\operatorname{rad}}(r)\, $ and \eqref{2.34} occurs, it follows that
\begin{equation} 
\begin{cases}
\frac {\partial }{\partial r} \big ( \Delta r \big )  + \frac {(\Delta r)^2}{n-1} + (n-1) G_1(r) \le 0 \quad \operatorname{on}\quad \overset {\circ} B_{t_1}(x_0) \\
\Delta r = \frac{n-1}{r} + O(1)\quad \operatorname{as}\quad t \to 0^{+}\, .
\end{cases}\label{2.38}\end{equation}
Let $g_1 = \Delta r\, ,$ $\kappa_1 = n-1 = \kappa_2\, .$ Then \eqref{2.38} becomes \eqref{2.24} and \eqref{2.35} becomes \eqref{2.14}. Via 
\eqref{2.24} and \eqref{2.35} with $G_i$ satisfying \eqref{2.15} on $(0, t_1) \cap (0, t_2)$ and $\kappa_1$ satisfying \eqref{2.16}. Applying Theorem \ref{T: 2.3} we have \eqref{2.30} holds on $(0, t_1)$, i.e., \eqref{2.36} holds in $\overset {\circ} B_{t_1}(x_0)\, .$ 
If in addition, \eqref{2.34} occurs in $D(x_0)\, ,$
then \eqref{2.36} holds pointwise on $D(x_0)$. By Lemma A or using Green's Identity and a double limiting argument 
 (see \cite [Lemma 9.1] {HLRW},\cite {PRS}, \cite {WL}), \eqref{2.36} holds weakly on $M$.
\end{proof}

In applying Theorem \ref{T: 2.5}, when a specific Ricci curvature lower bound assumption is given, we do not need to assume $f_2$ in \eqref{2.35} for comparison. Instead, we find $f_2$ in \eqref{2.44} and estimate $\frac {f_2^{\prime}}{f_2}$ in \eqref{2.49} and obtain:

\begin{theorem}
$(1)$ If  \begin{equation}- (n-1)\frac {A(A-1)}{r^2}\leq \operatorname{Ric}^{\rm{M}}_{\rm{rad}}(r),\quad \operatorname{where} \quad  A \ge 1\,  \label{2.39}\end{equation} 
on $\overset {\circ} B_{t_1}(x_0)\subset D(x_0),$
then
\begin{equation}
 \Delta r  \leq (n-1)\frac{A}{r}\label{2.40}
\end{equation}
on $\overset {\circ} B_{t_1}(x_0)$. If in addition, \eqref{2.39} occurs on $D(x_0)\, ,$  then \eqref{2.40} holds pointwise on $D(x_0)$ and weakly on $M\, .$

\noindent
$(2)$ Equivalently, if \begin{equation}\quad - (n-1) \frac {A(A-1)}{(c+r)^2}\leq \operatorname{Ric}^{\rm{M}}_{\rm{rad}}(r), \quad  A \ge 1\, 
\label{2.41}\end{equation} on $\overset {\circ} B_{t_1}(x_0)\subset D(x_0)$, $\operatorname{where} \,  c \ge 0\, ,$
then \eqref{2.40} holds on $\overset {\circ} B_{t_1}(x_0)$. If in addition \eqref{2.41} occurs on $D(x_0)$, then \eqref{2.40} holds pointwise on $D(x_0)$ and weakly on $M\, .$
\label{T: 2.6}
\end{theorem}

\begin{proof}
$ (2) \Rightarrow (1)$
Apparently $(1)$ is a special case $c = 0$ in $(2)$.
$ (1) \Rightarrow (2)$  If \eqref{2.41}  occurs, then \eqref{2.39} takes place, since $[- (n-1) \frac {A(A-1)}{(c+r)^2}, \infty)\subset [- (n-1) \frac {A(A-1)}{r^2}, \infty)\, .$ We then apply $(1)\, $.
Now we prove $(2)$. First assume \eqref{2.41} occurs for $c > 0$.
Choose
$\phi=(n-1)r^{A}\, ,$ where $A \ge  1\, .$
Then $\phi^{\prime} = (n-1)Ar^{A - 1}\, .$  
Hence,  for $r > 0\, $ 
\begin{equation}\frac {\phi^{\prime}}{\phi} = \frac{A}{r}\label{2.42}
\end{equation}
and  $\phi^{\prime \prime} = (n-1) A (A - 1)r^{A - 2}\, .$  
That is, for $r > 0\, ,$ $\phi$ is a solution of the following differential equation:
 \begin{equation}\phi^{\prime \prime} + \frac{-A (A - 1)}{r^2} \phi = 0\, .\label{2.43}
\end{equation}

Let $f_2$ be a positive solution of 
\begin{equation}
\begin{cases}
f_2^{\prime\prime} + G_2 f_2 = 0\quad  \operatorname{a.}\, \operatorname{e.}\quad \operatorname{in}\quad (0, t_2) \\
f_2(0) = 0\, , f_2^{\prime}(0) = \kappa _2\, ,
 \end{cases}\label{2.44}
\end{equation} 
where $G_2 = \frac {-A (A - 1)}{(c+r)^2}\, ,$
$\kappa_ 2 = n-1$, and 
\begin{equation}
t_2 = \sup \{ r: f_2 > 0\quad \operatorname{on}\quad (0, r)\}\, . \label{2.45}
\end{equation}
We note $t_2 = \infty\, .$ This can be seen by comparing the solution $f_2$ of 
\eqref{2.44} with the solution $\tilde{f_2}(r) = (n-1)r$ of the following
$$\begin{cases}
\tilde{f_2}^{\prime\prime} + 0 \cdot \tilde{f_2} = 0, \\
\tilde{f_2}(0) = 0\, , \tilde{f_2}^{\prime}(0) = n-1\, ,
 \end{cases}
$$ and applying a standard Sturm comparison theorem.
Furthermore, \begin{equation}(\phi^{\prime} f_2 - f_2^{\prime} \phi )(0) = 0\, .\label{2.46}
\end{equation} 
In view of \eqref{2.43}, \eqref{2.44}, \eqref{2.45}, and $f_2 \phi  \big( \frac {A(A - 1)}{r^2} - \frac {A(A - 1)}{(c+r)^2}\big) \ge 0 $, for $r \in (0, \infty)\, $
\begin{equation}\aligned
(\phi^{\prime} f_2 - f_2^{\prime} \phi )^{\prime} & = \phi^{\prime\prime} f_2 - f_2^{\prime\prime} \phi
\\
& \ge 0
\endaligned
\label{2.47}
\end{equation}

The monotonicity then implies that \begin{equation}\phi^{\prime} f_2 \ge f_2^{\prime} \phi\, \label{2.48}
\end{equation} for $r \in (0, \infty)\, $
which in turn via \eqref{2.42} yields

\begin{equation}\frac {f_2^{\prime}}{f_2} \le \frac{\phi^{\prime}}{\phi} =  \frac{A}{r}\label{2.49}\end{equation}
on $(0, \infty)\, .$ Applying comparison Theorem \ref{T: 2.5} in which $G_1(r)= \frac {-A (A - 1)}{(c+r)^2} \ge  \frac {-A (A - 1)}{(c+r)^2} = G_2\, (c > 0)\, ,$ where $G_1$ is defined in $(0, t_1)\, ,$ $G_2$ is defined on $(0, \infty)\, $ and $\frac{f^{\prime}_2}{f_2}$ is estimated as in \eqref{2.49}, we have 
shown \eqref{2.40} holds on $\overset {\circ} B_{t_1}(x_0)$ under \eqref{2.41} for every $c > 0\, .$ 
Now we prove $(1)$ by passing $c \to 0$ in \eqref{2.41} in the following way: For every $x \in \overset {\circ} B_{t_1}(x_0)$, there exist sequences $\{x_i\}$ in a radial geodesic ray in $\overset {\circ} B_{t_1}(x_0)$ and $\{c_i > 0\}$ such that $x_i \to x\, , c_i \to 0\, ,$ and $r(x_i) + c_i = r(x)\, .$
It follows that for every $c_i > 0\, ,$ 
\begin{equation}\label{2.50} \begin{aligned}
&\quad \operatorname{Ric}^{\rm{M}}_{\rm{rad}}(r) (x_i) \ge  - (n-1)\frac {A(A-1)}{\big (r(x_i) + c_i \big )^2}  
 \Rightarrow  \Delta r  (x_i) \leq (n-1)\frac{A}{r} (x_i)\, ,\\
 & \operatorname{and}\, \operatorname{hence}\, ,\quad \operatorname{as}\quad i \to \infty\, ,\\
&\quad \operatorname{Ric}^{\rm{M}}_{\rm{rad}}(r) (x) \ge  - (n-1)\frac {A(A-1)}{\big (r(x) \big )^2}  
 \Rightarrow  \Delta r  (x) \leq (n-1)\frac{A}{r} (x)\, .
 \end{aligned}
\end{equation}  

This proves that \eqref{2.40} holds on $\overset {\circ} B_{t_1}(x_0)$ under \eqref{2.41} for $c \ge 0\, ,$ or under \eqref{2.39} on $\overset {\circ} B_{t_1}(x_0)$.
If in addition \eqref{2.39} or \eqref{2.41} occurs on $D(x_0)\, ,$  then
applying Lemma A (Lemma 9.1 in \cite{HLRW}), or a double limiting argument (see \cite [Sect. 3]{WL}), we obtain the desired \eqref{2.40} pointwise on $D(x_0)\, ,$ and weakly on $M\, .$ This proves $(1) \Rightarrow (2)$, and 
$(1)\, .$
\end{proof}
\begin{theorem}
$(1)$ If  \begin{equation}(n-1)\frac {B_1(1-B_1)}{r^2}\leq \operatorname{Ric}^{\rm{M}}_{\rm{rad}}(r),\quad \operatorname{where} \quad  0 \le B_1 \le 1\, \label{2.51}\end{equation} 
on $\overset {\circ} B_{t_1}(x_0)\subset D(x_0)$, then
\begin{equation}
 \Delta r  \leq (n-1)\frac{1+\sqrt{1+4B_1(1-B_1)}}{2r}\label{2.52}
\end{equation}
in $\overset {\circ} B_{t_1}(x_0)$. If in addition, \eqref{2.51} occurs on $D(x_0)\, ,$  then \eqref{2.52} holds pointwise on $D(x_0)$ and weakly on $M\, .$

\noindent
$(2)$ Equivalently, if \begin{equation} (n-1) \frac {B_1(1-B_1)}{(c+r)^2}\leq \operatorname{Ric}^{\rm{M}}_{\rm{rad}}(r) \, ,  \quad  0 \le B_1 \le 1\, 
\label{2.53}\end{equation} on $\overset {\circ} B_{t_1}(x_0) \subset D(x_0)\, $, $\operatorname{where}\,  c \ge 0\, ,$ then \eqref{2.52} holds in $\overset {\circ} B_{t_1}(x_0)$. If in addition \eqref{2.51} occurs on $D(x_0)$, then \eqref{2.52} holds pointwise on $D(x_0)$ and weakly on $M\, .$
\label{T: 2.7}
\end{theorem}

\begin{proof}
$ (2) \Rightarrow (1)$
Obviously, $(1)$ is a special case $c = 0$ in $(2)$.
To prove $(1) \Rightarrow (2)$, enough to show \eqref{2.52} holds under \eqref{2.53}, for any $c > 0$. 
Choose
$\phi_1 = r^{\alpha}\, ,$ where $\alpha = \frac{1+\sqrt{1+4B_1(1-B_1)}}{2}\, .$
Then $\phi_1^{\prime} = \alpha r^{\alpha - 1}\, $ and $(2\alpha - 1)^2 = 1+4B_1(1-B_1)\, , $ i.e. $\alpha (\alpha - 1) =  B_1(1-B_1)\, . $
Hence,  
 \begin{equation}\frac {\phi_1^{\prime}}{\phi_1} = \frac{\alpha}{r} = \frac{1+\sqrt{1+4B_1(1-B_1)}}{2r}\label{2.54}
\end{equation}$\phi_1^{\prime \prime} = \alpha (\alpha - 1) r^{\alpha - 2}\, ,$  i.e.,
for $r > 0\, ,$ \begin{equation}\phi_1^{\prime \prime} - \frac{B_1(1-B_1)}{r^2} \phi_1 = 0\, .\label{2.55}
\end{equation}

Let $f_2 > 0$ satisfy \eqref{2.44},
where $G_2 = - \frac {B_1(1-B_1)}{(c+r)^2}\, ,$  $\kappa_ 2 = n-1$ and $t_2$ is as in \eqref{2.45}. Then $t_2 = \infty\, ,$ by applying the standard Sturm comparison theorem as in the proof of Theorem \ref{T: 2.6}. 
Similarly, we have \eqref{2.46}. \eqref{2.47} and  \eqref{2.48} hold for $r \in (0, \infty)$ in which $\phi = \phi_1$, since $f_2 \phi_1 \big(  \frac {B_1(1-B_1)}{r^2} + G_2 \big)\ge 0\, ,$ for $r \in (0, \infty)\, .$ 
It follows from \eqref{2.54} that 
\begin{equation}\frac {f_2^{\prime}}{f_2} \le \frac{\phi_1^{\prime}}{\phi_1} =  \frac{1+\sqrt{1+4B_1(1-B_1)}}{2r}\label{2.56}
\end{equation}
on $(0, \infty)\, .$
Applying comparison Theorem \ref{T: 2.5}, in which $G_1\big(G_1 : (0, t_1) \to \mathbb R\big ) = \frac {B_1(1-B_1)}{(c+r)^2} \ge$ $- \frac {B_1(1-B_1)}{(c+r)^2} = G_2 \big (G_2 : (0, t_2) \to \mathbb R\big ), \, c > 0$ on $(0, t_1) \cap (0, t_2)$, $\alpha = \frac{1+\sqrt{1+4B_1(1-B_1)}}{2}\, $ and $\frac{f^{\prime}_2}{f_2} \le \frac{1+\sqrt{1+4B_1(1-B_1)}}{2r}$,  we have 
shown \eqref{2.52} holds on $\overset {\circ} B_{t_1}(x_0)$ under \eqref{2.53} for any $c > 0$. Proceeding the limiting argument as in the proof of Theorem \ref{T: 2.6} and \eqref{2.50}, we prove similarly $(1) \Rightarrow (2)$, and 
$(1)\, .$
\end{proof}

\begin{corollary}
If the radial curvature $K$ satisfies  \begin{equation}- \frac {A(A-1)}{r^2}\leq K(r)\, \big (\operatorname{or}\, \operatorname{equivalently},\, - \frac {A(A-1)}{(c+r)^2}\leq K(r)\, , \, c \ge 0\big )\, \operatorname{where}\quad  1 \le A\,  \label{2.57}\end{equation} 
on $\overset {\circ} B_{t_1}(x_0)\subset D(x_0)$, then
\eqref{2.40} holds on $\overset {\circ} B_{t_1}(x_0)$, and if in addition \eqref{2.57} occurs on $D(x_0)\, ,$ then \eqref{2.43} holds pointwise on $D(x_0)$ and weakly on $M$.\label{C: 2.6}
\end{corollary}
\begin{corollary}
If the radial curvature $K$ satisfies  \begin{equation} \frac {B_1(1-B_1)}{r^2}\leq K(r)\, \big (\operatorname{or}\, \operatorname{equivalently}, \, \frac {B_1(1-B_1)}{(c+r)^2}\leq K(r)\, , \, c \ge 0\big )\,  \operatorname{where}\,  0 \le B_1 \le 1\,  \label{2.58}\end{equation}$\operatorname{on} \quad \overset {\circ} B_{t_1}(x_0)\subset D(x_0)\, ,$  
then \eqref{2.52} holds on $\overset {\circ} B_{t_1}(x_0)$, and if in addition \eqref{2.58} occurs on $D(x_0)\, ,$ then \eqref{2.52} holds pointwise on $D(x_0)$ and weakly on $M$.
\label{C: 2.7}
\end{corollary}\smallskip

\noindent
\subsection{\bf Under the radial curvature assumptions}\smallskip

The following Theorem strengthens main theorems in  \cite [Theorem 4.1] {HLRW}, and \cite [Theorem D] {W3}.
We also give direct and simple proofs with applications from the view point of dualities..

\begin{theorem}$($Hessian and Laplacian Comparison Theorems$)\, $
Let 
\begin{equation}G_1(r) \le K(r)  
\label{2.59}\end{equation}
on $\overset {\circ} B_{t_1}(x_0)\subset D(x_0) $
$\big ( \operatorname{resp.}$
\begin{equation}
K(r) \le \widetilde{G_2}(r) \label{2.60}
\end{equation}
on $\overset {\circ} B_{\tilde{t_2}}(x_0)\subset D(x_0)\, \big ),$
and let 
\noindent
$f_2 \in C([0, t_2]) \cap C^1(0,t_2)$ with 
$f_2^{\prime} \in AC(0,t_2)$ be a 
positive solution of  \eqref{2.14},
where $G_i : (0, t_i) \to \mathbb R\,  $satisfy \eqref{2.15} on $(0, t_1) \cap (0, t_2)$,  where $1 \le \kappa _2\, .$
$\big ( \operatorname{resp.}\,  \operatorname{let}\, f_1 \in C([0, \tilde{t_1}]) \cap C^1(0,\tilde{t_1})$ with
$f_1^{\prime} \in AC(0,\tilde{t_1}) $ be a positive solution of 
\begin{equation}
\begin{cases} 
f_1^{\prime\prime} + \widetilde{G_1}f_1 \le 0\, \operatorname{on}\, (0, \tilde{t_1}) \\
f_1(0) = 0\, , f_1^{\prime}(0) = \kappa_1\, ,
\end{cases}
\label{2.61}\end{equation}
where $\widetilde{G_i} : (0, \tilde{t_i}) \to \mathbb R\, $ satisfy 
\begin{equation}
\widetilde{G_2} \le \widetilde{G_1}\label{2.62}\end{equation}
on $(0, \tilde{t_1}) \cap (0, \tilde{t_2})$,  where $0 < \kappa_1 \le 1\big ).$ 

\noindent
Then $\quad t_1 \le t_2\, , $ 
\begin{equation}
\operatorname{Hess} r \le \frac{\kappa_2 f_2^{\prime}}{f_2}(g-dr\otimes dr)\quad \operatorname{and}\quad  \Delta r \le (n-1)\frac{\kappa_2 f_2^{\prime}}{f_2}(r)
\label{2.63} \end{equation} on $\overset {\circ} B_{t_1}(x_0)$ $\big ( \operatorname{resp.}\, $ $\quad \tilde{t_1} \le \tilde{t_2}\, , $ 
\begin{equation}\frac{\kappa_1 f_1^{\prime}}{f_1}\big( g-dr\otimes dr \big) \le \operatorname{Hess} r \quad \operatorname{and}\quad  (n-1)\frac{\kappa_1 f_1^{\prime}}{f_1}(r)\le \Delta r \label{2.64}\end{equation}
$\operatorname{on}\,  \overset {\circ} B_{\tilde{t_1}}(x_0)\big )$ in the sense of quadratic forms. If in addition \eqref{2.59} occurs on $D(x_0)$, then the second part of \eqref{2.63} holds pointwise on $D(x_0)$ and weakly in $M\, .$
\label{T: 2.8}
\end{theorem}
\begin{proof}
Following the notation in \cite {HLRW}. Let $\gamma$ be the unit speed geodesic curve joining $x_0=\gamma (0)\, $ to $x=\gamma (t_0)\, ,$ and
$V$ be a parallel vector field along $\gamma(t)\, ,$ for $0 \le t \le t_0\, .$ A direct computation shows that 
\begin{equation}
\frac {d}{dt} \langle \nabla_V \nabla r , V \rangle + \langle  \nabla _{V} \nabla r, \nabla _{V} \nabla r \rangle
= - \langle R(V, \nabla r ) \nabla r, V\rangle
\le - G_1\label{2.65}
\end{equation}
(see \cite [p.174, (4.5)] {HLRW}). Define $$ \lambda_{\operatorname{max}}(x) = \underset{\{v \in T_x(M)\backslash \{0\},  v \bot \nabla r (x)\}}{\operatorname{max}} \frac {\operatorname{Hess} r (v,v)}{\langle v, v \rangle}\,.$$

Select a unit vector $v$ at $x = \gamma (t_0)$ such that $$\langle \nabla_v \nabla r , v \rangle : = \operatorname{Hess} r (v,v) = \lambda_{\operatorname{max}}\circ \gamma(t_0)\, .$$ Then
$$\langle  \nabla _{v} \nabla r, \nabla _{v} \nabla r \rangle = \lambda_{\operatorname{max}}^2 \circ \gamma(t_0)\, .$$

Let $g_1 = \lambda_{\operatorname{max}}\circ \gamma\, ,$ and let the parallel vector field $V$ along $\gamma$ satisfying $V(t_0)=v\, .$  Then the function $ \operatorname{Hess} r (V,V) - g_1(t)\le 0\, ,$ attains its maximum value $0$ at $t=t_0\, ,$  and at this point 
\[  \frac {d}{dt}\bigg|_{t=t_0} \operatorname{Hess} r (V,V) = \frac {d}{dt}\bigg|_{t=t_0}g_1(t)\, .
\]
It follows from \eqref{2.65} that $g_1$ satisfies \eqref{2.27}
(see \cite [p.175] {HLRW}). By assumption, $f_2$ is a positive solution of \eqref{2.14},
with $\kappa _2 \ge 1\, ,$ and  \eqref{2.15} holds on $(0, t_1)$. 
Applying  Theorem \ref{T: 2.3} in which $g_1 = \lambda_{\operatorname{max}}\circ \gamma$, we have $$\operatorname{Hess} r (w,w) \le \operatorname{Hess} r (v,v) = g_1 \le \frac {\kappa_2 f_2^{\, \prime}}{f_2}$$ $\operatorname{on}\, \overset {\circ} B_{t_1}(x_0)$ for any unit vector $w \bot \nabla r (x)$ at $x\, .$ Taking the trace, we obtain the desired \eqref{2.63}\, $\operatorname{on}\, \overset {\circ} B_{t_1}(x_0)$. If in addition \eqref{2.59} occurs on $D(x_0)$, then the second part of \eqref{2.63}  holds weakly in $M\, $ by Lemma A.

\noindent
Similarly, let $\tilde V$ be a parallel vector field along $\gamma(t)\, ,$ for $0 \le t \le t_0\, .$ By the radial curvature assumption \eqref{2.60},
\begin{equation}
\frac {d}{dt} \langle \nabla_{\tilde V} \nabla r , {\tilde V} \rangle + \langle  \nabla _{\tilde V} \nabla r, \nabla _{\tilde V} \nabla r \rangle
= - \langle R(\tilde V, \nabla r ) \nabla r, \tilde V\rangle
\ge - \widetilde{G_2}\, .\label{2.66}
\end{equation}
Define $$ \lambda_{\operatorname{min}}(x) = \underset{\{v \in T_x(M)\backslash \{0\},  v \bot \nabla r (x)\}}{\operatorname{min}} \frac {\operatorname{Hess} r (v,v)}{\langle v, v \rangle}\, .$$

Select a unit vector $\tilde v$ at $x = \gamma (t_0)$ such that $$\langle \nabla_{\tilde v} \nabla r , \tilde v \rangle : = \operatorname{Hess} r (\tilde v,\tilde v) = \lambda_{\operatorname{min}}\circ \gamma(t_0)\, .$$ Then
$$\langle  \nabla _{\tilde v} \nabla r, \nabla _{\tilde v} \nabla r \rangle = \lambda_{\operatorname{min}}^2 \circ \gamma(t_0)\, .$$

Let $g_2 = \lambda_{\operatorname{min}}\circ \gamma\, ,$ and let the parallel vector field $\tilde V$ along $\gamma$ satisfying $\tilde V(t_0)=\tilde v\, .$  Then the function $ \operatorname{Hess} r (\tilde V,\tilde V) - g_2(t)\ge 0\, ,$ attains its minimum value $0$ at $t=t_0\, ,$  and at this point 
\[  \frac {d}{dt}\bigg|_{t=t_0} \operatorname{Hess} r (\tilde V,\tilde V) = \frac {d}{dt}\bigg|_{t=t_0}g_2(t)\, .
\]
It follows from  \eqref{2.66} that $g_2$ satisfies 
\begin{equation}\begin{cases}
{g_2}^{\, \prime} + {g_2}^2 + \widetilde {G_2} \ge\, 0\quad  \operatorname{a.}\, \operatorname{e.}\quad  \operatorname{in}\quad (0, \tilde {t_2})  \\
g_2(t) = \frac {1}{t} + O(1)\quad \operatorname{as}\quad t \to 0^{+}\, , 
\end{cases}\label{2.67}
\end{equation}
We note $f_1$ is a positive solution of \eqref{2.61}, where $0 < \kappa _1 \le 1$. 
(see \cite [p.176] {HLRW}) and \eqref{2.62} holds on $(0, \tilde{t_1}) \cap (0, \tilde{t_2})$.
Applying Theorem \ref{T: 2.4}, in which $G_1 = \tilde{G_1}, t_1 = \tilde {t_1}$, $G_2 = \tilde{G_2}, t_2 = \tilde {t_2}$ and $\kappa_2 = 1$, we have
 
$$\frac {\kappa_1 f_1^{\, \prime}}{\tilde{f_1}}\le g_2 = \operatorname{Hess} r (\tilde v,\tilde v) \le \operatorname{Hess} r (w,w) $$ on $\overset {\circ} B_{\tilde{t_1}}(x_0)$ for any unit vector $w \bot \nabla r (x)$ at $x\, .$ Taking the trace, we obtain the desired \eqref{2.64} on $\overset {\circ} B_{\tilde{t_1}}(x_0)$.
\end{proof}
\begin{remark2.1}If $f_1 \in C([0, \tilde{t_1}]) \cap C^1(0,\tilde{t_1})$ with
$f_1^{\prime} \in AC(0,\tilde{t_1}) $ is a positive solution of 
\begin{equation}
\begin{cases}
f_1^{\prime\prime} + \widetilde{G_1}f_1 = 0\quad \operatorname{on}\quad (0, \tilde{t_1}) \\
f_1(0) = 0\, , f_1^{\prime}(0) = \kappa_1
 \end{cases}
\label{2.68}\end{equation}
then obviously $f_1$ is a positive solution of \eqref{2.61} Hence, we are ready for applying Theorem \ref{T: 2.8}. The advantage of \eqref{2.68} over \eqref{2.61} is that a solution $f_1$ of \eqref{2.68} will make the subsequent \eqref{3.11}, the monotonicity of $\phi^{\prime} f_1 - f_1^{\prime} \phi$, and the estimate \eqref{3.41} work.  
\end{remark2.1}

\begin{corollary} $(1)\, $ Let the radial curvature $K$ satisfy  \eqref{2.59} $\bigg( \operatorname{resp.}\,  \eqref{2.60} \bigg)$,
and let $f_2 \in C([0, t_2]) \cap C^1(0,t_2)$ with 
$f_2^{\prime} \in AC(0,t_2)$ $\big ( \operatorname{resp.}\,  f_1 \in C([0, \tilde{t_1}]) \cap C^1(0,\tilde{t_1})$ with 
$f_1^{\prime} \in AC(0,\tilde{t_1}) \big )$ be a 
positive solution of \eqref{2.23}
$ \bigg( \operatorname{resp.}\,  
\eqref{2.61}, \kappa _1 = 1\bigg )$.
Assume \eqref{2.15} occurs on $(0, t_1) \cap (0, t_2)$ 
$
\bigg( \operatorname{resp.} \, \eqref{2.62}\, $$ \operatorname{occurs}\, $
on $(0, \tilde{t_1}) \cap (0, \tilde{t_2}) \bigg)$. Then
\eqref{2.63} holds on $\overset {\circ} B_{t_1}(x_0)$ for $\kappa_2 =1$
$\bigg( \operatorname{resp.}\, \eqref{2.64}\, \operatorname{holds}\,$ on $\overset {\circ} B_{\tilde{t_1}}(x_0)$ for $\kappa_1 =1\bigg)$. If in addition \eqref{2.59} occurs on $D(x_0)$, then the second part of \eqref{2.63}  

\noindent
$(2)\, $ Theorem \ref{T: 2.8} is equivalent to $(1)\, .$
\label{C: 2.8}
\end{corollary}
\smallskip

\noindent
\subsection{\bf Under the sectional curvature and Ricci curvature assumptions}\smallskip
\begin{corollary}
Denote $\operatorname{Ric}^M$ the Ricci curvature of $M$ and $\operatorname{Sec}^M$ the sectional curvature of $M\, .$
If  we replace $``\operatorname{Ric}^{\rm{M}}_{\rm{rad}}"$ in Theorems \ref{T: 2.5}, \ref{T: 2.6} and \ref{T: 2.7}  by $``\operatorname{Ric}^M"\, ,$ the results remain to be true. Likewise,
if  we replace $``K\, \operatorname{or}\, K(r)"$ in Theorems \ref{T: 2.8}, \ref{T: 3.1}, \ref{T: 3.2}, \ref{T: 3.3}, \ref{T: 3.4}, \ref{T: 3.5},  and Corollaries \ref{C: 2.6} - \ref{C: 2.8}, \ref{C: 3.1} - \ref{C: 3.5}  by $``\operatorname{Sec}^M"\, ,$ the results remain to be true. 
\label{C: 2.9}
\end{corollary}

\begin{proof} This follows at once from Definition 2.1, the definition of the radial curvature, and the above Theorems and Corollaries stated in Corollary \ref{C: 2.9}.
\end{proof}
\section{Curvature in Comparison Theorems}

Hessian and Laplacian Comparison Theorem \ref{T: 2.8} has many geometric applications under curvature assumptions. 
Let  $ A\, , A_1\, , B\, , B_1\, $ be constants and $K(r)$ be the
radial curvature of $M$. 
\begin{theorem}$(1)$ If \begin{equation}- \frac {A(A-1)}{r^2}\leq K(r) \quad \operatorname{where} \quad  A \ge 1\, \label{3.1}\end{equation}
on $\overset {\circ} B_{t_1}(x_0) \subset D(x_0)$, then
\begin{equation}\begin{aligned}
\operatorname{Hess} r & \leq \frac{A}{r}\bigg(
g-dr\otimes dr\bigg)\, \operatorname{in}\, \operatorname{the}\, \operatorname{sense}\, \operatorname{of}\, \operatorname{quadratic}\, \operatorname{forms}\end{aligned}\label{3.2}
\end{equation}
and
\begin{equation}\begin{aligned}
 \Delta r & \leq (n-1)\frac{A}{r}\end{aligned}\label{3.3}
\end{equation}
hold on $\overset {\circ} B_{t_1}(x_0)$.  If in addition \eqref{3.1} occurs on $D(x_0)\, ,$ then
\eqref{3.2} and \eqref{3.3} hold pointwise on $D(x_0)\, ,$ and \eqref{3.3} holds weakly on $M\, .$   

\noindent
$(2)$ Equivalently, if \begin{equation}  - \frac {A(A-1)}{(c+r)^2}\leq K(r)\, ,\quad  A \ge 1\, 
\label{3.4}\end{equation} 
on $\overset {\circ} B_{t_1}(x_0) \subset D(x_0)$, $\operatorname{where}\,  c \ge 0\, ,$ then \eqref{3.2} and \eqref{3.3} hold on $\overset {\circ} B_{t_1}(x_0)$. If in addition \eqref{3.4} occurs on $D(x_0)$, then \eqref{3.2} and \eqref{3.3} hold pointwise on $D(x_0)\, ,$ and weakly on $M\, .$   
\label{T: 3.1}
\end{theorem}

\begin{proof}[Proof of Theorem  \ref{T: 3.1}]
$ (2) \Rightarrow (1)$
$(1)$ is the special case $c = 0$ in $(2)$.
$ (1) \Rightarrow (2)$ If \eqref{3.4}  occurs, then \eqref{3.1}  takes place, since $[- \frac {A(A-1)}{(c+r)^2}, \infty)\subset [- \frac {A(A-1)}{r^2}, \infty)\, .$ Then apply $(1)\, $.
Now we prove $(2)$. First assume \eqref{3.4} occurs for $c > 0$.
Proceed as in the proof of Theorem \ref{T: 2.6}, choose
$\phi=r^{A}\, ,$ where $A \ge  1\, .$
Then $\phi^{\prime} = Ar^{A - 1}\, .$  Hence,  for $r > 0\, ,$ we have \eqref{2.42} and \eqref{2.43}.
Let $f_2$ be a positive solution of 
\eqref{2.44}, where $G_2 = \frac{-A (A - 1)}{(c+r)^2}\, , \kappa _2 = 1$ 
and 
$t_2$ be as in \eqref{2.45}.
We note $t_2 = \infty\, ,$ by a standard Sturm comparison theorem.
Furthermore, $\eqref{2.46}\, (\phi^{\prime} f_2 - f_2^{\prime} \phi )(0) = 0\, $ holds. 
In view of \eqref{2.43}, \eqref{2.44} and \eqref{2.45},  we have 
$\eqref{2.47}\, (\phi^{\prime} f_2 - f_2^{\prime} \phi )^{\prime}  
= f_2 \phi \big (\frac{A (A - 1)}{r^2}-\frac{A (A - 1)}{(c+r)^2}\big )
\ge 0\, ,$ for $r \in (0, \infty)\, .$
The monotonicity then implies that $\phi^{\prime} f_2 \ge f_2^{\prime} \phi\, $ for $r \in (0, \infty)\, $ which in turn via (\ref{2.42}) yields
\begin{equation}\frac {f_2^{\prime}}{f_2} \le \frac{\phi^{\prime}}{\phi} =  \frac{A}{r}\, \tag{2.49}\end{equation}
$\operatorname{on}\, (0, \infty)\, .$ Applying comparison Theorem \ref{T: 2.8} in which $G_1 = \frac {-A (A - 1)}{(c+r)^2} \ge  \frac {-A (A - 1)}{(c+r)^2} = G_2\, (c > 0)\, , \kappa_2 =1\, , \frac {f_2^{\prime}}{f_2} \le \frac {A}{r}\, ,$  and taking the trace, we have 
shown on $\overset {\circ} B_{t_1}(x_0)\, ,$ \eqref{2.63} holds, i.e., \eqref{3.2} and \eqref{3.3} hold  under \eqref{3.4} for every $c > 0$. Proceeding the limiting argument as in \eqref{2.50} and applying Lemma A, we prove similarly  
$(1) \Rightarrow (2)\, ,$ 
and $(1)\, .$
\end{proof}

\begin{theorem}$(1)$ If  \begin{equation} K(r) \le - \frac {A_1(A_1-1)}{r^2}\quad \operatorname{on} \quad M\backslash \{x_0\}\quad \operatorname{where} \quad   A_1\ge 1\, ,\label{3.5}\end{equation}
then $\operatorname{Hess} r $ and $\Delta r$ satisfy 
\begin{equation}
 \frac{A_1}{r}\bigg(
g-dr\otimes dr\bigg) \le\operatorname{Hess} r\quad \operatorname{and}\quad (n-1)\frac{A_1}{r} \le \Delta r \, \label{3.6}
\end{equation}
on $M \backslash \{x_0\}\, ,$
respectively. 

\noindent
$(2)$ Equivalently, $\operatorname{if}$ \begin{equation} K(r) \le - \frac {A_1(A_1-1)}{(c+r)^2}, \quad A_1 \ge 1\,  
\label{3.7}\end{equation}
$\operatorname{on} \, M\backslash \{x_0\}\, \operatorname{where} \, c \ge 0\, ,$   then \eqref{3.6} holds.
\label{T: 3.2}
\end{theorem}

\begin{proof}[Proof of Theorem \ref{T: 3.2}] $ (2) \Rightarrow (1)\, $  $(1)$ is the special case $c = 0$ in $(2)$.
$ (1) \Rightarrow (2)\, $  Enough to show $(2)$ holds for the case $c > 0$ in \eqref{3.7}.
Choose
$\phi = (c+r)^{A_1}\, ,$ where $A_1 \ge  1\, .$
Then $\phi^{\prime} = A_1 (c+r)^{A_1 - 1}\, .$  
Hence,  for $r > 0\, $ 
\begin{equation}\frac {\phi^{\prime}}{\phi} = \frac{A_1}{c+r}\label{3.8}
\end{equation}
and  $\phi^{\prime \prime} = A_1 (A_1 - 1) (c+r)^{A_1 - 2}\, .$  
That is, for $r > 0\, ,$ $\phi$ is a solution of the equation:
 \begin{equation}\phi^{\prime \prime} + \frac{-A_1 (A_1 - 1)}{(c+r)^2} \phi = 0\, .\label{3.9}
\end{equation}

Let $f_1$ be a positive solution of 
\begin{equation}
\begin{cases}
f_1^{\prime\prime} + \widetilde{G_1}f_1 = 0\quad \operatorname{on}\quad (0, \tilde{t_1}) \\
f_1(0) = 0\, , f_1^{\prime}(0) = \kappa_1
 \end{cases}
\tag{2.68}\end{equation}
where $\widetilde{G_1} = \frac {-A_1 (A_1 - 1)}{(c+r)^2}\, , \kappa _1 = 1$ and

\begin{equation}
\tilde {t}_1 =  \sup \{ r: f_1 > 0\quad \operatorname{on}\quad (0, r)\}
\label{3.10}
\end{equation}

We note $\tilde {t}_1
 = \infty\, ,$ by the same standard Sturm comparison theorem.
In view of \eqref{3.9}, \eqref{2.68} and \eqref{3.10}, for $r \in (0, \infty)\, $

\begin{equation}
\aligned
(\phi^{\prime} f_1 - f_1^{\prime} \phi )^{\prime} &  =  f_1 \phi  \big( \frac {A_1(A_1 - 1)}{(c+r)^2} - \frac {A_1(A_1 - 1)}{(c+r)^2} \big)\\
& = 0\, .
\endaligned
\label{3.11}
\end{equation}

Since $(\phi^{\prime} f_1 - f_1^{\prime} \phi) (0) = - c^{A_1} \le 0\, ,$ the monotonicity then implies that $\phi^{\prime} f_1 \le f_1^{\prime} \phi\, $ on $(0, \infty)$ which in turn via \eqref{3.8} yields

\begin{equation}\frac{A_1}{c+r} = \frac{\phi^{\prime}}{\phi} \le \frac {f_1^{\prime}}{f_1}\label{3.12}\end{equation}
$\operatorname{on}\, (0, \infty)\, .$
Applying Remark 2.1 and Theorem \ref{T: 2.8} in which $\widetilde{G_1} (r) = \frac {-A_1 (A_1 - 1)}{(c+r)^2} = \widetilde{G_2} (r) (c > 0) , \kappa_1 = 1, \frac {f_1^{\prime}}{f_1} \ge \frac{A_1}{c+r}\, $ and $\tilde{t_1} = \infty$ we have shown via \eqref{2.64} that
\begin{equation}
%\notag
 \frac{A_1}{c+r}\bigg(
g-dr\otimes dr\bigg) \le\operatorname{Hess} r\quad \operatorname{and}\quad (n-1)\frac{A_1}{c+r} \le \Delta r \tag{3.6}
\end{equation}
 $\operatorname{on}\, M\backslash \{x_0\}\, ,$ under \eqref{3.7} for every $c > 0$. Now proceed analogously to the proof of Theorem \ref{T: 2.6} :
For every $x \in M\backslash \{x_0\}$, there exist sequences $\{x_i\}$ in a radial geodesic ray in $M\backslash (x_0)$ and $\{c_i > 0\}$ such that $x_i \to x\, , c_i \to 0\, ,$ and $r(x_i) + c_i = r(x)\, .$
It follows that for every $c_i > 0\, ,$ 
\begin{equation}\label{3.13} \begin{aligned}
\quad K(r) (x_i) \le  - \frac {A_1(A_1-1)}{\big (c_i + r(x_i)\big )^2}\quad   
 &\Rightarrow\quad   \Delta r  (x_i) \geq (n-1)\frac{A_1}{c_i + r (x_i)} \, ,\\
 \operatorname{and}\, \operatorname{hence}\, ,\quad &\operatorname{as}\quad i \to \infty\, ,\\
\quad K(r) (x) \le  - \frac {A_1(A_1-1)}{\big (r(x) \big )^2}\quad   
 &\Rightarrow\quad   \Delta r  (x) \ge (n-1)\frac{A_1}{r} (x)\, , \end{aligned}\end{equation}
and similarly, $\operatorname{Hess} r (x) \ge \frac{A_1}{r}\bigg(
g-dr\otimes dr\bigg) (x)$ for every $x \in M \backslash \{x_0\}\, .
$ This 
proves $(1) \Rightarrow (2)$ and  
$(1)$.

\end{proof}

\begin{corollary} $(1)$
If the
radial curvature $K$ satisfies  
\begin{equation}
- \frac {A(A-1)}{r^2}\leq K(r)\le - \frac {A_1(A_1-1)}{r^2}\quad \operatorname{where} \quad   A \ge A_1 \ge 1\, \label{3.14}\end{equation}
$\operatorname{on} \, M\backslash \{x_0\}\, ,$ then 
\begin{equation}\begin{aligned}
 \frac{A_1}{r}\bigg(
g-dr\otimes dr\bigg) \le \operatorname{Hess} r & \leq \frac{A}{r}\bigg(
g-dr\otimes dr\bigg)\, \operatorname{in}\, \operatorname{the}\, \operatorname{sense}\, \operatorname{of}\, \operatorname{quadratic}\, \operatorname{forms},\\
(n-1)\frac{A_1}{r} \le  \Delta r & \leq (n-1)\frac{A}{r}\quad \operatorname{pointwise}\quad \operatorname{on}\, M\backslash \{x_0\} \quad \operatorname{and}\\
& \Delta r \leq (n-1)\frac{A}{r}\quad \operatorname{weakly}\, \operatorname{on}\, M.\end{aligned}\label{3.15}
\end{equation} 
\noindent
$(2)$ Equivalently, if $K$  satisfies  \begin{equation}- \frac {A(A-1)}{(c+r)^2}\leq K(r) \le - \frac {A_1(A_1-1)}{(c+r)^2}\, ,\quad A  \ge A_1 \ge 1\, 
\label{3.16}\end{equation} 
$\operatorname{on} \quad M \backslash \{x_0\}\, , \operatorname{where} \,  c \ge 0\, ,$ then \eqref{3.15} holds. 
\label{C: 3.1}
\end{corollary}

\begin{proof} This follows at once from Theorems \ref{T: 3.1} and \ref{T: 3.2}.  
\end{proof}

\begin{theoremA}$\, (1$$)$$($\cite {HLRW}$)\, $  Let the
radial curvature $K$  satisfy
\begin{equation}
-\frac {A}{r^2}\leq K(r)\leq -\frac {A_1}{r^2}\quad \operatorname{where}\quad 0 \le A_1 \le A\, \label{3.17}\end{equation}
$\operatorname{on}\, M\backslash \{x_0\}.$ Then 
\begin{equation}
\frac{1+\sqrt{1+4A_1}}{2r}\bigg(g-dr\otimes dr\bigg) \le  \operatorname{Hess} (r) \le \frac{1+\sqrt{1+4A}}{2r}\bigg(
g-dr\otimes dr\bigg)\quad \operatorname{on}\quad M\backslash \{x_0\}\, ,\label{3.18}\end{equation}
\begin{equation}
(n-1)\frac{1+\sqrt{1+4A_1}}{2r} \le  \Delta r  \le (n-1)\frac{1+\sqrt{1+4A}}{2r}\, \operatorname{pointwise}\, \operatorname{on}\, M\backslash \{x_0\}\,  ,\operatorname{and}\,\label{3.19}\end{equation}

\begin{equation}\Delta r \le (n-1)\frac{1+\sqrt{1+4A}}{2r} \quad \operatorname{weakly}\, \operatorname{on}\, M.\label{3.20}\end{equation}

\noindent
$(2)$ Equivalently, if 
\begin{equation}
-\frac {A}{(c+r)^2}\leq K(r)\leq -\frac {A_1}{(c+r)^2}\quad \operatorname{where}\quad 0 \le A_1 \le A\, \label{3.21}\end{equation}
$\operatorname{on}\, M\backslash \{x_0\},$ then \eqref{3.18}-\eqref{3.20} hold.  
\end{theoremA}

\begin{proof}[{\bf Proof of $\operatorname{Theorem} \operatorname{A}$}]
Let  $A
(A-1) = a^2$ in \eqref{3.1}$\big ($resp. $A_1
(A_1-1) = a_1^2\, \text{in}\, \eqref{3.5} \big )$. Then $A = \frac{1+ \sqrt{1+4a^2}}{2}\, $ (resp. $A_1 = \frac{1+ \sqrt{1+4a_1^2}}{2}\, ) \ge 1$. 
Hence, $\frac {A}{r} = \frac{1 + \sqrt{1+4a^2}}{2r}$ (resp. $\frac {A_1}{r} = \frac{1 + \sqrt{1+4a_1^2}}{2r}$). 
Substitute these into Theorem \ref{T: 3.1} (resp. Theorem \ref{T: 3.2}) so that this Theorem is rephrased in terms of $a^2$ (resp. $a_1^2$). Replacing $a^2$ (resp. $a_1^2$) in this rephrase by $A$ (resp. $A_1$), we transform Corollary \ref{C: 3.1} into
Theorem A.\end{proof}

\begin{corollary} $(1)$
If the
radial curvature $K$ satisfies \begin{equation}K(r) = - \frac {A(A-1)}{r^2}\, , \operatorname{where}\,  1 \le A\, \big (\operatorname{resp.}\, K(r) = -\frac {A}{r^2}\, ,\operatorname{where} \,  0 \le A\big )\,  \label{3.22}\end{equation} 
$\operatorname{on}\, M\backslash \{x_0\}\, ,$ then
\begin{equation}\begin{aligned}
\operatorname{Hess} r &= \frac{A}{r}\, \big (\operatorname{resp.}\, \operatorname{Hess} r = \frac{1+\sqrt{1+4A}}{2r} (
g-dr\otimes dr) \big)\quad \operatorname{and}\\
\quad \Delta r & = (n-1)\frac{A}{r}\, \big (\operatorname{resp.}\,  \Delta r  = (n-1)\frac{1+\sqrt{1+4A}}{2r}\big )\, \label{3.23} \end{aligned} \end{equation} $\operatorname{ on}\, M \backslash \{x_0\}\, .$

\noindent
$(2)$ Equivalently, if $K$  satisfies  \begin{equation}K(r) = - \frac {A(A-1)}{(c+r)^2}\, , \operatorname{where}\,  1 \le A\, \big (\operatorname{resp.}\, K(r) = -\frac {A}{(c+r)^2}\, ,\operatorname{where} \,  0 \le A\big )\,  \label{3.24}\end{equation} 
on $M\backslash \{x_0\},\, \operatorname{and}\, c \ge 0\, ,$ then \eqref{3.23} holds on $M \backslash \{x_0\}.$
\label{C: 3.2}
\end{corollary}
\begin{proof} This follows at once from combining Theorems \ref{T: 3.1} and \ref{T: 3.2} in which $A_1 = A$.  
\end{proof}

\begin{theorem}\label{T: 3.3}
$(1)$ If \begin{equation}\frac {B_1(1-B_1)}{r^2}\leq K(r)\quad \operatorname{where} \quad  0 \le B_1 \le 1\, \label{3.25}\end{equation}
on $\overset {\circ} B_{t_1}(x_0) \subset D(x_0)$, then
$\operatorname{Hess} r $ and $\Delta r$ satisfy 
 \begin{equation} \begin{aligned}
 \operatorname{Hess} r & \le \frac{1+\sqrt{1+4B_1(1-B_1)}}{2r}\bigg(
g-dr\otimes dr\bigg)\quad \operatorname{and}\end{aligned}\label{3.26}\end{equation}
 \begin{equation} \begin{aligned}
\Delta r & \le (n-1) \frac{1+\sqrt{1+4B_1(1-B_1)}}{2r}\end{aligned}\label{3.27}\end{equation}
on $\overset {\circ} B_{t_1}(x_0) \subset D(x_0)$
respectively. If in addition \eqref{3.25} occurs on $D(x_0)\, ,$ then \eqref{3.26} holds pointwise on $D(x_0)\, $ and \eqref{3.27} holds weakly on $M$.

\noindent
$(2)$ Equivalently, if \begin{equation} \frac {B_1(1-B_1)}{(c+r)^2}\leq K(r),\quad 0 \le B_1 \le 1 
\label{3.28}\end{equation} on $\overset {\circ} B_{t_1}(x_0) \subset D(x_0)$,  $\operatorname{where}\, c \ge  0\, ,$  then \eqref{3.26} and \eqref{3.27} hold on $\overset {\circ} B_{t_1}(x_0)$. If in addition \eqref{3.28} occurs on $D(x_0)$, then \eqref{3.26} holds pointwise on $D(x_0)\, $ and 
\eqref{3.27} holds weakly on $M$.
\end{theorem}

\begin{proof}
  $ (2) \Rightarrow (1)\, $
$(1)$ is the special case $c = 0$ in $(2)$. To prove $(1) \Rightarrow (2)$, enough to show that on $\overset {\circ} B_{t_1}(x_0)\, ,$ \eqref{3.26} and \eqref{3.27} hold under \eqref{3.28}, on $\overset {\circ} B_{t_1}(x_0)$ for any $c > 0$, since we can combine this with $(1)\, .$  
Proceed as in the proof of Theorem \ref{T: 2.7}, choose
$\phi_1 = r^{\alpha}\, ,$ where $\alpha = \frac{1+\sqrt{1+4B_1(1-B_1)}}{2}\, , 0 \le B_1 \le 1\, .$
Then $\alpha \ge 1\, ,$ $\phi_1 ^{\prime} = \alpha r^{\alpha - 1}\, $ and $(2\alpha - 1)^2 = 1+4B_1(1-B_1)\, , $ i.e. $\alpha (\alpha - 1) =  B_1(1-B_1)\, .$    
Hence, for $r > 0\, ,$ we have $\frac{{\phi_1}^{\prime}}{\phi_1} = \frac {\alpha}{r}\, ,$ and $\phi_1^{\prime \prime} = \alpha (\alpha - 1) r^{\alpha - 2}\, ,$  i.e. \eqref{2.54}  and \eqref{2.55} hold.
Let $f_2$ be a positive solution of \eqref{2.44}, in which $G_2 = - \frac{B_1(1-B_1)}{(c+r)^2}\, , \kappa _2 = 1\, $
and 
$t_2$ be as in \eqref{2.45}.
Then by the standard comparison theorem as before, $t_2=\infty.$ 

Furthermore, we have \eqref{2.46}, \eqref{2.47} and  \eqref{2.48} hold on $\overset {\circ} B_{t_1}(x_0)$ in which $\phi = \phi_1$, since $f_2 \phi_1 \big(  \frac {B_1(1-B_1)}{r^2} + G_2 \big)\ge 0\, ,$ for $r \in (0, \infty)\, .$ 
It follows from \eqref{2.54} that \eqref{2.56}, i.e., 
$\frac {f_2^{\prime}}{f_2} \le \frac{\phi_1^{\prime}}{\phi_1} =  \frac{1+\sqrt{1+4B_1(1-B_1)}}{2r}$ holds on $(0, \infty)\, .$   

Applying comparison Theorem \ref{T: 2.8}, in which $G_1 = \frac {B_1(1-B_1)}{(c+r)^2} \ge - \frac {B_1(1-B_1)}{(c+r)^2} = G_2\, (c > 0)\, , \kappa_2 =1\, , \frac {f_2^{\prime}}{f_2} \le \frac {1+\sqrt{1+4B_1(1-B_1)}}{2r}\, ,$   we have 
shown that on $\overset {\circ} B_{t_1}(x_0)\, ,$  \eqref{2.63} holds, i.e., \eqref{3.26} and \eqref{3.27} hold under \eqref{3.28} for any $c > 0$. This proves $(1) \Rightarrow (2)$. This also proves $(1)$, since \eqref{3.25} occurs on $\overset {\circ} B_{t_1}(x_0)$ implies that \eqref{3.28} happens 
on $\overset {\circ} B_{t_1}(x_0)$ for any $c >0$.
If in addition \eqref{3.25} occurs on $D(x_0)$, then  \eqref{3.26} holds pointwise on $D(x_0)\, $ and \eqref{3.27} holds weakly on $M$ by Lemma A (\cite [Lemma 9.1]{HLRW}), or a double limiting argument (see \cite [Sect. 3]{WL}). 
\end{proof}

\begin{remark1} In proving Theorem \ref{T: 3.3}, we choose axillary function
$\phi_1 = r^{\alpha}\, ,$ where $\alpha = \frac{1+\sqrt{1+4B_1(1-B_1)}}{2}\, .$ We may choose, for example a different axillary function 
$\phi_1 = r^{{\alpha_1} +1}\, ,$ where ${\alpha_1} = \frac{1+\sqrt{1-4B_1(1-B_1)}}{2}\, ,$ and obtain estimates 
 \begin{equation} \begin{aligned}
 \operatorname{Hess} r & \le \frac{|B_1 - \frac 12| + \frac 32}{r}\bigg(
g-dr\otimes dr\bigg)\quad \operatorname{and}\\
\Delta r \, \, & \le (n-1) \frac{|B_1 - \frac 12| + \frac 32}{r}
\quad \operatorname{pointwise}\quad \operatorname{on}\quad \overset {\circ} B_{t_1}(x_0) \subset D(x_0)\end{aligned}\label{3.29}\end{equation}
If in addition \eqref{3.25} occurs on $D(x_0)\, ,$ then
\begin{equation}
\Delta r \le (n-1) \frac{|B_1 - \frac 12| + \frac 32}{r}
\quad \operatorname{weakly}\quad \operatorname{on}\quad M.
\label{3.30}
\end{equation}
However, the upper bound estimate $\frac{1+\sqrt{1+4B_1(1-B_1)}}{2r}\, $ obtained in \eqref{3.26} and \eqref{3.27} by selecting $\phi_1 = r^{\alpha}\, $ is better than the upper bound estimate $\frac{|B_1 - \frac 12| + \frac 32}{r}$ obtained in \eqref{3.29} and \eqref{3.30} by selecting $\phi_1 = r^{{\alpha_1}+1}\, .$ Indeed, $\frac{1+\sqrt{1+4B_1(1-B_1)}}{2r} \le \frac{B_1+1}{r} = \frac{B_1 - \frac 12 + \frac 32}{r} \le \frac{|B_1 - \frac 12| + \frac 32}{r}\, .$
For the completeness or comparison, we include the following derivation for the estimates \eqref{3.29} and \eqref{3.30}.

Note $(2{\alpha_1} - 1) = |2B_1 -1|$.
Analogously,  $\phi_1 = r^{{\alpha_1} +1}$ implies $\phi_1^{\prime} = ({\alpha_1} +1) r^{{\alpha_1}}\, .$ For $r > 0\, ,$ 
\begin{equation}\frac {\phi_1^{\prime}}{\phi_1} = \frac{{\alpha_1}+1}{r}= \frac{|B_1 - \frac 12| + \frac 32}{r}\label{3.31}
\end{equation} and
$\phi_1^{\prime \prime} =({\alpha_1} + 1) {\alpha_1}  r^{B_1 - 1}\, ,$  i.e.
\begin{equation}\phi_1^{\prime \prime} + \frac{-({\alpha_1}+1){\alpha_1}}{r^2} \phi_1 = 0\, .\label{3.32}
\end{equation}
Let $f_2$ be a positive solution of 
\eqref{2.44}, where $G_2 = \frac {-({\alpha_1}+1){\alpha_1}}{(c+r)^2}\, , \kappa_2 = 1$
and let 
$t_2$ be as in \eqref{2.45}.
Then $t_2=\infty,$
$ (\phi_1^{\prime} f_2 - f_2^{\prime} \phi_1 ) (0) = 0 
\, ,$ and for $r \in (0, \infty)\, $ 
$$(\phi_1^{\prime} f_2 - f_2^{\prime} \phi_1 )^{\prime} =  f_2 \phi_1 \big(  \frac {{\alpha_1}({\alpha_1} +1)}{r^2} + G_2 \big) \ge 0\, .$$
Monotonicity implies
$$\frac {f_2^{\prime}}{f_2} \le \frac{\phi_1^{\prime}}{\phi_1} =  \frac{|B_1 - \frac 12| + \frac 32}{r}\quad \operatorname{on}\quad (0, \infty)\, $$

Applying Theorem \ref{T: 2.8}, in which $G_1 = \frac {B_1(1-B_1)}{(c+r)^2} \ge 0 \ge \frac {-({\alpha_1}+1){\alpha_1}}{(c+r)^2} = G_2\, , \kappa_2 =1\, , \frac {f_2^{\prime}}{f_2} = \frac{|B_1 - \frac 12| + \frac 32}{r}\, ,$
and taking the trace, we obtain the estimates \eqref{3.29}. Applying Lemma A, if in addition, \eqref{3.25} occurs on $D(x_0)$, \eqref{3.30} follows .
\end{remark1}

\begin{corollary} 
If the
radial curvature $K$  satisfies
\begin{equation}
\frac {B_1}{r^2} \le K(r)\,   \big ( \operatorname{resp.}\,  \frac {B_1}{(c+r)^2}\leq  K(r)\, , c \ge 0\big )\, ,\, \operatorname{where}\, 0 \le B_1 \le \frac 14\,  \label{3.33}\end{equation}
$\operatorname{on}\, \overset {\circ} B_{t_1}(x_0) \subset D(x_0)\, ,$ then 
\begin{equation}
\operatorname{Hess} r  \le \frac{1+\sqrt{1+4B_1}}{2r}\bigg(g-dr\otimes dr\bigg)\quad \operatorname{and}\quad \Delta r \le (n-1)\frac{1+\sqrt{1+4B_1}}{2r} \label{3.34}\end{equation}
holds pointwise on $\overset {\circ} B_{t_1}(x_0)\, .$ If in addition, \eqref{3.33} holds in $D(x_0)$, then the second part of \eqref{3.34} holds weakly on $M\, .$\label{C: 3.3}
\end{corollary}

\begin{proof}
This follows immediately from substituting $B_1(1 - B_1)$ in Theorem \ref{T: 3.3} for  $B_1$, and $B_1(1 - B_1) = - (B_1 - \frac 12)^2 + \frac 14 \le \frac 14\, ,$  $B_1(1 - B_1) \ge 0$ if $0 \le B_1 \le 1\, .$ 
\end{proof}

\begin{theorem}
$(1)$ If  \begin{equation} K(r) \le \frac {B(1-B)}{r^2}\, , \quad  0 \le B \le 1 \label{3.35}\end{equation}
on $\overset {\circ} B_{t_2}(x_0) \subset D(x_0)$, then $\operatorname{Hess} r $ and $\Delta r$ satisfy 
 \begin{equation} \begin{aligned}
\frac{|B - \frac 12| + \frac 12}{r} \bigg(g-dr\otimes dr\bigg) \le & \operatorname{Hess} r \quad \operatorname{and}\\
(n-1)\frac{|B - \frac 12| + \frac 12}{r} \le & \, \, \Delta r \\
 \end{aligned}
\label{3.36}
\end{equation}
on $\overset {\circ} B_{t_2}(x_0)$ respectively. 

\noindent
$(2)$ Equivalently, if \begin{equation} K(r) \le \frac {B(1-B)}{(c+r)^2}\, , \quad   0 \le B \le 1 
\label{3.37}\end{equation}
on $\overset {\circ} B_{t_2}(x_0) \subset D(x_0)\, , \operatorname{where} \, c \ge 0\, ,$  then \eqref{3.36} holds on $\overset {\circ} B_{t_2}(x_0)\, .$
\label{T: 3.4}
\end{theorem}

\begin{proof}
 $(2)$ $\Rightarrow$ $(1)\, $  $(1)$ is the special case $c = 0$ in $(2)$.
 $(1)$ $\Rightarrow$ $(2)\, $ If \eqref{3.37}  occurs, then \eqref{3.35}  occurs, since $(- \infty, \frac {B(1-B)}{(c+r)^2}]\subset (- \infty, \frac {B(1-B)}{r^2}]\, .$ We then apply $(1)\, .$
To prove $(2)$, we first assume \eqref{3.37} occurs for $c > 0$.
Choose $\beta = \frac{1+\sqrt{1-4B(1-B)}}{2}\, ,$ with $0 \le B \le  1\, ,$ and $\phi_2 = r^{\beta}\, .$
Then $2\beta - 1 = \sqrt {(2B- 1)^2}\, , $ i.e.  \begin{equation}\begin{aligned}\beta = |B - \frac 12| + \frac 12=\max \{B, 1-B\} = \begin{cases}B\quad &\text{if}\quad \frac 12 \le B\\
1-B\quad & \text{if}\quad 0 \le B < \frac 12\, ,\end{cases}\end{aligned}\label{3.38}\end{equation} $\beta (\beta - 1) = B(B-1)$ and $\phi_2^{\prime} = \beta r^{\beta - 1}\, $ for $r > 0\, .$ Hence,
\begin{equation}\frac {\phi_2^{\prime}}{\phi_2} = \frac{\beta}{r}\label{3.39}\end{equation}
for $r > 0\, ,$ and for $0 \le B \le 1\, ,$ 
 $\phi_2^{\prime \prime} = - B (1-B) r^{\beta - 2}\, ,$  i.e.
 \begin{equation}\phi_2^{\prime \prime} + \frac{B(1-B)}{r^2} \phi_2 = 0\, .\label{3.40}\end{equation} 
Let $f_1$ be a positive solution of \eqref{2.68},
where $\widetilde {G_1} =  \frac {B(1-B)}{(c+r)^2}\, , \kappa _1 = 1$
and let 
$\tilde{t_1}$ be as in \eqref{3.10}. Then 
$$ \aligned
(\phi_2^{\prime} f_1 - f_1^{\prime} \phi_2 ) (0) & = 0\quad {if}\quad \beta = 1\\
\underset{r \to 0^+}{\lim} (\phi_2^{\prime} f_1 - f_1^{\prime} \phi_2 ) (r) & = \underset{r \to 0^+}{\lim}  \frac {\beta}{1-\beta}f_1^{\prime}(r)\, r^{\beta} = 0\quad {if}\quad \beta \ne 1 .
\endaligned
$$
by l'Hospital's Rule.  Moreover, in view of  \eqref{3.40}, \eqref{2.68} and \eqref{3.10}, for $r \in (0, \tilde{t_1})\, $
$$
\aligned
(\phi_2^{\prime} f_1 - f_1^{\prime} \phi_2 )^{\prime} & = \phi_2 f_1 \big(\frac {-B(1-B)}{r^2} + \frac {B(1-B)}{(c+r)^2} \big)\\
&  \le  0
\endaligned
$$
Monotonicity implies that 
\begin{equation}
 \frac{\beta}{r} = \frac{\phi_2^{\prime}}{\phi_2} \le \frac {f_1^{\prime}}{f_1}\quad \operatorname{on}\quad (0, \tilde{t_1})\, .\label{3.41}\end{equation}
 
  Integrating \eqref{3.41} over $[\epsilon_1 \, , \tilde{t_1}-\epsilon_2] \subset (0, \tilde{t_1})\, ,$ and passing $\epsilon_2 \to 0$,  we have

\begin{equation}\notag
0 <  C(\epsilon_1) r^{\beta} \le f_1\quad \operatorname{on}\quad [\epsilon_1, \tilde{t_1}],\end{equation}
where $C(\epsilon_1) > 0\, $ is a constant depending on $\epsilon_1\, .$
Thus $f_1 > 0\, $ on $(0, \tilde{t_1}]\, .$ We claim $\tilde{t_1} = \infty\, .$ Otherwise there would exist a $\delta >0$ such that $\tilde{t_1} + \delta < \infty$ at which $f_1 > 0$ by the continuity, and would lead to, via \eqref{3.10} a contradiction
$${\tilde{t_1}} < {\tilde{t_1}} + \delta \le {\tilde{t_1}}\, .$$

Applying Theorem \ref{T: 2.8}, in which $\widetilde{G_1} = \frac {B(1-B)}{(c+r)^2} \ge \frac {B(1-B)}{(c+r)^2} = \widetilde{G_2}\, ,(c > 0)\, , \kappa_1 =1\, , \frac {f_1^{\prime}}{f_1} \ge \frac {\beta}{r}\, , \tilde{t_1} = t_1, \tilde{t_2} = t_2$, we have 
shown  
on $\overset {\circ} B_{t_2}(x_0)\, ,$ \eqref{3.36} holds under \eqref{3.37} for every $c > 0$.  
To prove $(1)\, ,$ we proceed analogously to the proof of Theorem \ref{T: 3.2} :
For every $x \in \overset {\circ} B_{t_2}(x_0)$, there exist sequences $\{x_i\}$  in a radial geodesic ray in $\overset {\circ} B_{t_1}(x_0)$ and $\{c_i > 0\}$ such that $x_i \to x\, , c_i \to 0\, ,$ and $r(x_i) + c_i = r(x)\, .$
It follows that for every $c_i > 0\, ,$ 
\begin{equation}\label{3.42} \begin{aligned}
&\quad K(r) (x_i) \le  \frac {B(1-B)}{\big (c_i + r(x_i)\big )^2}\quad  
 \Rightarrow \quad \Delta r (x_i) \ge (n-1)\frac{|B - \frac 12| + \frac 12}{c_i + r (x_i)} \, ,\\
 & \operatorname{and}\, \operatorname{consequently}\, ,\quad \operatorname{as}\quad i \to \infty\, ,\\
&\quad K(r) (x) \le  \frac {B(1-B)}{r^2} (x)\quad 
 \Rightarrow\quad  \Delta r (x) \ge (n-1)\frac{|B - \frac 12| + \frac 12}{r(x)}  \, ,\end{aligned}\end{equation}
 and \eqref{3.36} holds on $\overset {\circ} B_{t_2}(x_0)\, .$

\end{proof}

\begin{corollary}
If the
radial curvature $K$ satisfies 
\begin{equation} 
K(r) \le \frac {B}{r^2}\, \big (\operatorname{resp.}\,  K(r) \le \frac {B}{(c+r)^2}\, , c \ge 0\big )\, ,\, \operatorname{where}\, 0 \le B \le \frac 14\, \label{3.43}\end{equation}
$\operatorname{on} \, \overset {\circ} B_{t_2}(x_0) \subset D(x_0)$, then
\begin{equation} \begin{aligned}
& \frac{1+\sqrt{1-4B}}{2r} \le\,  \operatorname{Hess} r   \bigg(
g-dr\otimes dr\bigg)\quad
\operatorname{and}\\
 & (n-1)  \frac{1+\sqrt{1-4B}}{2r} \le\, \, \Delta r 
\end{aligned}\label{3.44}
\end{equation}\label{C: 3.4}
hold $\operatorname{pointwise}\, \operatorname{on}\, \overset {\circ} B_{t_2}(x_0)\, .$ \end{corollary}

\begin{proof}
Let  $B
(1- B) = b^2$ in \eqref{3.35}. Then $B = \frac{1\pm \sqrt{1-4b^2}}{2}\, ,$ and by completing the square, we have $$b^2 = B
(1- B) = -(B - \frac 12)^2 + \frac 14 \le \frac 14\, .$$ 
Hence, $|B - \frac 12| + \frac 12 = \frac{1 + \sqrt{1-4b^2}}{2}$. 
Substitute these into Theorem \ref{T: 3.4}, so that this Theorem is rephrased in terms of $b^2$. We then replace $b^2$ in this rephrased Theorem by $B$ and obtain the desired. 
\end{proof}

\begin{theorem}
If the
radial curvature $K$   satisfies 
\begin{equation} 
\frac {B_1}{r^2} \le K(r) \le \frac {B}{r^2} \quad \operatorname{where} \quad  0 \le B_1 \le B \le \frac 14\, \label{3.45}\end{equation}
on $\overset {\circ} B_{\tau}(x_0) \subset D(x_0)\, ,$ then
 \begin{equation} \begin{aligned}
& \frac{1+\sqrt{1-4B}}{2r} \bigg(
g-dr\otimes dr\bigg) \le  \operatorname{Hess} r  \le \frac{1+\sqrt{1+4B_1}}{2r}\bigg(g-dr\otimes dr\bigg)\quad \operatorname{and} \\
&  (n-1)  \frac{1+\sqrt{1-4B}}{2r} \le  \,  \, \Delta r \le (n-1)\frac{1+\sqrt{1+4B_1}}{2r}
\end{aligned}\label{3.46}
\end{equation}
hold pointwise on $\overset {\circ} B_{\tau}(x_0)$.
If in addition \eqref{3.33} occurs on $D(x_0)\, ,$ then 
\begin{equation}
\Delta r \leq  (n-1) \frac{1+\sqrt{1+4B_1}}{2r}\label{3.47}\end{equation}
holds weakly $\operatorname{on}\, M\, .$

\noindent
$(2)$ Equivalently, if $K$ satisfies  
\begin{equation}
\frac {B_1}{(c+r)^2}\leq K(r) \le \frac {B}{(c+r)^2}\, ,\quad 0 \le B_1 \le B
 \le \frac 14\,
\label{3.48}\end{equation}
on $\overset {\circ} B_{\tau}(x_0) \subset D(x_0)\, , \operatorname{where} \,  c \ge 0\, ,$ 
then \eqref{3.46} holds on $\overset {\circ} B_{\tau}(x_0)$. If in addition, \eqref{3.33} occurs on $D(x_0)\, ,$ then \eqref{3.47} holds weakly $\operatorname{on}\, M\,.$
 \label{T: 3.5} \end{theorem}
\begin{proof}
This follows at once from Corollaries \ref{C: 3.3} and \ref{C: 3.4}.
\end{proof}

\begin{corollary} $(1)$
Let the
radial curvature $K$   satisfy 
 \begin{equation} \frac {B_1(1-B_1)}{r^2}\leq K(r)\le \frac {B(1-B)}{r^2}\, , \quad  0 \le B \, , B_1 \le 1\, \label{3.49}\end{equation} 
$\operatorname{on} \, \overset {\circ} B_{\tau}(x_0) \subset D(x_0)\, .$ Then 
\begin{equation}\begin{aligned}
& \frac {|B - \frac 12| + \frac 12}{r} \bigg(g-dr\otimes dr\bigg)    \le \operatorname{Hess} r  \le \frac{1+\sqrt{1+4B_1(1-B_1)}}{2r}\bigg(
g-dr\otimes dr\bigg)\quad 
\operatorname{and} \\
& (n-1)\frac{|B - \frac 12| + \frac 12}{r} \le  \Delta r  \le (n-1) \frac{1+\sqrt{1+4B_1(1-B_1)}}{2r}\end{aligned}\label{3.50}
\end{equation} 
hold $\operatorname{pointwise}\, \operatorname{on}\, \overset {\circ} B_{\tau}(x_0)$.
If in addition, \eqref{3.25} occurs on $D(x_0)\, ,$ then
\begin{equation}
\Delta r \leq  (n-1) \frac{1+\sqrt{1+4B_1(1-B_1)}}{2r}\tag{3.27}
\end{equation}
holds $\operatorname{weakly}\, \operatorname{on}\, M.$

\noindent
$(2)$ Equivalently, if $K$  satisfies  \begin{equation}  \frac {B_1(1-B_1)}{(c+r)^2}\leq K(r) \le  \frac {B(1-B)}{(c+r)^2}\, ,\quad  0 \le B\, , B_1 \le 1\, 
\label{3.51}\end{equation} on $\overset {\circ} B_{\tau}(x_0),$ where $c \ge 0\, ,$ then \eqref{3.50} holds on $\overset {\circ} B_{\tau}(x_0)$. If in addition, \eqref{3.28} occurs on $D(x_0)\, ,$ then \eqref{3.27} holds $\operatorname{weakly}\, \operatorname{on}\, M\, .$  
\label{C: 3.5}
\end{corollary}

\begin{corollary} 
If the
radial curvature $K$  satisfies 
 \begin{equation} K(r) = 0 \, (\operatorname{resp.}\, \le 0\, ,\, \ge 0 )\label{3.52}\end{equation} 
$\operatorname{on} \, M\backslash \{x_0\}\, ,$ then 
\begin{equation}\begin{aligned}
 \operatorname{Hess} r &  = \, (\operatorname{resp.}\,  \operatorname{Hess} r \ge \, , \operatorname{Hess} r \le \,   )\,  \frac{1}{r}\bigg(
g-dr\otimes dr\bigg)\quad 
\operatorname{and} \\
\Delta r &  = \, (\operatorname{resp.}\, \Delta r  \ge \, ,\, \Delta r  \le \,  )\,  (n-1) \frac{1}{r}\end{aligned}\label{3.53}
\end{equation} $\operatorname{on} \, M\backslash \{x_0\}\, .
$ 
\label{C: 3.6}
\end{corollary}

\begin{proof}
This follows at once from Corollary \ref{C: 3.5} in which $B_1 = B = 1\,,$ or Theorem \ref{T: 3.5}, in which $B_1 = B = 0\, ,$ or Corollary \ref{C: 3.2} in which $A = 1 \, (\operatorname{resp.} A =0). $ 
\end{proof}

\begin{corollary} 
If the
radial curvature $K$   satisfies 
 \begin{equation} \begin{aligned}-\frac {A}{r^2}\leq & K(r)\leq \frac {B}{r^2}\\
  (\operatorname{resp.}\, -\frac {A}{(c+r)^2}\leq & K(r)\leq \frac {B}{(c+r)^2})\, , \quad 0 \le A\, , \, 0 \le B \le \frac 14 \end{aligned}\label{3.54}\end{equation}
$\operatorname{on}\, M\backslash \{x_0\},$ where $c > 0\, ,$ then 
\begin{equation}
\frac{1+\sqrt{1-4B}}{2r}\bigg(g-dr\otimes dr\bigg) \le  \operatorname{Hess} (r) \le \frac{1+\sqrt{1+4A}}{2r}\bigg(
g-dr\otimes dr\bigg)\quad \operatorname{and}\tag{3.55}\label{3.55}\end{equation}
\begin{equation}
(n-1)\frac{1+\sqrt{1-4B}}{2r} \le  \Delta r  \le (n-1)\frac{1+\sqrt{1+4A}}{2r}\, \label{3.56}\end{equation} 
hold on $M\backslash \{x_0\}.$
\label{C: 3.7}
\end{corollary}

\begin{proof}
This follows at once from Theorems A and Corollary \ref{C: 3.4}. 
\end{proof}

\section{Geometric Applications in Mean Curvature}
The following is an immediate geometric application:

\begin{theorem}$($Mean Curvature Comparison Theorems$)$\label{T:3.5} Let $H(r)$ be the mean curvature of the geodesic sphere $\partial B_{r}(x_0)$ of radius $r$ centered at $x_0$ in $M$ with respect to the unit outward normal. Then

\begin{equation}\label{4.1}
H(r) \le \left\{
\begin{aligned}
\frac{A}{r}, \quad & \text {if } \operatorname{Ric}^{\rm{M}}_{\rm{rad}}(r)  \text { satisfies}\, \eqref{2.39}\, \text{or}\, \eqref{2.41}\,  , \\
\frac{1+\sqrt{1+4B_1(1-B_1)}}{2r},  \quad & \text {if } \operatorname{Ric}^{\rm{M}}_{\rm{rad}}(r)  \text { satisfies}\, \eqref{2.51}\, or\, \eqref{2.53}\, ;
\end{aligned}
\right.
\end{equation}

\begin{equation}\label{4.2}
H(r) \le \left\{
\begin{aligned}
\frac{A}{r}, \quad & \text {if } K(r)  \text { satisfies}\, \eqref{3.1}\, or\, \eqref{3.4}\, , \\
\frac{1+\sqrt{1+4B_1(1-B_1)}}{2r},  \quad & \text {if } K(r)  \text { satisfies} \, \eqref{3.25}\, or\, \eqref{3.28}\, , \\
\frac{1+\sqrt{1+4B_1}}{2r},  \quad & \text {if } K(r)  \text { satisfies} \, \eqref{3.33}\, ;
\end{aligned}
\right.
\end{equation}

\begin{equation}\label{4.3}
H(r) \ge \left\{
\begin{aligned}
\frac{A_1}{r}, \quad   & \text {if } K(r)  \text { satisfies}\, \eqref{3.5}\, or\, \eqref{3.7}\, , \\
\frac{|B - \frac 12| + \frac 12}{r}, \quad & \text{if } K(r) \text { satisfies}\, \eqref{3.35}\, or\, \eqref{3.37} \, , \\
\frac{1+\sqrt{1-4B}}{2r}, \quad & \text{if } K(r) \text { satisfies}\, \eqref{3.43}\, ;
\end{aligned}
\right.
\end{equation}
  
\begin{equation}\label{4.4}
\left\{
\begin{aligned}
\frac{A_1}{r}   & \le  H(r)  \le \frac{A}{r},\quad   & \text {if }  K(r)  \text {satisfies}\, \eqref{3.14}\, or\, \eqref{3.16}\, ,   \\
\frac{1+\sqrt{1+4A_1}}{2r}, & \le   H(r)  \le (n-1)\frac{1+\sqrt{1+4A}}{2r}, \quad  & \text{if } K(r) \text{ satisfies} \eqref{3.17}\, or\, \eqref{3.21}\, ,\\
\frac{1+\sqrt{1-4B}}{2r}, & \le  H(r)  \le \frac{1+\sqrt{1+4B_1}}{2r},\, &  \text{if } K(r) \text{ satisfies}\, \eqref{3.45}\, or\, \eqref{3.48}\, , \\
\frac{|B - \frac 12| + \frac 12}{r} & \le  H(r)  \le \frac{1+\sqrt{1+4B_1(1-B_1)}}{2r},\quad  & \text{if } K(r) \text{ satisfies}\, \eqref{3.49}\, or\, \eqref{3.51}\, ,\\
\frac{1+\sqrt{1-4B}}{2r} & \le  H(r)  \le \frac{1+\sqrt{1+4A}}{2r},\quad  & \text{if } K(r) \text{ satisfies}\, \eqref{3.54}\, ;
\end{aligned}
\right.
\end{equation}

\begin{equation}\label{4.5}
H(r) \left\{
\begin{aligned}
& = \frac{A}{r}, \, \big (\operatorname{resp.}\, \frac{1+\sqrt{1+4A}}{2r}\big ), &  \text{if } K(r) \text{ satisfies}\, \eqref{3.22}\, or\, \eqref{3.24}\, , \\
& =\,  (\operatorname{resp.}\, \ge \, ,\,  \le \,  )\,  \frac{1}{r},&  \text{if } K(r) \text{ satisfies}\, \eqref{3.52}\, ;\end{aligned}
\right.
\end{equation}
where the corresponding geodesic spheres $\partial B_{r}(x_0)$ are as in Theorems \ref{T: 3.1}-\ref{T: 3.5}, Theorem A, and Corollaries \ref{C: 3.1}-\ref{C: 3.7}.
\label{T: 4.1}
\end{theorem}

\begin{proof} By Gauss lemma
\begin{equation}
\begin{aligned}
\frac {1}{n-1} \Delta r & : = \frac {1}{n-1} \text{trace} \big( \text{Hess}\, r \big ) \\
& \, = \frac {1}{n-1} \big (\text{the}\, \text{trace}\, \text{of}\, \text{the}\, \text{second}\, \text{fundamental}\, \text{form}\, \text{of}\, \partial B_{r}(x_0)\, \text{in}\, M\big )\\
& : = H(r)
\end{aligned}
\end{equation}
(see \cite [(3.28)] {W3}). The results follow from Theorems \ref{T: 2.6}-\ref{T: 2.7}, \ref{T: 3.1}-\ref{T: 3.5}, Theorem A, and Corollaries \ref{C: 3.1}-\ref{C: 3.7}.
\end{proof}

\section{The Growth of Bundle-Valued Differential Forms and Their Interrelationship}

Let $(M,g)$ be a smooth Riemannian manifold. Let $\xi :E\rightarrow M$ be a smooth Riemannian vector bundle over $(M,g)\, ,$ i.e. a vector bundle such that at each fiber is equipped with a positive inner product $\langle \quad , \quad \rangle_E\, .$
 Set $A^k(\xi )=\Gamma (\Lambda
^k T^{*}M\otimes E)$ the space of smooth $k-$forms on $M$ with
values in the vector bundle $\xi :E\rightarrow M$.  For two forms $\Omega ,\Omega ^{\prime }\in A^k(\xi )$, the
induced inner product $ \langle\Omega ,\Omega ^{\prime }\rangle$ is defined as in \eqref{6.4}. For $\Omega \in A^k(\xi )$, set  $|\Omega|^2 = \langle \Omega, \Omega \rangle$. Then $|\Omega|^q = \langle \Omega, \Omega \rangle ^{\frac q2}$ and we are ready to make the following 
\begin{definition5.1}\label {D: 5.1} For a given $q\in \mathbb R\, ,$ a function or a differential form or a bundle-valued differential form $f$ has \emph{$p$-{finite growth}} $($or, simply, \emph{is
$p$-{finite}}$)$ if there exists $x_0 \in M$ such that
\begin{equation}
\liminf_{r\rightarrow\infty}\frac{1}{r^p}\int_{B(x_0;r)}|f|^{q}\, dv
<\infty\, , \label{5.1}
\end{equation}
and has \emph{$p$-{infinite growth}} $($or, simply, \emph{is
$p$-infinite}$)$ otherwise. \smallskip

For a given $q\in \mathbb R\, ,$ a  function or a differential form or a bundle-valued differential form $f$ has \emph{$p$-mild growth}
$($or, simply, \emph{is $p$-mild}$)$ if there exist  $ x_0 \in M\,
,$ and a strictly increasing sequence of $\{r_j\}^\infty_0$ going
to infinity, such that for every $l_0>0$, we have
\begin{equation}
\begin{array}{rll}
\sum\limits_{j=\ell_0}^{\infty}
\bigg(\frac{(r_{j+1}-r_j)^p}{\int_{B(x_0;r_{j+1})\backslash
B(x_0;r_{j})}|f|^q\, dv}\bigg)^{\frac1{p-1}}=\infty \, ,
\end{array}    \label{5.2}
\end{equation}
and has \emph{$p$-severe growth} $($or, simply, \emph{is
$p$-severe}$)$ otherwise. \smallskip

For a given $q\in \mathbb R\, ,$ a function or a differential form or a bundle-valued differential form $f$ has \emph{$p$-obtuse growth}
$($or, simply, \emph{is $p$-obtuse}$)$ if there exists $x_0 \in M$
such that for every $a>0$, we have
\begin{equation}
\begin{array}{rll}
\int^\infty_a\bigg( \frac{1}{\int_{\partial
B(x_0;r)}|f|^qds}\bigg)^\frac{1}{p-1}dr =
 \infty \, ,   \label{5.3}
\end{array}
\end{equation}
and has \emph{$p$-acute growth} $($or, simply, \emph{is
$p$-acute}$)$ otherwise. \smallskip

For a given $q\in \mathbb R\, ,$ a function or a differential form or a bundle-valued differential form $f$ has \emph{$p$-moderate
growth} $($or, simply, \emph{is $p$-moderate}$)$ if there exist  $
x_0 \in M$, and $\psi(r)\in {\mathcal F}$, such that
\begin{equation}\label{5.4}
\limsup _{r \to \infty}\frac {1}{r^p \psi^{p-1} (r)}\int_{B(x_0;r)}
|f|^{q}\, dv < \infty \, ,
\end{equation}
and has \emph{$p$-immoderate growth} $($or, simply, \emph{is
$p$-immoderate}$)$ otherwise, where
\begin{equation}\label{5.5} {\mathcal F} = \{\psi:[a,\infty)\longrightarrow
(0,\infty) |\int^{\infty}_{a}\frac{dr}{r\psi(r)}= \infty \ \ for \ \
some \ \ a \ge 0 \}\, .\end{equation} $($Notice that the functions
in {$\mathcal F$} are not necessarily monotone.$)$ \smallskip

For a given $q\in \mathbb R\, ,$ a function or a differential form or a bundle-valued differential form $f$ has \emph{$p$-small growth}
$($or, simply, \emph{is $p$-small}$)$ if there exists $ x_0 \in
M\, ,$ such that for every $a
>0\, ,$we have
\begin{equation}
\begin{array}{rll}
\int
_{a}^{\infty}\bigg(\frac{r}{\int_{B(x_0;r)}|f|^{q}\, dv}\bigg)^{\frac1{p-1}}
dr = \infty \, ,
\end{array}    \label{5.6}
\end{equation}
and has \emph{$p$-large growth} $($or, simply, \emph{is
$p$-large}$)$ otherwise. \end{definition5.1}

\begin{definition5.2}\label {D: 5.2} For a given $q\in \mathbb R\, ,$ a function or a differential form or a bundle-valued differential form $f$ has
\emph{$p$-balanced growth} $(or, simply, \emph{is $p$-balanced})$
if $f$ has one of the following: \emph{$p$-finite},
\emph{$p$-mild}, \emph{$p$-obtuse}, \emph{$p$-moderate}, or
\emph{$p$-small} growth, and has \emph{$p$-imbalanced growth} $( $
or, simply,
is $p$-\emph{imbalanced}$)$ otherwise.
\end{definition5.2}

The above definitions of ``$p$-balanced, $p$-finite, $p$-mild,
$p$-obtuse, $p$-moderate, $p$-small" and their counter-parts
``$p$-imbalanced, $p$-infinite, $p$-severe, $p$-acute,
$p$-immoderate, $p$-large" growth depend on $q$, and $q$ will be
specified in the context in which the definition is used.
\begin{theorem} For a given $q \in \mathbb R\, ,$ $f$ is \emph{$p$-moderate}$($resp. \emph{$p$-immoderate}$)$ if and
only if $f$ is \emph{$p$-small}$($resp. \emph{$p$-large}$)$.
\label{T: 5.1}
\end{theorem}

\begin{proof} ($\Rightarrow$)   
Since $f$ has $p$-moderate growth, we may assume, by the definition of the limit superior of functions,  there
exists a $\psi\in \mathcal{F} \big (\operatorname{as}\,  \operatorname{in}\, \eqref{5.5} \big )$, and a constant $K^{\prime} > 0$ such that
$\frac {1}{r^p \psi^{p-1}(r)}  \int_{B(x_0;r)}
|f|^{q}\, dv  \le K^{\prime}$ for $r > \ell_0\, .$ This implies that
 \begin{equation}
 \frac {r}{\int_{B(x_0;r)}
|f|^{q}\, dv} \ge  \frac {1}{K^{\prime}} \frac {1}{r^{p-1}  \psi^{p-1}(r)} \quad \operatorname{for} \quad r > {\ell}_0\, .\label{5.7}\end{equation}
Taking both sides to the power $\frac {1}{p-1}$ and integrating, we have
 \begin{equation}
\int _a^{\infty} \bigg (\frac {r}{\int_{B(x_0;r)}
|f|^{q}\, dv   } \bigg )^{\frac {1}{p-1}}\, dr \ge  \frac {1}{(K^{\prime})^{\frac {1}{p-1}}}\int _a^{\infty} \frac {1}{r  \psi(r)}\, dr \quad \operatorname{for} \quad r > {\ell}_0\, .\label{5.8}\end{equation}

If $f$ were $p$-large, \eqref{5.8} would lead to $\int _a^{\infty} \frac {1}{r  \psi(r)}\, dr < \infty$, contradicting \eqref{5.5}.
\smallskip

\noindent
($\Leftarrow$)\quad If $\int_{B(x_0;r)}
|f|^{q}\, dv = 0$ for $r > a > 0$, then we are done. Or, we let $\psi (r) = ( \frac {1}{r^p}\int_{B(x_0;r)}
|f|^{q}\, dv )^{\frac1{p-1}}\, ,\, r > a\, .$ Then \begin{equation}\label{5.9}\frac {1}{r \psi(r)}=\bigg( \frac {r}{\int_{B(x_0;r)}
|f|^{q}\, dv }
 \bigg)^{\frac1{p-1}}\, .\end{equation}
Integrating \eqref{5.9} from $a$ to $\infty$, we have

$$ \int _a^{\infty}\frac{1}{r\psi(r)}dr =  \int _a^{\infty}\bigg( \frac {r}{\int_{B(x_0;r)}
|f|^{q}\, dv }
 \bigg)^{\frac1{p-1}}dr = \infty, $$ by the assumption of $f$ being $p$-small. Hence
$\psi(r)\in \mathcal F$.
Suppose the
 contrary, $f$ were $p$-immoderate, i.e. (\ref{5.4}) were not true. Then we would 
 have, via \eqref{5.9}
 
 \noindent
 $\displaystyle
\limsup\limits_{r \rightarrow \infty}\frac {1}{r^p \psi ^{p-1} (r)}
\int_{B(x_0;r)}
|f|^{q}\, dv  =  \infty $, contradicting 

\noindent
$\limsup\limits_{r \rightarrow \infty}\frac {1}{r^p \psi ^{p-1} (r)}
\int_{B(x_0;r)}
|f|^{q}\, dv  =  \limsup\limits_{r \rightarrow \infty} \frac 1r \bigg( \frac {r}{\int_{B(x_0;r)}
|f|^{q}\, dv }
 \bigg)
\int_{B(x_0;r)}
|f|^{q}\, dv = 1\, .$
\end{proof}

\begin{theorem} 
For a given $q \in \mathbb R\, ,$ if $f$ has $p$-small growth then $f$ has $p$-mild growth. \label{T: 5.2}
\end{theorem}
\begin{proof}
For a strictly
increasing sequence $\{r_j\}$ with $r_{j+1}=2r_j\, ,$ 
we obtain
 \begin{equation}\label{5.10}
 \begin{array}{rll}
 \sum\limits_{j=\ell_0}^{\ell}
\bigg(\frac{(r_{j+1}-r_j)^p}{\int_{B(x_0;r_{j+1})\backslash
B(x_0;r_{j})}|f|^q\, dv}\bigg)^{\frac1{p-1}}  & \ge \sum
\limits_{j=\ell_0}^{\ell} \bigg(\frac{r_{j+1}-r_j}{\int_{B(x_0;r_{j+1})}|f|^q\, dv} \cdot (r_{j+1}-r_j)^{p-1}\bigg)^{\frac1{p-1}}  \\
& = \sum
\limits_{j=\ell_0}^{\ell} \bigg(\frac{r_{j+1}-r_j}{\int_{B(x_0;r_{j+1})}|f|^q\, dv} \bigg)^{\frac1{p-1}} \cdot (r_{j+1}-r_j)  \\
& = \sum
\limits_{j=\ell_0}^{\ell} \bigg(\frac {\frac {1}{2^2}\, r_{j+2}}{\int_{B(x_0;r_{j+1})}|f|^q\, dv} \bigg)^{\frac1{p-1}} \cdot \frac 12 (r_{j+2}-r_{j+1})  \\
& \ge \sum
\limits_{j=\ell_0}^{\ell}  \frac {1}{2^\frac{p+1}{p-1}} \int _{r_{j +
1}}^{r_{j+2}} \bigg( \frac{r}{\int_{B(x_0;r)}|f|^q\, dv}
\bigg)^{\frac{1}{p-1}}dr\, \\ 
& = \frac {1}{2^\frac{p+1}{p-1}} \int _{r_{\ell_0 +
1}}^{r_{\ell+2}} \bigg( \frac{r}{\int_{B(x_0;r)}|f|^q\, dv}
\bigg)^{\frac{1}{p-1}}dr\, .\end{array}
 \end{equation}
 where we have applied the Mean Value Theorem for integrals to the forth step.
 
Now suppose 
 contrary, $f$ were $p$-severe. Letting $l\rightarrow\infty$ in \eqref{5.10} would imply
\[
\infty >  \sum\limits_{j=\ell_0}^{\ell}
\bigg(\frac{(r_{j+1}-r_j)^p}{\int_{B(x_0;r_{j+1})\backslash
B(x_0;r_{j})}|f|^q\, dv}\bigg)^{\frac1{p-1}}  \ge \frac {1}{2^\frac{p+1}{p-1}} \int _{r_{\ell_0 +
1}}^{\infty} \bigg( \frac {r}{\int_{B(x_0;r)}|f|^q\, dv}
\bigg)^{\frac{1}{p-1}}dr\, ,\] contradicting the hypothesis that $f$ has a $p$-small growth.
\end{proof}
\begin{theorem} For a given $q \in \mathbb R\, ,$ \begin{itemize}
 \item[(i)] If $f$ is \emph{$p$-acute}\, $($resp. \emph{$p$-mild}$)$ then $f$ is
\emph{$p$-severe\, }$($resp. \emph{$p$-obtuse}$)$. If $f$ is
\emph{$p$-severe\, }$($resp. \emph{$p$-moderate} or \emph{$p$-small}\, $)$
then $f$ is \emph{$p$-immoderate} and \emph{$p$-large}\, $($resp.
\emph{$p$-mild}$)$.

 \item[(ii)] Conversely, suppose $\int_{B(x_0;r)}|f|^{q}dv$
 is convex in $r$, then f is \emph{$p$-immoderate} or \emph{$p$-large} implies that
f is \emph{$p$-severe} or \emph{$p$-acute}. Hence, f is
\emph{$p$-severe}, \emph{$p$-acute}, \emph{$p$-immoderate}, and
\emph{$p$-large} are all equivalent. 

\end{itemize}\label{T: 5.3}
\end{theorem}

\begin{proof}\quad(i) $\forall\,
0<s<t$, by the H\"older inequality, one gets:
\[
\begin{array}{rll}
t-s&=\int_{s}^{t}\big (\frac{d}{dr}\int_{B(x_0;r)}(f^2+\E)^{\frac{q}{2}}dv\big )^{\frac{1}{p}}\big (\frac{1}{\frac{d}{dr}\int_{B(x_0;r)}(f^2+\E)^{\frac{q}{2}}dv}\big )^{\frac{1}{p}}dr\\
&\le\bigg(\int_{s}^{t}\big (\frac{d}{dr}\int_{B(x_0;r)}(f^2+\E)^{\frac{q}{2}}dv\big )dr\bigg)^{\frac{1}{p}}\bigg(\int_{s}^{t}\big (\frac{1}{\frac{d}{dr}\int_{B(x_0;r)}(f^2+\E)^{\frac{q}{2}}dv}\big )^{\frac{1}{p-1}}dr\bigg)^{\frac{p-1}{p}}.\\
\end{array}
\]

Taking this to the power of $\frac {p}{p-1}$ on both sides, we have
\begin{equation}
\begin{array}{rll}
\bigg(\frac{(t-s)^p}{\int_{B(x_0;t)\backslash
B(x_0;s)}(f^2+\E)^{\frac{q}{2}}dv }\bigg)^{\frac 1{p-1}}\le
\int^t_s\big(\frac{1}{\frac{d}{dr}\int_{B(x_0;r)}(f^2+\E)^{\frac{q}{2}}}\big)^\frac{1}{p-1}dr\, .
\end{array}    \label{5.11}
\end{equation}

Hence, for every strictly increasing sequence of $\{r_j\}$ and every
$r_{\ell_0}>a$, setting $s=r_i, t=r_{i+1}$ and summing over $i$ from
$\ell_0$ to $\ell$ and letting $\ell\rightarrow \infty$, one gets:
\begin{equation}
\begin{array}{rll}
\sum\limits^{\infty}_{i=\ell_0}\bigg(\frac{(r_{i+1}-r_{i})^p}{\int_{B(x_0;r_{i+1})\backslash
B(x_0;s)}(f^2+\epsilon)^{\frac{q}{2}}dv}\bigg)^\frac{1}{p-1}\le
\int^{\infty}_{a}\big(\frac{1}{\frac{d}{dr}\int_{B(x_0;r)}|f|^q}\big)^\frac{1}{p-1}dr\, .
\end{array}    \label{5.12}
\end{equation}
Let $\E\rightarrow 0$, we get the desired first assertion by the
coarea formula and the Dominated Convergence Theorem. The proof of
the second assertion follows from Theorem \ref{T: 5.1} and Theorem \ref{T:
 5.2}. \\
(ii) By the convexity assumption, comparing the slope of the tangent and the slope of the chord at two distinct points (analogous to \cite {Gr}), we have  $$r > 0\, \Rightarrow \frac{d}{dr}\int_{B(x_0;r)}|f|^{q}dv
\ge \frac {\int_{B(x_0;r)}|f|^{q}dv - \int_{B(x_0;0)}|f|^{q}dv}{r-0} \, \operatorname{a.e.}\, .$$    Hence, $
\int_{a}^{\infty}\big(\frac{1}{\frac{d}{dr}\int_{B(x_0;r)}|f|^{q}dv}\big)^{\frac{1}{p-1}}dr
\le\int_{a}^{\infty}\big(\frac{r}{\int_{B(x_0;r)}|f|^{q}dv}\big)^{\frac{1}{p-1}}dr
< \infty$. The assertion follows from the coarea formula.
\end{proof}

\begin{proof}[{\bf Proof of Theorem 5.4}]
This follows at once from Theorems \ref{T: 5.1}, \ref{T: 5.2} and \ref{T: 5.3}.\end{proof}

\begin{definition5.3}  A function or differential form or bundle-valued differential form $f$ on $M$ is said to {\it belong to $L^p(M)$} $($or, simply, is $L^p)$ , denoted by $f \in L^p(M)$ if $ \int_{M} |f|^p dv < \infty\, .$ 
\end{definition5.3}

\begin{corollary}\label{C: 5.1}
 Every $L^q$ function or differential form or bundle-valued differential form $f$ on $M$ has \emph{$p$-balanced} growth, $p \ge 0\, ,$ and in fact, has \emph{$p$-finite},
\emph{$p$-mild}, \emph{$p$-obtuse}, \emph{$p$-moderate}, and
\emph{$p$-small} growth, $p \ge 0\, ,$  for the same value of $q$
% in the context in which the definition is used.
\end{corollary}

\begin{proof}
Since $f$ is $L^q$ on $M$, by Definition 5.3, we may assume $\int_{M}|f|^{q}\, dv = \tilde K$ for some constant $\tilde K < \infty\, .$ It follows from Definition 5.1 that every $L^q$ function or differential form or bundle-valued differential form $f$ on $M$ is both \emph{$p$-finite} and \emph{$p$-small}, $p \ge 0\, .$ Indeed, 
\[ \begin{aligned} \liminf_{r\rightarrow\infty}\frac{1}{r^p}\int_{B(x_0;r)}|f|^{q}\, dv & \le \liminf_{r\rightarrow\infty}\frac{1}{r^p}\tilde K = 0 
<\infty \\
\operatorname{and}\qquad\int _{a}^{\infty}\bigg(\frac{r}{\int_{B(x_0;r)}|f|^{q}\, dv}\bigg)^{\frac1{p-1}} & \ge \int _{a}^{\infty}\bigg(\frac{r}{\tilde K}\bigg)^{\frac1{p-1}}
dr = \infty\, . \end{aligned}
\]
Now the assertion follows from Definition 5.2 and Theorem 5.4.
\end{proof}

\section{Generalized harmonic forms with values in vector bundle, decomposition, integrability and growth}
By Stokes' theorem, a smooth
differential form $\omega$ on a compact Riemannian manifold is harmonic if and only if it is closed and co-closed. That is,  
\begin{equation}\label{6.1} \Delta \omega
= 0 \qquad {\rm if}\quad {\rm and}\quad {\rm only}\quad {\rm
if}\qquad d\, \omega = 0\quad {\rm and}\quad d^{\star}\omega = 0.
\end{equation}

The {\it exterior
differential operator} $d^\nabla :A^k(\xi )\rightarrow
A^{k+1}(\xi )$ relative to the connection $\nabla ^E$ is given by
\begin{equation}
\aligned
 d^\nabla \sigma \,
(X_1,...,X_{k+1})=&\sum_{i=1}^{k+1}(-1)^{i+1}\nabla
_{X_i}^E(\sigma (X_1,...,\widehat{X}_i,...,X_{k+1})) \\
&+\sum_{i<j}(-1)^{i+j}\sigma ([X_i,X_j],X_1,...,\widehat{X}_i,...,\widehat{X}%
_j,...,X_{k+1})\, ,\endaligned \label{6.2}
\end{equation}
where the symbols covered by $\, \, \widehat{}\, \, $ are omitted. Since the
Levi-Civita connection on $TM$ is torsion-free, we also have
\begin{equation}
(d^\nabla \sigma
)(X_1,...,X_{k+1})=\sum_{i=1}^{k+1}(-1)^{i+1}(\nabla _{X_i}\sigma
)(X_1,...,\widehat{X}_i,...,X_{k+1})\, .  \label{6.3}
\end{equation}

For two forms $\Omega ,\Omega ^{\prime }\in A^k(\xi )$, the
induced inner product is defined as follows:
\begin{equation}
\aligned
 \langle\Omega ,\Omega ^{\prime }\rangle&=\sum_{i_1<\cdots <i_k}\langle\Omega
(e_{i_1},...,e_{i_k}),\Omega ^{\prime }(e_{i_1},...,e_{i_k})\rangle_E \\
&=\frac 1{k!}\sum_{i_1,...,i_k}\langle\Omega
(e_{i_1},...,e_{i_k}),\Omega ^{\prime
}(e_{i_1},...,e_{i_k})\rangle_E\, ,
\endaligned\label{6.4}
\end{equation}
where $\{e_1, \cdots e_n\}$ is a local orthonormal frame field on $(M,g)\, .$ For $\Omega \in A^k(\xi )$, set $|\Omega|^2 = \langle \Omega, \Omega \rangle$ defined as in \eqref{6.4}. 
\smallskip

\begin{definition6.1} A bundle-valued differential form $\Omega\in A^k(\xi )$ on $M$ is said to satisfy 
{\bf Condition W} if 
\begin{equation}\label{6.5}
\aligned 
&|\langle d(|\Omega|^2)\wedge\Omega, d^\nabla\Omega\rangle | \leq 2|\Omega|^2|d^\nabla \Omega|^2\\
&\big |\langle d(| \star \Omega|^2)\wedge  \star\Omega, d^\nabla \star\Omega\rangle\big | \leq 2| \star\Omega|^2|d^\nabla  \star\Omega|^2\, .
\endaligned
\end{equation}\label{D:6.1}
\end{definition6.1}

Relative to the Riemannian structures of $E$ and $TM$, the
{\it codifferential operator} $\delta ^\nabla :A^k(\xi )\rightarrow
A^{k-1}(\xi )$ is characterized as the adjoint of $d$ via the
formula
$$
\int_M\langle d^\nabla \sigma ,\rho \rangle dv_g=\int_M\langle
\sigma ,\delta ^\nabla \rho\rangle dv_g\, ,
$$
where $\sigma \in A^{k-1}(\xi ),\rho \in A^k(\xi )$ , one of which
has compact support, and $dv_g$ is the volume element associated with the metric $g$ on $TM\, .$ Then
\begin{equation}
(\delta ^\nabla \rho )(X_1,...,X_{k-1})=-\sum_i(\nabla
_{e_i}\rho )(e_i,X_1,...,X_{k-1})\, .\label{6.6}  
\end{equation}

Let $\Delta ^\nabla: A^k(\xi) \to A^k(\xi) $ be given by \begin{equation}\Delta ^\nabla= -(d^\nabla \delta^\nabla + \delta^\nabla d^\nabla)\, .\label{6.7}\end{equation}

\begin{definition6.2}
$\Omega \in A^k(\xi )$ is said to be a {\it harmonic form with values in the vector bundle} $\xi :E\rightarrow M\, ,$ if $\Delta ^\nabla \Omega = 0$, and 
a {\it generalized harmonic form with values in the vector bundle} $\xi :E\rightarrow M\, ,$ if $\langle \Omega, \Delta^\nabla \Omega \rangle \ge 0$. 
\end{definition6.2}

\begin{theorem}$($Duality Theorem$)$\label{T: 6.1} $\quad (i)$ $\Omega \in A^k(\xi )$ is harmonic if and only if $\star\, \Omega \in A^{n-k}(\xi )\, $ is harmonic. $\quad (ii)$ $\Omega \in A^k(\xi )$ has $2$-balanced growth on $M$, for $q=2\, ,$ or for $1 < q(\ne2) < 3$ with $\Omega$ satisfying $\operatorname{Condition} \operatorname{W}$ $\big ($see \eqref{6.5}$\big ),$
if and only if  $\, \star\, \Omega \in A^{n-k}(\xi )\, $ has $2$-balanced growth on $M$, for $q=2\, ,$ or for $1 < q(\ne2) < 3$ with $\star\, \Omega$ satisfying $\operatorname{Condition} \operatorname{W}$ $\big ($see \eqref{6.5}$\big ).\quad (iii)$   
$\star\, \Omega$ is a solution of $\langle \star\, \Omega, \Delta \star\, \Omega\rangle \ge 0$ on $M$ if and only if $\star\, \Omega$ is closed and co-closed. 
\end{theorem}

\begin{theorem}$($Unity Theorem$)$
If a bundle-valued differential $k$-form $\Omega \in A^k(\xi )$ has $2$-balanced growth, for $q=2\, ,$ or for $1 < q(\ne2) < 3$ with $\Omega$ satisfying $\operatorname{Condition} \operatorname{W}$, then the following six statements 
$($i$)$ $\langle\Omega, \Delta ^\nabla \Omega\rangle \ge 0\,,$ $($ii$)$ $\Delta ^\nabla \Omega = 0\, ,$ $($iii$)$ $d^\nabla\, \Omega = \delta^{^\nabla}\Omega = 0\, ,$ $($iv$)$ $\langle \star\, \Omega, \Delta ^\nabla \star\, \Omega\rangle \ge 0\, ,$
$($v$)$ $\Delta ^\nabla \star\, \Omega = 0\, ,$ $($vi$)$ $d^\nabla\,\star\, \Omega = \delta^{\nabla\,} \star\, \Omega = 0$  are equivalent.
\label{T: 6.2}
\end{theorem}

\begin{proof}
$\big (($i$)$ $\Leftarrow$ $($ii$)$ $\Leftarrow$ $($iii$)\, \big )$ It is obvious. $\quad $ $\big (($i$)$ $\Rightarrow$ $($iii$)\big )$ 
Choose a
smooth cut-off function $\psi(x)$ as in \cite [(3.1)]{W2}, i.e.
for any $x_0 \in M$ and any pair of positive numbers $s,t$ with $s
< t\, ,$  a rotationally symmetric Lipschitz continuous
nonnegative function $\psi(x)=\psi(x;s,t)$ satisfies $\psi \equiv
1$ on $B(s)$, $\psi \equiv 0$ off $B(t)\, ,$ and $| \nabla \psi|
\leq \frac{C_1}{t-s}, \ \ a.e. \ \ on \ \ M\, ,$ where $C_1
> 0$ is a constant (independent of $x_0,s,t$). Let $1 < q < \infty\, 
,$ to be determined later.  For any constant $\E > 0\, ,$ we compute
\begin{equation}\label{6.8}
\begin{array}{rll}
0& \le & \int _{B(t)}\langle
\psi^2(|\Omega|^2+\epsilon)^{\frac{q}{2}-1}\Omega,\Delta^\nabla
\Omega\rangle \, dv\\
\end{array}
\end{equation}
and obtain

\begin{equation}\label{6.9}
\begin{array}{rll}
&\big (\int_{B(t)}\psi^2(|\Omega|^2+\epsilon)^{\frac{q}{2}-1}(|d^\nabla \Omega|^2+|\delta^{\nabla}\Omega|^2)\, dv\big )^2
\\
&\leq \bigg (\frac {2\sqrt{2}C_1}{1-|q-2|}\big )^2\big( \frac {1}{(t-s)^2}\int_{B(t)\backslash
B(s)}(|\Omega|^2+\epsilon)^{\frac{q}{2}}\, dv\bigg )\\
& \qquad \cdot \bigg (\int_{B(t)\backslash
B(s)}\psi^2(|\Omega|^2+\epsilon)^{\frac{q}{2}-1}(|d^\nabla \Omega|^2+|\delta^{\nabla}\Omega|^2)\, dv\bigg )\, .
\end{array}
\end{equation}
Proceeding as in \cite {W4}, we obtain the desired $($iii$)$. Hence, $($i$)$ $\Leftrightarrow$ $($ii$)$ $\Leftrightarrow$ $($iii$)$.
On the other hand, applying Duality Theorem \ref{T: 6.1}, we prove $($ii$)$ $\Leftrightarrow$ $($v$)\, ,$ and $($iv$)$ $\Leftrightarrow$ $($v$)$ $\Leftrightarrow$ $($vi$)\, .$
Consequently, $($i$)$ through $($vi$)$ are all equivalent.
\end{proof}

\begin{corollary}
Every $L^2$ bundle-valued harmonic $k$-form $\Omega \in A^{k}(\xi)$ is closed and co-closed and every $L^q, 1 < q(\ne 2) < 3$ bundle-valued  harmonic $k$-form $\Omega \in A^{k}(\xi)$ satisfying Condition $\rm{W}$ is closed and co-closed. \label{C:6.1}
\end{corollary}

\section{Applications in Geometric Differential-Integral Inequalities on manifolds}

The derived Laplacian comparison Theorems \ref{T: 2.6} - \ref{T: 2.7} and Hessian comparison Theorems \ref{T: 3.1} - \ref{T: 3.4}  have many applications. In particular, they shed light on geometric differential-integral inequalities on Riemannian manifolds.  Using an analog of Bochner's Method or $``B^2 - 4AC \le 0"$ Method, Wei and Li \cite {WL} proved the following theorem. 

\begin{theorem}$($\cite {WL}$)$ For every $u\in W_{0}^{1,2}(M\backslash \left\{ x_{0}\right\} ),$
and every $a,b\in
\mathbb{R}\, ,$ the following inequality holds on a manifold $M$ with a pole $x_0\, :$

\begin{equation} \label{7.1}
\begin{array}{rll}
\quad \frac 12\big |\underset{M}{\int }\,\frac{\left\vert
u\right\vert
^{2}}{r^{a+b+1}}(r \Delta r - a - b) dv\big |

\leq \left(\underset{M}{\int }\, \frac{| u|^{2}}{r^{2a}}dv\right)
^{\frac{1}{2}}\left( \underset{M}{\int }\, \frac{\left\vert \nabla
u\right\vert ^{2}}{r^{2b}}dv\right) ^{\frac{1}{2}},
\end{array}
\end{equation}
where $dv$ is the volume element of $M\, .$ \label{T: 7.1}
\end{theorem}

\subsection{Generalized sharp Caffarelli-Kohn-Nirenberg type inequalities on
Riemannian manifolds}

\begin{definition7.1}
$u : M \to \mathbb R$ is said to belong to $W_0^{1,p}(M)$ $($or, simply, is $W_0^{1,p})$ , denoted by $u \in W_0^{1,p}(M)$  if there exists a sequence $\{u_i\}$ in $C_0^{\infty}(M)$ such that  $\left(\int_{M} |u-u_i|^p + |\nabla (u- u_i)|^p dv\right)^{\frac 1p} \to 0$, as $i \to \infty\, .$
\end{definition7.1}

Applying Corollary \ref{C: 3.6} to Theorem \ref{T: 7.1},  Wei and Li \cite {WL} proved the following theorem.
\begin{theorem}$($\cite {WL}$)$\label{T: 7.2}
 Let $M$ be an $n$-dimensional complete
Riemannian manifold with a pole such that the radial curvature $K(r)$ of $M$ satisfies one of the
following three conditions:

\begin{equation}
\begin{aligned}
( i ) & K(r)\ge  0 \, \operatorname{and}\,  n \le a+b+1,\\
( ii ) & K(r) \leq  0\,   \operatorname{and}\,  a+b+1\leq n, \\
( iii ) & K(r) = 0\, \operatorname{and}\,  a,b \in \mathbb{R} \operatorname{are}\, \operatorname{any}\, \operatorname{constants}.
\end{aligned}\label{7.2}
\end{equation}
Then
 for every $u\in W_{0}^{1,2}(M\backslash \left\{ x_{0}\right\} ),$
and $a,b\in\label{7.2}
\mathbb{R}$

\begin{equation} \label{7.3}
{C}\underset{M}{\int }\frac{\left\vert u\right\vert
^{2}}{r^{a+b+1}}dv\leq \left( \underset{M}{\int }\frac{\left\vert
u\right\vert ^{2}}{r^{2a}}dv\right) ^{\frac{1}{2}}\left( \underset{M}{\int }%
\frac{\left\vert \nabla u\right\vert ^{2}}{r^{2b}}dv\right) ^{\frac{1}{2}},
\end{equation}%
where the constant
$C$ is given by
\begin{equation}\label{7.4}
C = C(a,b) = \left\{
\begin{array}{cc}
-\frac{n-(a+b+1)}{2}, & \text {if } K(r)  \text { satisfies $($i$)$},\\
\frac{n-(a+b+1)}{2}, & \text{if } K(r) \text{ satisfies $($ii$)$},\\
\bigg|\frac{n-(a+b+1)}{2}\bigg|, & \text{if } K(r) \text{
satisfies} (iii)\, .
\end{array}
\right.
\end{equation}
\end{theorem}

The case $M=\mathbb{R}^n$ is due to Caffarelli et al. \cite {CKN} and Costa \cite{C}.

\begin{theorem}$($Generalized sharp Caffarelli-Kohn-Nirenberg type inequalities under radial Ricci curvature assmumptions$)$\label{T: 7.3}
 Let $M$ be an $n$-dimensional complete
Riemannian manifold with a pole such that the radial Ricci curvature $\text{Ric}^{\rm{M}}_{\rm{rad}}$ of $M$ satisfies   one of the
following five conditions:

\begin{equation} 
\begin{aligned}
( i )  & \text{Ric}^{\rm{M}}_{\rm{rad}} (r) \ge - (n-1)\frac {A(A-1)}{r^2}\,  \operatorname{and}\,  n \le  \frac {a+ b+ A}{A}, \\
( ii )  & \text{Ric}^{\rm{M}}_{\rm{rad}} (r) \ge - (n-1)\frac {A(A-1)}{(c+r)^2}\,  \operatorname{and}\,  n \le  \frac {a+ b+ A}{A}, \\
( iii )  & \text{Ric}^{\rm{M}}_{\rm{rad}} (r) \ge   (n-1)\frac {B_1(1-B_1)}{r^2}\,  \operatorname{and}\,  n \le \frac{2a+2b+1+\sqrt{1+4B_1(1-B_1)}}{1+\sqrt{1+4B_1(1-B_1)}},  \\
( iv )  & \text{Ric}^{\rm{M}}_{\rm{rad}} (r) \ge   (n-1)\frac {B_1(1-B_1)}{(c+r)^2}\,  \operatorname{and}\,  n \le \frac{2a+2b+1+\sqrt{1+4B_1(1-B_1)}}{1+\sqrt{1+4B_1(1-B_1)}},  \\
( v )  & \text{Ric}^{\rm{M}}_{\rm{rad}} (r) \ge   0\,  \operatorname{and}\,  n \le  a+ b+1\, ,
\end{aligned}\label{7.5}
\end{equation}
where  $0 \le B_1 \le 1 \le A\, $ are constants. Then 
 for every $u\in W_{0}^{1,2}(M\backslash \left\{ x_{0}\right\} ),$
and $a,b\in
\mathbb{R}$

\begin{equation} \tag{7.3}
{C}\underset{M}{\int }\frac{\left\vert u\right\vert
^{2}}{r^{a+b+1}}dv\leq \left( \underset{M}{\int }\frac{\left\vert
u\right\vert ^{2}}{r^{2a}}dv\right) ^{\frac{1}{2}}\left( \underset{M}{\int }%
\frac{\left\vert \nabla u\right\vert ^{2}}{r^{2b}}dv\right) ^{\frac{1}{2}}\, ,
\end{equation}%
where
\begin{equation}\label{7.6}
\begin{aligned}
C & = C(a,b,A,B_1)\\
& = \left\{
\begin{array}{cc}
\frac{a+b-(n-1)A}{2}, & \text {if } \text{Ric}^{\rm{M}}_{\rm{rad}}  \text { satisfies } (i)\, or\, (ii),\\
\frac{2a+2b-(n-1)\big (1+\sqrt{1+4B_1(1-B_1)}\big )}{4}, & \text{if } \text{Ric}^{\rm{M}}_{\rm{rad}}  \text{ satisfies } (iii)\, or\, (iv),\\
\frac {a+b+1-n}{2}, & \text{if } \text{Ric}^{\rm{M}}_{\rm{rad}}  \text{
satisfies } (v).
\end{array}
\right.
\end{aligned}
\end{equation}
\end{theorem}

\begin{proof}
This follows from Theorem$\ref{T: 2.6}$ and Theorem $\ref{T: 2.7}$ that we can apply $(\ref{2.40})$ and $(\ref{2.52})\, $ to simplify the integral inequality $(\ref{7.1}),$
based on Theorem \ref{T: 7.1} and obtain the desired.
\end{proof}

\begin{theorem}$($Generalized sharp Caffarelli-Kohn-Nirenberg type inequalities under radial curvature assmumptions$)$\label{T: 7.4}
 Let $M$ be an $n$-dimensional complete
Riemannian manifold with a pole such that the radial curvature $K(r)$ of $M$ satisfies one of the
following twelve conditions:

\begin{equation}
\begin{aligned}
( i )  & K(r)  \ge - \frac {A(A-1)}{r^2}\qquad    \operatorname{and}\qquad   n \le  \frac {a+b+A}{A}, \\
( ii )  & K(r)  \ge - \frac {A(A-1)}{(c+r)^2}\qquad  \operatorname{and}\qquad  n \le  \frac {a+b+ A}{A}, \\
( iii )  & K(r)  \le - \frac {A_1(A_1-1)}{r^2}\qquad    \operatorname{and}\qquad   n \ge  \frac {a+b+A_1}{A_1}, \\
( iv )  & K(r)  \le - \frac {A_1(A_1-1)}{(c+r)^2}\qquad  \operatorname{and}\qquad  n \ge  \frac {a+b+ A_1}{A_1}, \\
( v )  & K(r)  = - \frac {A(A-1)}{r^2}\qquad    \operatorname{and}\qquad    a, b \in \mathbb {R}\, \operatorname{are}\, \operatorname{any}\, \operatorname{constants},\\
( vi )  & K(r)  = - \frac {A(A-1)}{(c+r)^2}\qquad  \operatorname{and}\qquad   a, b \in \mathbb {R}\, \operatorname{are}\, \operatorname{any}\, \operatorname{constants}, \\
( vii )  & K(r)  = - \frac {A}{r^2}\qquad   \qquad  \operatorname{and}\qquad    a, b \in \mathbb {R}\, \operatorname{are}\, \operatorname{any}\, \operatorname{constants},\\
( viii )  & K(r)  = - \frac {A}{(c+r)^2}\qquad  \operatorname{and}\qquad   a, b \in \mathbb {R}\, \operatorname{are}\, \operatorname{any}\, \operatorname{constants}, \\
( ix) & K(r) \ge \frac {B_1(1-B_1)}{r^2}\qquad    \operatorname{and}\qquad    n \le \frac{2a+2b+1+\sqrt{1+4B_1(1-B_1)}}{1+\sqrt{1+4B_1(1-B_1)}},\\
( x) & K(r) \ge \frac {B_1(1-B_1)}{(c+r)^2}\qquad    \operatorname{and}\qquad    n \le \frac{2a+2b+1+\sqrt{1+4B_1(1-B_1)}}{1+\sqrt{1+4B_1(1-B_1)}},\\
( xi ) & K(r) \leq  \frac {B(1-B)}{r^2}\qquad     \operatorname{and}\qquad    n \ge \frac{a+b+|B - \frac 12| + \frac 12}{|B - \frac 12| + \frac 12}, \\
( xii ) & K(r) \le \frac {B(1-B)}{(c+r)^2}\qquad   \operatorname{and}\qquad     n \ge \frac{a+b+|B - \frac 12| + \frac 12}{|B - \frac 12| + \frac 12}\, ,
\end{aligned}\label{7.7}
\end{equation}
where $0 \le B, B_1 \le 1 \le A\, ,$ and $0 \le c$ are constants. Then 
 for every $u\in W_{0}^{1,2}(M\backslash \left\{ x_{0}\right\} ),$
and $a,b\in
\mathbb{R}$

\begin{equation} \tag{7.3}
{C}\underset{M}{\int }\frac{\left\vert u\right\vert
^{2}}{r^{a+b+1}}dv\leq \left( \underset{M}{\int }\frac{\left\vert
u\right\vert ^{2}}{r^{2a}}dv\right) ^{\frac{1}{2}}\left( \underset{M}{\int }%
\frac{\left\vert \nabla u\right\vert ^{2}}{r^{2b}}dv\right) ^{\frac{1}{2}}\, ,
\end{equation}%
where the constant
$C$ is given by
\begin{equation}\label{7.8}
C = C(a,b,B) = \left\{
\begin{array}{cc}
\frac{a+b-(n-1)A}{2}, & \text {if } K(r)  \text{ satisfies } (i)\, \text{or}\, (ii),\\
- \frac{a+b-(n-1)A}{2}, & \text {if } K(r)  \text{ satisfies } (iii)\, \text{or}\, (iv),\\
\big |\frac{a+b-(n-1)A}{2}\big |, & \text {if } K(r)  \text{ satisfies } (v)\, \text{or}\, (vi),\\
\big | \frac{  2a+2b-(n-1)\big (1+\sqrt{1+4A}\big )}{4} \big |, & \text {if } K(r)  \text { satisfies } (vii)\, or\, (viii),\\
\frac{  2a+2b-(n-1)\big (1+\sqrt{1+4B_1(1-B_1)}\big )}{4}, & \text {if } K(r)  \text { satisfies } (ix)\, or\, (x),\\
 \frac{(n-1)(|B - \frac 12| + \frac 12)-a-b}{2}, & \text{if } K(r)\text{ satisfies } (xi)\, or\, (xii).
\end{array}
\right.
\end{equation}
\end{theorem}

\begin{proof}
This follows from Theorems $\ref{T: 3.1}$, $\ref{T: 3.2}$, Corollary $\ref{C: 3.2}$, and  Theorems $\ref{T: 3.3}$, $\ref{T: 3.4}$ that we can apply $(\ref{3.3})$, $(\ref{3.6})$, $(\ref{3.23})\, ,$ $(\ref{3.27})$, and $(\ref{3.36}) $  to simplify the integral inequality $(\ref{7.1}),$
based on Theorem \ref{T: 7.1}.
\end{proof}

\subsection{Embedding theorems for weighted Sobolev spaces of functions}

Hessian comparison theorems via these geometric inequalities lead to
embedding theorems for weighted Sobolev spaces of functions on Riemannian manifolds. 
Let $M$ be manifold as in Theorem \ref{T: 7.3}, or Theorem \ref{T: 7.4}, or Theorem \ref{T: 7.2}. Following Costa's notation \cite {C}, or \cite {WL}, we let $D_{\gamma }^{1,2}(M)$ denote the completion of $C_{0}^{\infty }(M\backslash
\left\{ x_{0}\right\})$
with respect to the norm

\begin{equation}
\left\vert \left\vert u\right\vert \right\vert _{D_{\gamma
}^{1,2}(M)}:=\left( \underset{M}{\int }\frac{\left\vert \nabla
u\right\vert ^{2}}{r^{2\gamma }}dv\right) ^{\frac{1}{2}}\, .\label{7.9}
\end{equation}
$L_{\gamma }^{2}(M)$ denote the completion of $C_{0}^{\infty }(M\backslash
\left\{ x_{0}\right\})$
with respect to the norm

\begin{equation}
\left\vert \left\vert u\right\vert \right\vert _{L_{\gamma }^{2}(M)}:=\left(
\underset{M}{\int }\frac{\left\vert u\right\vert ^{2}}{r^{2\gamma }}%
dv\right) ^{\frac{1}{2}}\, \label{7.10}
\end{equation}%
and $H_{a,b}^{1}(M)$ denote the completion of $C_{0}^{\infty }(M\backslash
\left\{ x_{0}\right\} )$ with respect to the Sobolev norm

\begin{equation}
\left\vert \left\vert u\right\vert \right\vert _{H_{a,b}^{1}(M)}:=\bigg (
\underset{M}{\int }\big ( \frac{\left\vert u\right\vert ^{2}}{r^{2a}}+\frac{
\left\vert \nabla u\right\vert ^{2}}{r^{2b}}\big ) dv\bigg ) ^{\frac{1}{2}}\, .\label{7.11}
\end{equation}

If $\gamma = 0$ in \eqref{7.9}, we simply write $D_{0
}^{1,2}(M) = D^{1,2}(M)\, .$ Applying arithmetic-mean-geometric-mean inequality to Theorems \ref{T: 7.3} - \ref{T: 7.4} and \ref{T: 7.2}, we have

\begin{theorem}$($Embedding Theorem$)$\label{T: 7.5} Let $M$ be a manifold of radial curvature $K$ satisfying one of the five conditions in $(\ref{7.5})$, or one of the twelve conditions in $(\ref{7.7})$, or one of the three conditions in $(\ref{7.2})\, .$ Then the following continuous embeddings hold

\begin{equation}
H_{a,b}^{1}(M) \subset L_{\frac {a+b+1}{2}}^{2}(M)\qquad and \qquad H_{b,a}^{1}(M) \subset L_{\frac {a+b+1}{2}}^{2}(M)\, .\label{7.12}
\end{equation}
\end{theorem}

\subsection{Geometric differential-integral inequalities on manifolds}

As a consequence of embedding theorem, we have geometric differential-integral inequalities on manifolds:

\begin{theorem}\label{T: 7.6} Let $M$ be a manifold of radial curvature $K$ satisfying one of the five conditions in $(\ref{7.5})$, or one of the twelve conditions in $(\ref{7.7})$. Then we have the following seven geometric differential-integral inequalities:

\textbf{i$)$ }For any $u\in H_{b+1,b}^{1}(M)\, ,$

\begin{equation}\label{7.13}
C_1\underset{M}{\int }\frac{\left\vert
u\right\vert ^{2}}{r^{2(b+1)}}dv\leq \underset{M}{\int }\frac{\left\vert
\nabla u\right\vert ^{2}}{r^{2b}}dv\, ,
\end{equation}
where 
\begin{equation}\label{7.14}
C_1 = \left\{
\begin{array}{cc}
\left( \frac{(n-1)A-1}{2}-b\right) ^{2}, & \text {if } \operatorname{Ric}^{\rm{M}}_{\rm{rad}}(r)  \text { satisfies}\, \eqref{7.5} (i)\, \text{or}\, \eqref{7.5}(ii)\,   \\
& \operatorname{or}\, K(r)  \text {satisfies}\, \text{any}\, \text{of}\, \eqref{7.7} (i)-(vi)\, ,\\
\left(\frac{  4b+2-(n-1)\big (1+\sqrt{1+4A}\big )}{4}\right) ^{2}, & \text {if } K(r)  \text { satisfies}\, \eqref{7.5} (vii)\,  \text{or}\, \eqref{7.5}(viii)\, ,  \\
\left(\frac{4b+2-(n-1)\big (1+\sqrt{1+4B_1(1-B_1)}\big )}{4}\right) ^{2}, &\text {if }\operatorname{Ric}^{\rm{M}}_{\rm{rad}}(r)\text { satisfies}\, \eqref{7.5} (iii)\,  \text{or}\, \eqref{7.5}(iv)\\
& \operatorname{or} \, K(r)  \text { satisfies }\, \eqref{7.7} (ix)\, \text{or}\, \eqref{7.7}(x)\, , \\
\left( \frac {2b + 2 -n}{2}\right)^{2}, &\text {if } \operatorname{Ric}^{\rm{M}}_{\rm{rad}}(r)  \text { satisfies}\, \eqref{7.5} (v)\, ,\qquad \qquad  \\
\left(\frac{(n-1)|B - \frac 12| + \frac {n-3}{2}-2b}{2}\right) ^{2}, & \text{if } K(r)\text{ satisfies }\, \eqref{7.7}  (xi)\, \text{or}\, \eqref{7.7}(xii),
\end{array}
\right.
\end{equation}
$\operatorname{in}\, \operatorname{which}\, a = b + 1\, $ for \eqref{7.5} and \eqref{7.7};

\textbf{ii$)$ }For any $u\in H_{a,a+1}^{1}(M)\, ,$ 

\begin{equation}\label{7.15}
C_2 \left( \underset{M}{\int }\frac{\left\vert
u\right\vert ^{2}}{r^{2(a+1)}}dv\right)^2\leq \left( \underset{M}{\int }\frac{%
\left\vert u\right\vert ^{2}}{r^{2a}}dv\right) \left( \underset%
{M}{\int }\frac{\left\vert \nabla u\right\vert ^{2}}{r^{2(a+1)}}dv\right)\, ,
\end{equation}
where 
\begin{equation}\label{7.16}
C_2= \left\{
\begin{array}{cc}
\left( \frac{(n-1)A-1}{2} - a\right) ^{2}, & \text {if } \operatorname{Ric}^{\rm{M}}_{\rm{rad}}(r)  \text { satisfies}\, \eqref{7.5} (i)\, \text{or}\, \eqref{7.5}(ii)\,  \\
& \operatorname{or} \, K(r)  \text {satisfies}\, \text{any}\, \text{of}\, \eqref{7.7} (i)-(vi)\, ,\\
\left(\frac{  4a+2-(n-1)\big (1+\sqrt{1+4A}\big )}{4}\right) ^{2}, & \text {if } K(r)  \text { satisfies}\, \eqref{7.5} (vii)\,  \text{or}\, \eqref{7.5}(viii)\, ,  \\
\left(\frac{  4a+2-(n-1)\big (1+\sqrt{1+4B_1(1-B_1)}\big )}{4}\right) ^{2}, & \text {if } \operatorname{Ric}^{\rm{M}}_{\rm{rad}}(r)  \text { satisfies}\, \eqref{7.5} (iii)\,  \text{or}\, \eqref{7.5}(iv)\,  \\
& \operatorname{or} \, K(r)  \text { satisfies }\, \eqref{7.7} (ix)\, \text{or}\, \eqref{7.7}(x)\, , \\
\left( \frac {2a + 2 -n}{2}\right)^{2}, & \text {if } \operatorname{Ric}^{\rm{M}}_{\rm{rad}}(r)  \text { satisfies}\, \eqref{7.5} (v)\, ,  \qquad \qquad  \\
\left(\frac{(n-1)|B - \frac 12| + \frac {n-3}{2}-2a}{2}\right) ^{2}, & \text{if } K(r)\text{ satisfies }\, \eqref{7.7}  (xi)\, \text{or}\, \eqref{7.7}(xii),
\end{array}
\right.
\end{equation}
$\operatorname{in}\, \operatorname{which}\, b = a + 1\, $ for \eqref{7.5} and \eqref{7.7}\, ;

\textbf{iii$)$ }If $u\in H_{-(b+1),b}^{1}(M)$ then $u$ $\in L^{2}(M)$
and

\begin{equation}\label{7.17}
C_3 \left( \underset{M}{\int }\left\vert u\right\vert
^{2}dv\right) ^2 \leq \left( \underset{M}{\int }r^{2(b+1)}\left\vert u\right\vert
^{2}dv\right) \left( \underset{M}{\int }\frac{\left\vert
\nabla u\right\vert ^{2}}{r^{2b}}dv\right) ,
\end{equation}
where 
\begin{equation}\label{7.18}
C_3= \left\{
\begin{array}{cc}
\left( \frac{(n-1)A+1}{2} \right) ^{2}, & \text {if } \operatorname{Ric}^{\rm{M}}_{\rm{rad}}(r)  \text { satisfies}\, \eqref{7.5} (i)\, \text{or}\, \eqref{7.5}(ii)\,  \\
& \operatorname{or} \, K(r)  \text {satisfies}\, \text{any}\, \text{of}\, \eqref{7.7} (i)-(vi)\, ,\\
\left(\frac{  2+(n-1)\big (1+\sqrt{1+4A}\big )}{4}\right) ^{2}, & \text {if } K(r)  \text { satisfies}\, \eqref{7.5} (vii)\,  \text{or}\, \eqref{7.5}(viii)\, ,  \\
\left(\frac{  2+(n-1)\big (1+\sqrt{1+4B_1(1-B_1)}\big )}{4}\right) ^{2}, & \text {if } \operatorname{Ric}^{\rm{M}}_{\rm{rad}}(r)  \text { satisfies}\, \eqref{7.5} (iii)\,  \text{or}\, \eqref{7.5}(iv)\,  \\
& \operatorname{or} \, K(r)  \text { satisfies }\, \eqref{7.7} (ix)\, \text{or}\, \eqref{7.7}(x)\, , \\
\frac {n^{2}}{4}, & \text {if } \operatorname{Ric}^{\rm{M}}_{\rm{rad}}(r)  \text { satisfies}\, \eqref{7.5} (v)\, ,  \qquad \qquad  \\
\left(\frac{(n-1)|B - \frac 12| + \frac {n+1}{2}}{2}\right) ^{2}, & \text{if } K(r)\text{ satisfies }\, \eqref{7.7}  (xi)\, \text{or}\, \eqref{7.7}(xii),
\end{array}
\right.
\end{equation}
$\operatorname{in}\, \operatorname{which}\, a = -  b - 1\, $ for \eqref{7.5} and \eqref{7.7};

\textbf{iv$)$ } If $u\in H_{0,1}^{1}(M),$ then $u$ $\in L_{1}^{2}(M)$
and

\begin{equation}\label{7.19}
C_4\left( \underset{M}{\int }\frac{\left\vert u\right\vert ^{2}}{r^{2}}%
dv\right)^2\leq \left( \underset{M}{\int }\left\vert u\right\vert ^{2}dv\right)\left( \underset{M}{\int }\frac{\left\vert \nabla u\right\vert
^{2}}{r^{2}}dv\right)\, ,
\end{equation}
where 
\begin{equation}\label{7.20}
C_4= \left\{
\begin{array}{cc}
\left( \frac{(n-1)A-1}{2} \right) ^{2},  & \text {if } \operatorname{Ric}^{\rm{M}}_{\rm{rad}}(r)  \text { satisfies}\, \eqref{7.5} (i)\, \text{or}\, \eqref{7.5}(ii)\,  \\
& \operatorname{or} \, K(r)  \text {satisfies}\, \text{any}\, \text{of}\, \eqref{7.7} (i)-(vi)\, ,\\
\left(\frac{  2-(n-1)\big (1+\sqrt{1+4A}\big )}{4}\right) ^{2}, & \text {if } K(r)  \text { satisfies}\, \eqref{7.5} (vii)\,  \text{or}\, \eqref{7.5}(viii)\, ,  \\
\left(\frac{  (n-1)\big (1+\sqrt{1+4B_1(1-B_1)}\big )-2}{4}\right) ^{2}, & \text {if } \operatorname{Ric}^{\rm{M}}_{\rm{rad}}(r)  \text { satisfies}\, \eqref{7.5} (iii)\,  \text{or}\, \eqref{7.5}(iv)\,  \\
& \, \operatorname{or}\,  K(r)  \text { satisfies }\, \eqref{7.7} (ix)\, \text{or}\, \eqref{7.7}(x)\, , \\
(\frac {n-2}{2})^{2}, \,& \text {if } \operatorname{Ric}^{\rm{M}}_{\rm{rad}}(r)  \text { satisfies}\, \eqref{7.5} (v)\, ,  \qquad \qquad  \\
\left(\frac{(n-1)|B - \frac 12| + \frac {n-3}{2}}{2}\right) ^{2}, & \text{if } K(r)\text{ satisfies }\, \eqref{7.7}  (xi)\, \text{or}\, \eqref{7.7}(xii),
\end{array}
\right.
\end{equation}
$\operatorname{in}\, \operatorname{which}\, a = 0\, , b = 1\, $ for \eqref{7.5} and \eqref{7.7}\, ;

\textbf{v$)$ }If $u\in H_{-1,1}^{1}(M),$ then $u$ $\in
L_{\frac{1}{2}}^{2}(M)$ and

\begin{equation}\label{7.21}
C_5 \left( \underset{M}{\int }\frac{\left\vert u\right\vert
^{2}}{r}dv\right) ^2 \leq \left( \underset{M}{\int }r^{2}\left\vert u\right\vert
^{2}dv\right) \left( \underset{M}{\int }\frac{\left\vert
\nabla u\right\vert ^{2}}{r^{2}}dv\right)\, ,
\end{equation}
where 
\begin{equation}\label{7.22}
C_5= \left\{
\begin{array}{cc}
\left( \frac{(n-1)A}{2} \right) ^{2},  & \text {if } \operatorname{Ric}^{\rm{M}}_{\rm{rad}}(r)  \text { satisfies}\, \eqref{7.5} (i)\, \text{or}\, \eqref{7.5}(ii)\,    \\
& \operatorname{or}\, K(r)  \text {satisfies}\, \text{any}\, \text{of}\, \eqref{7.7} (i)-(vi)\, ,\\
\left(\frac{ (n-1)\big (1+\sqrt{1+4A}\big )}{4}\right) ^{2}, & \text {if } K(r)  \text { satisfies}\, \eqref{7.5} (vii)\,  \text{or}\, \eqref{7.5}(viii)\, ,  \\
\left(\frac{  (n-1)\big (1+\sqrt{1+4B_1(1-B_1)}\big )}{4}\right) ^{2}, & \text {if } \operatorname{Ric}^{\rm{M}}_{\rm{rad}}(r)  \text { satisfies}\, \eqref{7.5} (iii)\, \text{or}\, \eqref{7.5}(iv)\,     \\
& \operatorname{or}\, K(r)  \text { satisfies }\, \text{any}\, \text{of}\, \eqref{7.7} (ix)\, \text{or}\, \eqref{7.7}(x)\, , \\
(\frac {n-1}{2})^{2}, & \text {if } \operatorname{Ric}^{\rm{M}}_{\rm{rad}}(r)  \text { satisfies}\, \eqref{7.5} (v)\, ,  \qquad \qquad  \\
\left(\frac{(n-1)(|B - \frac 12| + \frac 12 )}{2}\right) ^{2}, & \text{if } K(r)\text{ satisfies }\, \eqref{7.7}  (xi)\, \text{or}\, \eqref{7.7}(xii)\, ,
\end{array}
\right.
\end{equation}
$\operatorname{in}\, \operatorname{which}\, a = -1\, , b = 1\, $ for \eqref{7.5} and \eqref{7.7};

\textbf{vi$)$ }If $u\in H^{1}(M)=$ $H_{0,0}^{1}(M),$ then $u$ $\in L_{\frac{1%
}{2}}^{2}(M)$ and

\begin{equation}\label{7.23}
C_6 \left( \underset{M}{\int }\frac{\left\vert u\right\vert
^{2}}{r}dv\right)^2 \leq \left( \underset{M}{\int }\left\vert u\right\vert
^{2}dv\right) \left( \underset{M}{\int }\left\vert \nabla
u\right\vert ^{2}dv\right)\, , 
\end{equation}
where 
\begin{equation}\label{7.24}
C_6 = C_5\, ,
\end{equation}
$\operatorname{in}\, \operatorname{which}\, a = 0\, , b = 0\, $ for \eqref{7.5} and \eqref{7.7}\, ;

\textbf{vii$)$ } For any $u\in D^{1,2}(M),$ 

\begin{equation}\label{7.25}
C_7 \underset{M}{\int }\frac{\left\vert u\right\vert ^{2}}{r^{2}}%
dv \leq \underset{M}{\int }\left\vert \nabla u\right\vert ^{2}dv\, ,
\end{equation}
where 
\begin{equation}\label{7.26}
C_7= \left\{
\begin{array}{cc}
\left( \frac{(n-1)A-1}{2} \right) ^{2},  & \text {if } \operatorname{Ric}^{\rm{M}}_{\rm{rad}}(r)  \text { satisfies}\, \eqref{7.5} (i)\, \text{or}\, \eqref{7.5}(ii)\,    \\
& \operatorname{or}\, K(r)  \text {satisfies}\, \text{any}\, \text{of}\, \eqref{7.7} (i)-(vi)\, ,\\
\left(\frac{2 - (n-1)\big (1+\sqrt{1+4A}\big )}{4}\right) ^{2}, & \text {if } K(r)  \text { satisfies}\, \eqref{7.5} (vii)\,  \text{or}\, \eqref{7.5}(viii)\, ,  \\
\left(\frac{(n-1)\big (1+\sqrt{1+4B_1(1-B_1)}\big )-2}{4}\right) ^{2}, & \text {if } \operatorname{Ric}^{\rm{M}}_{\rm{rad}}(r)  \text { satisfies}\, \eqref{7.5} (iii)\,  \text{or}\, \eqref{7.5}(iv)\,  \\
& \operatorname{or}\, K(r)  \text { satisfies }\, \eqref{7.7} (ix)\, \text{or}\, \eqref{7.7}(x)\, , \\
(\frac {n-2}{2})^{2}, & \text {if } \operatorname{Ric}^{\rm{M}}_{\rm{rad}}(r)  \text { satisfies}\, \eqref{7.5} (v)\, , \qquad \qquad  \\
\left(\frac{(n-1)|B - \frac 12| + \frac {n-3}{2}}{2}\right) ^{2}, & \text{if } K(r)\text{ satisfies }\, \eqref{7.7}  (xi)\, \text{or}\, \eqref{7.7}(xii),
\end{array}
\right.
\end{equation}
$\operatorname{in}\, \operatorname{which}\, a = 1\, , b = 0\, $ for \eqref{7.5} and \eqref{7.7}.
\end{theorem}

\begin{remark7.1}
The case $M= \mathbb{R}^n$ is due to Costa $($\cite{C}$)$.
\end{remark7.1}

\begin{proof} We make special choices in Theorems \ref{T: 7.3} and \ref{T: 7.4} as follows:

\textbf{i)} \ \ \ Let $a=b+1;$

\textbf{ii)} \ \ Let $b=a+1;$

\textbf{iii) }\ Let $a=-b-1;$

\textbf{iv)  }\ Let $a=0,b=1;$

\textbf{v)} \ \ Let $a=-1,b=1;$

\textbf{vi) }  Let $a=0,b=0; $

\textbf{vii) }  Let $a=1,b=0. $
\end{proof}

\subsection{Generalized sharp Hardy type inequalities on
Riemannian manifolds}

Laplacian comparison Theorems further lead to

\begin{theorem}$($Generalized Sharp Hardy Type Inequality$)$
\label{T: 7.7} Let $M$ be an $n$-manifold with a pole satisfying 
\begin{equation}\tag{2.39}\operatorname{Ric}^{\rm{M}}_{\rm{rad}}(r) \geq - (n-1)\frac {A(A-1)}{r^2}\quad \operatorname{where}\quad A \ge 1\, .\end{equation}  Then for every $u\in
W_{0}^{1,p}(M)$, $\frac {u}{r} \in L^p(M)$ with $p > (n-1) A + 1$, we have \begin{equation}  \label{7.27}
\big (\frac{p-1-(n-1)A}{p}\big ) ^{p}\int_{M}\frac{|u|^p}{r^p}dv\leq \int_{M}\left\vert \nabla u\right\vert ^{p}dv\, .
\end{equation}
\end{theorem}
\smallskip

We will apply a double limiting arguement to the following 

\begin{lemma7.1}$($Geometric Differential-Integral Inequality$)$ $($see \cite[(1.3)]{WL}, \cite[(3)]{CLW2}, \cite[(8.1)]{WW}$)$ Let $u\in C_{0}^{\infty}(M)$ and $\partial B_{\delta}(x_0)$ be the $C^1$ boundary of the geodesic ball $B_{\delta}(x_0)$ centered
at $x_0$ with radius $\delta >0$. Let $V$ be an open set with smooth boundary $\partial V$ such that $V\subset\subset M$, and $u=0$ off $V$. We choose a sufficiently small $\delta>0$ so that $\partial V\cap\partial B_{\delta}(x_0)=\emptyset$. Then  for every $\epsilon>0$ and  $p>1$, we have
\begin{equation}  \label{7.28}
\begin{array}{lll}
\left |-\int _{V\cap\partial B_\delta(x_0)}\frac{r }{r^{p}+\epsilon}%
|u|^p dS + \int _{M\backslash B_\delta(x_0)}\, \frac{%
(r^p+\epsilon)( r\Delta r +1)-p\,r^p }{(r^{p}+\epsilon)^2}\left\vert
u\right\vert ^{p}dv \right | &  &  \\
\leq p\bigg( \int _{M\backslash B_\delta(x_0)}\, \big(\frac{%
\left\vert u\right\vert ^{p-1}r}{r^p+\epsilon}\big)^{\frac {p}{p-1}} dv\bigg)%
^{\frac {p-1}{p}}\bigg( \int _{M\backslash B_\delta(x_0)}\,
\left\vert \nabla u\right\vert ^{p}dv\bigg)^{\frac 1p}, &  &  \\
&  &
\end{array}%
\end{equation}
where $dS$
and $dv$ are the volume element of $\partial B_{\delta}(x_0)$ and $M$
respectively.\end{lemma7.1}
\begin{proof}[Proof of Theorem \ref{T: 7.1}]
Applying the Laplacian comparison Theorem \ref{T: 2.6} under the assumption $%
\text{Ric}_{rad}\geq - (n-1)\frac {A(A-1)}{r^2}$, one has $r\Delta r+1 \leq (n-1) A + 1 < p$. 
By the triangle inequality, \eqref{7.28} implies
\begin{equation}  \label{7.29}
\begin{array}{lll}
 \int _{M\backslash B_\delta(x_0)}\, \frac{\big(p-(n-1)A -1\big)r^p- \big((n-1)A +1\big) \epsilon}{(r^{p}+\epsilon)^2}\left\vert
u\right\vert ^{p}dv  &  &  \\
\leq p\bigg( \int _{M\backslash B_\delta(x_0)}\, \big(\frac{%
\left\vert u\right\vert ^{p-1}r}{r^p+\epsilon}\big)^{\frac {p}{p-1}} dv\bigg)%
^{\frac {p-1}{p}}\bigg( \int _{M\backslash B_\delta(x_0)}\,
\left\vert \nabla u\right\vert ^{p}dv\bigg)^{\frac 1p}  + \left |\int _{\partial B_\delta(x_0)}\frac{r }{r^{p}+\epsilon}%
|u|^p dS \right |.&  &  \\
&  &
\end{array}%
\end{equation}
For sufficiently small $\delta>0$, one has
\begin{equation}\label{7.30}\int_{\partial B_{\delta}(x_0)}\frac{r}{r^p+\epsilon}|u|^p dS=0\quad \text{if}\quad x_0\notin V\end{equation} and
\begin{equation}\label{7.31}\left|\int_{\partial B_{\delta}(x_0)}\frac{r}{r^p+\epsilon}|u|^p dS\right|
\to 0\quad \text{as}\quad \delta\to 0\quad \text{if}\quad x_0\in V,\end{equation}
Indeed,  $\frac{r}{r^p+\epsilon}$ is a continuous, nondecreasing function for $r \in [0, \delta _0]$, where $\delta _0 = (\frac{\epsilon}{p-1})^{\frac 1p}$ and $u$ is bounded in $M$. Hence,  for $\delta < \delta _0\, ,$
\begin{equation}\label{7.32}\left|\int_{\partial B_{\delta}(x_0)}\frac{r}{r^p+\epsilon}|u|^p dS\right|
\leqslant\frac{\delta}{\delta ^p+\epsilon} \int _{\partial B_{\delta}(x_0)} \max_M |u|^p dS.
\end{equation} This implies \eqref{7.31}.
It follows from \eqref{7.28}, via \eqref{7.30} or \eqref{7.31} 
that for every $\epsilon>0$,
\begin{equation}  \label{7.33}\begin{aligned}
& \int _{M\backslash B_\delta(x_0)}\, \frac{\big(p-(n-1)A -1\big)r^p- \big((n-1)A +1\big) \epsilon}{(r^{p}+\epsilon)^2}\left\vert
u\right\vert ^{p}dv \\
& \le p\bigg(  \int _{M\backslash B_\delta(x_0)}\,  \big(\frac{\left\vert
u\right\vert ^{p-1}r}{r^p+\epsilon}\big)^{\frac {p}{p-1}} dv\bigg)^{\frac {%
p-1}{p}}\bigg( \int _{M}\, \left\vert \nabla u\right\vert ^{p}dv\bigg)^{\frac 1p}\, .\end{aligned}
\end{equation}

We observe that the integrands in the left, and in the first factor in the right of \eqref{7.33} are monotone and  uniformly bounded above by a positive constant multiple of $\big |\frac {u}{r}\big |^p$ on $M\, .$ Since $\frac {u}{r} \in L^p(M)$, by the dominated convergent theorem, as $\epsilon\to 0\, ,$
\begin{equation}\label{7.34}  \big (p-(n-1)A -1\big ) \left( \int _{M\backslash B_\delta(x_0)}\, \left|\frac {u}{r}\right|^p dv \right)
\leqslant p\left( \int _{M\backslash B_\delta(x_0)}\, \left|\frac {u}{r}\right|^p dv\right)^{\frac{p-1}{p}}
\left(\int_M| \nabla u|^p\, dv\right)^{\frac{1}{p}}.\end{equation}
Simplifying and raising it to the $p$-th power,
\begin{equation}\label{7.35}  \left(\frac{p-(n-1)A -1}{p}\right)^p \left(\int_{M}\left|\frac {u}{r}\right|^p dv\right)
\leqslant
\int_{M}|\nabla u|^p dv.\end{equation}

Now we {\it extend \eqref{7.35} from $u \in C^{\infty}_{0}(M)$ to $u \in W_{0}^{1,p}(M)\, .$} Let $\{u_i\}$ be a sequence of functions in  $ C_0^{\infty}(M)$ tending to $u \in W_{0}^{1,p}(M)\, .$  
Applying the inequality \eqref{7.35} to difference
$u_{i_m} - u_{i_n}\, ,$ we have

\begin{equation}\label{7.36}  \left(\int_{M}\left|\frac {u_{i_m} - u_{i_n}}{r}\right|^p dv\right)^{\frac 1p}
\leqslant
  \left(\frac{p}{p-(n-1)A -1}\right)\left(\int_{M}|\nabla u|^p dv\right )^{\frac 1p}.\end{equation}
  
\noindent
Hence $\{\frac {u_i}{r}\}$ is a Cauchy sequence  in $L^p(M)\, .$ Now we recall the following 
\begin{lemma7.2} $($see Theorem 3.3 in \cite{WW}$)\quad L^p(M)\, , 1 \le p \le \infty\, $ is complete, i.e. every Cauchy sequence $\{u_i\}$ in $L^p(M)$ converges $\big ($This means that if for every $\epsilon > 0$, there exists $N$ such that $\left(\int_{M} |u_{i} - u_{j}| ^p dv\right)^{\frac 1p} <  \epsilon$, when $i > N$ and $j > N$, then there exists a unique function $u\in L^p(M)$, such that $\left(\int_{M} |u_{i} - u| ^p \, dv \right)^{\frac {1}{p}} \to 0$, as $i \to \infty \big )$.
\end{lemma7.2}

 \begin{lemma7.3} $($see Theorem 3.4 in \cite{WW}$)\quad$ 
 If $\{u_i\}$ is a Cauchy sequence in $L^p(M)\, , 1 \le p \le \infty\, ,$ then there exists a subsequence  $\{u_{i_k}\}$ and a nonnegative function $U$ in $L^p(M)$ such that

\noindent
 $(1)$ $|u_{i_k}| \leqslant U$ almost everywhere in $M .$

 \noindent
 $(2)$ $\lim _{k \to \infty} u_{i_k} = u$ almost everywhere in $M .$
\end{lemma7.3}

In view of Lemma 7.2, there exists a limiting function $f(x) \in L^p(M)\, $ satisfying

\begin{equation}
\begin{aligned}
\int _{M} |f(x)|^p\, dv = \lim_{i\to \infty} \int _{M} \frac {|u_i(x)|^p}{r^p}\, dv
& \leqslant \bigg (\frac{p}{p-(n-1)A -1}\bigg )^p \lim_{i\to \infty} \int _{M} | \nabla u_{i} |^p \, dv\\
& \leqslant \bigg (\frac{p}{p-(n-1)A -1}\bigg )^p \int _{M} | \nabla u |^p \, dv.\\
\label{7.37}
\end{aligned}
\end{equation}
On the other hand, since $\frac {1}{r^p}$ is bounded in $M\backslash  B_{\epsilon}(x_0)\, ,$ where $  B_{\epsilon}(x_0)$ is the open geodesic ball of radius $\epsilon > 0\, ,$ centered at $x_0$, and the pointwise convergence in Lemma 7.3 (2), we have for every   $\epsilon > 0\, ,$

\begin{equation}
\begin{aligned}
\int _{M\backslash  B_{\epsilon}(x_0)} |f(x)|^p\, dv =  \lim_{i\to \infty} \int _{M\backslash  B_{\epsilon}(x_0)} \frac {|u_i(x)|^p }{r^p}\, dv &= \int _{M\backslash  B_{\epsilon}(x_0)} \frac {|u|^p }{r^p}\, dv\\
& = \int _{M} \chi _{M\backslash  B_{\epsilon}(x_0)} \frac {|u|^p }{r^p}\, dv,
\label{7.38}
\end{aligned}
\end{equation}
where $\chi _{M\backslash  B_{\epsilon}(x_0)}$ is the characteritic function on $ M\backslash  B_{\epsilon}(x_0)$.
As $\epsilon \to 0\, ,$ monotone convergence theorem and \eqref{7.38} imply that
\begin{equation}
\begin{aligned}
\int _{M} |f(x)|^p\, dv =  \lim_{i\to \infty} \int _{M} \frac {|u_i|^p }{r^p}\, dv & = \int _{M} \frac {|u|^p }{r^p}\, dv.\\
\end{aligned}
\label{7.39}\end{equation}
Substituting \eqref{7.39} into \eqref{7.37} we obtain the desired \eqref{7.27} for $u \in W_{0}^{1,p}(M)$.
\end{proof}

\begin{remark7.2}
Theorem \ref{T: 7.7}. is sharp and recaptures a theorem of Chen-Li-Wei $($see \cite {CLW1}$)$, when $A=1\, .$ A double limiting argument and techniques in \cite {WW} are employed in proving Theorem \ref{T: 7.7}. 
\end{remark7.2}

\section{Monotonicity Formulae}

Let $F:[0,\infty )\rightarrow [0,\infty )$
be a $C^2$
function such that $F^{\prime }>0$ on $[0,\infty )\, ,$ and $F(0)=0$. 
\begin{definition8.1}  The $F$-degree $d_F$ is defined to be

\begin{equation}  d_F=\sup_{t\geq 0}\frac{tF^{\prime }(t)}{F(t)}\, . \label{8.1}
\end{equation}
\end{definition8.1}
Let $\omega$ be a smooth $k$-forms on a smooth $n$-dimensional Riemannian manifold $M$ with
values in the vector bundle $\xi :E\rightarrow M$. At each fiber of $E$ is equipped with a positive inner product $\langle \quad , \quad \rangle_E\, .$
Set $|\omega|^2 = \langle \omega, \omega \rangle_E\, .$ 
The
$\mathcal{E}_{F,g}$-energy functional given by
\begin{equation}
\mathcal{E}_{F,g}(\omega )=\int_MF(\frac{|\omega |^2}2)dv_g \, . \label{8.2}
\end{equation}

The {\it stress-energy tensor $S_{F, \omega}$ associated with the $\mathcal{E}_{F,g}$-energy functional} is
defined as follows (see \cite {BE, Ba, ArM, LSC, DW}):
\begin{equation}
S_{F,\omega }(X,Y)=F(\frac{|\omega |^2}2)g(X,Y)-F^{\prime }(\frac{|\omega |^2%
}2)\langle i_X\omega ,i_Y\omega \rangle\, ,  \label{8.3}
\end{equation}
where $i_X\omega$ is the interior multiplication by the vector
field $X\, .$

\begin{definition8.2}
$\omega \in
A^k(\xi )$ ($k\geq 1$) is said to satisfy an
\emph {$F$-conservation law} if $S_{F,\omega }$ is divergence free, i.e. the $(0,1)$-type tensor field $\operatorname{div} S_{F,\omega }$ vanishes identically 
\begin{equation}
\operatorname{div} S_{F,\omega }\equiv 0.\label{8.4}
\end{equation}
\end{definition8.2}

Let $\flat $ denote the bundle isomorphism that identifies the vector field $X$ with the differential one-form $X^{\flat}$, and let $\nabla $ be the Riemannian connection of $M$. 
Then the covariant derivative $\nabla X^{\flat}$ of $X^{\flat}$ is a $(0,2)$-type tensor, given by 
\begin{equation} \nabla X^{\flat} (Y,Z) = \nabla _Z X^{\flat} Y = \langle \nabla _Z X, Y\rangle\, , \quad \forall\, X,Y \in \Gamma (M)\, . \label{8.5}
\end{equation}
If $X$ is conservative, then 

\begin{equation} X = \nabla f,\quad  X^{\flat} = df\quad \operatorname{and}\quad  \nabla X^{\flat} = \operatorname{Hess} (f)\, .\label{8.6}
\end{equation}
for some scalar potential $f$ (see \cite {CW3}, p. 1527).
 A direct computation yields (see, e.g., \cite {DW})\begin{equation} \operatorname{div}
(i_X S_{\omega}) = \langle S_{F,\omega
},\nabla X^{\flat}\rangle  +  (\operatorname{div}
S_{F, \omega}) (X)\, ,\quad \forall\, X\in \Gamma (M)\, . \label{8.7}
\end{equation} 
It follows from the divergence theorem that, if $\omega $ satisfies an $F$-conservation law, we
have for every bounded domain $D$ in $M\, $ with $C^1$ boundary $\partial D\, ,$
\begin{equation}
\int_{\partial D}S_{F,\omega }(X,\nu ) ds_g = \int_D\langle S_{F,\omega
},\nabla X^{\flat}\rangle dv_g\, , \label{8.8}
\end{equation}
where $\nu $ is unit outward normal vector field along $\partial
D$ with $(n-1)$-dimensional volume element $ds_g$. When we choose scalar potential $f (x) = \frac 12 r^2 (x)$, curvature via \eqref{8.6} and our Hessian comparison theorems will influence the behavior of the $F$-stress energy $S_{F,\omega
}$ and the underlying criticality $\omega$ with the help from the following. 

\begin{lemma8.1} $($\cite {DW}$)$ \label{L: 8.1} Let $M$ be a complete manifold with a pole $x_0$.
Assume that there exist two positive functions $h_1(r)$ and
$h_2(r)$ such that
\begin{equation}
h_1(r)[g-dr\otimes dr]\leq \operatorname{Hess} (r)\leq h_2(r)[g-dr\otimes dr]
\label{8.9}
\end{equation}
on $M\backslash \{x_0\}$. If $h_2(r)$ satisfies
\begin{equation}
rh_2(r)\geq 1\,,  \label{8.10}
\end{equation}
then
\begin{equation}
\langle S_{F,\omega },\nabla X^{\flat} \rangle\,\geq
\,\big(1+(n-1)rh_1(r)-2kd_F  r h_2(r)\big)F( \frac{|\omega |^2}2)\, ,  \label{8.11}
\end{equation}
where $X=r \nabla  r$.
\end{lemma8.1}

\begin{theorem} Let $(M,g)$ be an $n$-dimensional complete
Riemannian manifold with a pole $ x_0$. Let $\xi :E\rightarrow M$
be a Riemannian vector bundle on $M$ and $ \omega \in A^k(\xi )$.
Assume that the radial curvature $K(r)$ of $M$ satisfies one of the
following seven conditions:

\begin{equation}
\aligned
(i)&\quad \eqref{3.14}\, \operatorname{holds}\,  \operatorname{with}\,  1+(n-1)A_1-2kd_FA > 0; \\
(ii)&\quad \eqref{3.17}\, \operatorname{holds}\,   \operatorname{with}\,   1 + (n-1)\frac{1 + \sqrt {1+4A_1}}{2} -kd_F(1 + \sqrt {1+4A}) > 0;\\
(iii)&\quad \eqref{3.49}\, \operatorname{holds}\,   \operatorname{with}\,   1 + (n-1)(|B-\frac {1}{2}|+ \frac {1}{2}) -kd_F\big (1 + \sqrt {1+4B_1(1-B_1)}\big ) > 0;\\ 
(iv)&\quad \eqref{3.45}\, \operatorname{holds}\,  \operatorname{with}\,  1+ (n-1)\frac{1 + \sqrt {1-4B}}{2} -kd_F(1 + \sqrt {1+4B_1} )  > 0; \\
(v)& \quad -\alpha ^2\leq K(r)\leq -\beta ^2\quad  \operatorname{with}\quad \alpha > 0\, , \beta > 0 \quad \operatorname{and}\quad (n-1)\beta -2k\alpha d_F\geq 0;\\
(vi)& \quad K(r) = 0\quad  \operatorname{with}\quad  n-2kd_F>0;\\
(vii)& -\frac A{(1+r^2)^{1+\epsilon}}\leq K(r) \leq \frac B{(1+r^2)^{1+\epsilon}}\quad  \operatorname{with}\quad \epsilon > 0\, , A \ge 0\, , 0 < B < 2\epsilon\quad \operatorname{and}\\
&\qquad n - (n-1)\frac B{2\epsilon} -2k e^{\frac {A}{2\epsilon}}d_F > 0.
\endaligned\label{8.12}
\end{equation}
If $\omega $ satisfies an $F$-conservation law, then
\begin{equation}
\frac 1{\rho _1^\lambda }\int_{B_{\rho _1}(x_0)}F(\frac{|\omega
|^2}2) dv \leq \frac 1{\rho _2^\lambda }\int_{B_{\rho
_2}(x_0)}F(\frac{|\omega |^2}2) dv\label{8.13}
\end{equation}
for any $0<\rho _1\leq \rho _2$, where
\begin{equation}
\lambda =\left\{
\begin{array}{cc}
1+(n-1)A_1-2kd_FA, & \text{if }\quad K(r) \text{ satisfies $($i$)$}, \\
1 + (n-1)\frac{1 + \sqrt {1+4A_1}}{2} -kd_F(1 + \sqrt {1+4A}),& \text{if }\quad K(r) \text{ satisfies $($ii$)$}, \\
1 + (n-1)(|B-\frac {1}{2}| +\frac 12)-kd_F\big (1 + \sqrt {1+4B_1(1-B_1)}\big ), & \text{if }\quad K(r) \text{ satisfies $($iii$)$}, \\
1+ (n-1)\frac{1 + \sqrt {1-4B}}{2} -kd_F(1 + \sqrt {1+4B_1} ),  & \text{if }\quad K(r) \text{ satisfies $($iv$)$}, \\
n-2k\frac \alpha \beta d_F, & \text {if }\quad  K(r)  \text { satisfies $($v$)$},\\
n-2kd_F, &\text {if } K(r) \text{ satisfies $($vi$)$},\\
n - (n-1)\frac B{2\epsilon} -2k e^{\frac {A}{2\epsilon}}d_F, &\text{if } K(r) \text{
satisfies $($vii$)$}\, .

\end{array}
\right. \label{8.14}
\end{equation}\label{T: 8.1}
\end{theorem}

\begin{proof}Choose a smooth conservative vector field $X= \nabla (\frac 12 r^2)$ on $M\, .$ If $K(r)$ satisfies $(\ref{8.12}) $, then by Corollary \ref{C: 3.1}, Theorem A, Corollary \ref{C: 3.5},  and Theorem \ref{T: 3.5}, (\ref{8.10}) holds. Hence Lemma 8.1 is applicable and by (\ref{8.11}) , we have on $B_\rho(x_0)\backslash\{x_0\}\, ,$ for every $\rho >0,$
\begin{equation}
\aligned \langle S_{F,\omega },\nabla X^{\flat} \rangle
\ge \lambda F( \frac{|\omega |^2}2) \, ,
\endaligned
\label{8.15}
\end{equation}
where $\lambda$ is as in $(\ref{8.14}) \, .$   Thus, by the continuity of $\langle S_{F,\omega },\nabla X^{\flat} \rangle\, $ and $ F( \frac{|\omega |^2}2), \, $ and (\ref{8.3}) , we
have for every $\rho >0,$ \eqref{8.15} holds on $B_\rho(x_0)$ and
\begin{equation}
\aligned
& \rho\, \, F(\frac{|\omega |^2}2)\geq S_{F,\omega }(X,\frac \partial {\partial r}) \qquad \text{on} \quad \partial B_\rho(x_0)\, . \endaligned\label{8.16}
\end{equation}
It follows from (\ref{8.8}), \eqref{8.15}  and (\ref{8.16})  that
\begin{equation}
\rho\int_{\partial B_\rho(x_0)}F(\frac{|\omega |^2}2) ds \geq \lambda
\int_{B_\rho(x_0)}F( \frac{|\omega |^2}2) dv\, . \label{8.17}
\end{equation}
Hence we get from (\ref{8.17})  the following
\begin{equation}
\frac{\int_{\partial B_\rho(x_0)}F(\frac{|\omega |^2}2) ds}{\int_{B_\rho(x_0)}F(\frac{%
|\omega |^2}2) dv} \geq \frac \lambda \rho\, .  \label{8.18}
\end{equation}
The coarea formula implies that
$$
\frac d{d\rho}\int_{B_\rho(x_0)}F(\frac{|\omega |^2}2) dv =\int_{\partial B_\rho(x_0)}F(%
\frac{|\omega |^2}2) ds\, .
$$
Thus, we have
\begin{equation}
\frac{\frac d{d\rho}\int_{B_\rho(x_0)}F(\frac{|\omega
|^2}2) dv}{\int_{B_\rho(x_0)}F( \frac{|\omega |^2}2) dv}\geq \frac \lambda
\rho
\label{8.19}
\end{equation}
for a.e. $\rho >0\, .$ By integration $(\ref{8.19})$ over $[\rho _1,\rho _2]$, we have
$$
\ln \int_{B_{\rho _2}(x_0)}F(\frac{|\omega |^2}2) dv -\ln
\int_{B_{\rho _1}(x_0)}F(\frac{|\omega |^2}2) dv \geq \ln \rho
_2^\lambda -\ln \rho _1^\lambda\, . $$
This proves (\ref{8.13}).
\end{proof}

\section{Vanishing Theorems} 
In this section we list some results that are immediate applications of the monotonicity formulae in the last section.

\subsection{Vanishing theorems for vector bundle valued $k$-forms}

\begin{theorem} \label{T: 9.1}Suppose the radial curvature $K(r)$ of $M$ satisfies the condition $(\ref{8.12})$. If $\omega \in A^k(\xi )$ satisfies an
$F$-conservation law $(\ref{8.4})$ and
\begin{equation}
\int_{B_\rho(x_0)}F(\frac{|\omega |^2}2)\, dv = o(\rho^\lambda )\quad \text{as
} \rho\rightarrow \infty\label{9.1}
\end{equation}
where $\lambda $ is given by $(\ref{8.14})$ depending on the curvature $K(r)$, then $F(\frac{|\omega |^2}2)\equiv 0\, ,$ and hence $\omega \equiv 0$. In
particular, if $\omega $ has finite $\mathcal{E}_{F,g}$-energy, then
$\omega \equiv 0$.
\end{theorem}

\begin{proof}
By Theorem 8.1, the monotonicity formula \eqref{8.13} holds. Let $\lambda_ 2 \to \infty$ in \eqref{8.13}. Then \eqref{9.1} implies that $F(\frac{|\omega |^2}2)\equiv 0\, ,$ and hence $\omega \equiv 0$. 
\end{proof}

\subsection{Applications in $F$-Yang-Mills fields}

Let $R^\nabla $ be an $F$-Yang-Mills field, associated with an $F$-Yang-Mills connection $\nabla$ on the adjoint bundle $Ad(P)$ of a principle $G$-bundle over a manifold $M\, .$ Then $R^\nabla $ can be viewed as a $2$-form with values in the adjoint bundle over $M\, ,$ and by \cite [Theorem 3.1]{DW},
$\omega = R^\nabla$ satisfies an $F$-conservation law. We have the following vanishing theorem for $F$-Yang-Mills fields:

 \begin{theorem}\label{T: 9.2}Suppose the radial curvature $K(r)$ of $M$ satisfies the condition $(\ref{8.12})\, ,$ in which $k=2\, .$
 Assume $F$-Yang-Mills field  $R^{\nabla}$ satisfies the following growth
condition
\begin{equation}
\int_{B_\rho(x_0)} F(\frac{|R^{\nabla} |^2}2)\, dv = o(\rho^\lambda )\quad \text{as
} \rho\rightarrow \infty,\label{9.2}
\end{equation}
where $\lambda $ satisfies the condition $(\ref{8.14})\, $
for $k=2\, .$ Then $R^{\nabla} \equiv 0$ on $M\, .$
 In particular, every $F$-Yang-Mills field $R^{\nabla}$  with finite $F$-Yang-Mills energy vanishes on $M$.
\end{theorem}

\begin{proof} This follows at once from Theorem \ref{T: 9.1} in which $k=2$ and $\omega = R^\nabla\, .$
\end{proof}

This theorem becomes the following vanishing theorem for $p$-Yang-Mills fields, when $F(t) = \frac 1p (2t)^{\frac{p}{2}}, p > 1\, (\operatorname{Hence}\, d_F= \frac p2\, \operatorname{by}\, \operatorname{Definition} 8.1):$  
\begin{theorem} \label{T: 9.3}Suppose the radial curvature $K(r)$ of $M$ satisfies the one of the following seven conditions:

\begin{equation}
\aligned
(i)&\quad \eqref{3.14}\, \operatorname{holds}\,  \operatorname{with}\,  1+(n-1)A_1-2pA > 0; \\
(ii)&\quad (\ref{3.17})\, \operatorname{holds}\,   \operatorname{with}\,   1 + (n-1)\frac{1 + \sqrt {1+4A_1}}{2} -p(1 + \sqrt {1+4A}) > 0;\\
(iii)&\quad (\ref{3.49})\, \operatorname{holds}\,   \operatorname{with}\,   1 + (n-1)(|B-\frac 12|+\frac 12) - p\big (1 + \sqrt {1+4B_1(1-B_1)}\big ) > 0;\\ 
(iv)&\quad (\ref{3.45})\, \operatorname{holds}\,  \operatorname{with}\,  1+ (n-1)\frac{1 + \sqrt {1-4B}}{2} - p ( 1 + \sqrt {1+4B_1} ) > 0 \, ;\\
(v)& \quad -\alpha ^2\leq K(r)\leq -\beta ^2\quad  \operatorname{with}\quad \alpha > 0, \beta > 0 \quad \operatorname{and}\quad (n-1)\beta -2p \alpha \geq 0;\\
(vi)& \quad K(r) = 0\quad  \operatorname{with}\quad  n-2p>0;\\
(vii)& -\frac A{(1+r^2)^{1+\epsilon}}\leq K(r) \leq \frac B{(1+r^2)^{1+\epsilon}}\quad  \operatorname{with}\quad \epsilon > 0\, , A \ge 0\, , 0 < B < 2\epsilon\quad \operatorname{and}\\
&\qquad n - (n-1)\frac B{2\epsilon} -2p e^{\frac {A}{2\epsilon}} > 0.
\endaligned \label{9.3}
\end{equation}

Then every $p$-Yang-Mills field $R^{\nabla}$ with the following growth
condition vanishes:
\begin{equation}
\frac 1p \int_{B_\rho(x_0)}|R^{\nabla } |^p \, dv = o(\rho^\lambda )\quad \text{as
} \rho\rightarrow \infty\label{9.4}
\end{equation}
where \begin{equation}
\lambda =\left\{
\begin{array}{cc}
1+(n-1)A_1-2pA, & \text{if }\quad K(r) \text{ satisfies $($i$)$}, \\
1 + (n-1)\frac{1 + \sqrt {1+4A_1}}{2} -p(1 + \sqrt {1+4A}),& \text{if }\quad K(r) \text{ satisfies $($ii$)$}, \\
1 + (n-1)(|B-\frac 12|+\frac 12) - p \big (1 + \sqrt {1+4B_1(1-B_1)} \big ), & \text{if }\quad K(r) \text{ satisfies $($iii$)$}, \\
1 + (n-1)\frac{1 + \sqrt {1-4B(1-B)}}{2} - p - p \sqrt {1+4B_1}, & \text{if }\quad K(r) \text{ satisfies $($iv$)$},\\
n-2p\frac \alpha \beta,  & \text {if }\quad  K(r)  \text { satisfies $($v$)$},\\
n-2p, &\text {if } K(r) \text{ satisfies $($vi$)$},\\
n - (n-1)\frac B{2\epsilon} -2p e^{\frac {A}{2\epsilon}}, &\text{if } K(r) \text{
satisfies $($vii$)$}\, .

\end{array}
\right.\label{9.5}
\end{equation}
In particular, every $p$-Yang-Mills field $R^{\nabla}$  with finite $\mathcal{YM}_p$-energy vanishes on $M$.
\end{theorem}

\begin{corollary} Let  $M\, ,$ $N\, ,$ $K(r)\, ,$ $\lambda\, ,$ and the growth
condition (\ref{9.4}) be as in Theorem \ref{T: 9.3}, in which $p=2\, .$ Then every Yang-Mills field $R^\nabla \equiv 0$ on $M$.\label{C: 9.1}
\end{corollary}
\section{Liouville Type Theorems}

\subsection{Liouville type theorems for $F$-harmonic maps}

Let $u: M \to N$ be an $F$-harmonic map. Then its differential $du $ can be viewed as a $1$-form with values in the induced bundle $u^{-1}TN\, .$ Since
$\omega=du $ satisfies an $F$-conservation law $(\ref{8.4})\, ,$ we obtain
the following Liouville-type

\begin{theorem}\label{T: 10.1} Let $N$ be a Riemannian manifold. Suppose the radial curvature $K(r)$ of $M$ satisfies one of the seven conditions in \eqref{12.3}.
Then every $F$-harmonic map $u: M \to N$ with the following growth
condition is a constant:
\begin{equation}
\int_{B_\rho(x_0)}F(\frac{|du |^2}2)\, dv = o(\rho^\lambda )\quad \text{as
} \rho\rightarrow \infty\, ,\label{10.1}
\end{equation}
where $\lambda $ is as in \eqref{12.4}.  In particular, every $F$-harmonic map $u: M \to N$ with finite $F$-energy is a constant.
\end{theorem}

\begin{proof} This follows at once from Theorem \ref{T: 9.1}, in which $k=1$ and $\omega = du\, .$
\end{proof}

Analogously, we have the following Liouville Theorem for $p$-harmonic maps:

\begin{theorem}\label{T: 10.2} Let $N$ be a Riemannian manifold. Suppose the radial curvature $K(r)$ of $M$ satisfies one of the following seven conditions:
\begin{equation}
\aligned
(i)&\quad \eqref{3.14}\, \operatorname{holds}\,  \operatorname{with}\,  1+(n-1)A_1-pA > 0; \\
(ii)&\quad \eqref{3.17}\, \operatorname{holds}\,   \operatorname{with}\,   1 + (n-1)\frac{1 + \sqrt {1+4A_1}}{2} -p\frac{1 + \sqrt {1+4A}}{2} > 0;\\
(iii)&\quad \eqref{3.49}\, \operatorname{holds}\,   \operatorname{with}\,   1 + (n-1)(|B-\frac 12|+\frac 12) -p\frac{1 + \sqrt {1+4B_1(1-B_1)}}{2} > 0;\\ 
(iv)&\quad \eqref{3.45}\, \operatorname{holds}\,  \operatorname{with}\,  1+ (n-1)\frac{1 + \sqrt {1-4B}}{2} -p\frac{1 + \sqrt {1+4B_1}}{2} > 0; \\
(v)& \quad -\alpha ^2\leq K(r)\leq -\beta ^2\quad  \operatorname{with}\quad \alpha > 0, \beta > 0 \quad \operatorname{and}\quad (n-1)\beta -p\alpha \geq 0;\\
(vi)& \quad K(r) = 0\quad  \operatorname{with}\quad  n - pk>0;\\
(vii)& -\frac A{(1+r^2)^{1+\epsilon}}\leq K(r) \leq \frac B{(1+r^2)^{1+\epsilon}}\quad  \operatorname{with}\quad \epsilon > 0\, , A \ge 0\, , 0 < B < 2\epsilon\quad \operatorname{and}\\
&\qquad n - (n-1)\frac B{2\epsilon} -p e^{\frac {A}{2\epsilon}} > 0.
\endaligned\label{10.2}
\end{equation}

Then every $p$-harmonic map $u: M \to N$ with the following $p$-energy growth
condition is a constant:
\begin{equation}
\frac 1p \int_{B_\rho(x_0)}|du|^p\, dv = o(\rho^\lambda )\quad \text{as
} \rho\rightarrow \infty\, , \label{10.3}
\end{equation}
where \begin{equation}
\lambda =\left\{
\begin{array}{cc}
1+(n-1)A_1-pA, & \text{if }\quad K(r) \text{ satisfies $($i$)$}, \\
1 + (n-1)\frac{1 + \sqrt {1+4A_1}}{2} -p\frac{1 + \sqrt {1+4A}}{2},& \text{if }\quad K(r) \text{ satisfies $($ii$)$}, \\
1 + (n-1)(|B-\frac 12|+\frac 12) -p\frac{1 + \sqrt {1+4B_1(1-B_1)}}{2}, & \text{if }\quad K(r) \text{ satisfies $($iii$)$}, \\
1+ (n-1)\frac{1 + \sqrt {1-4B}}{2} -p\frac{1 + \sqrt {1+4B_1}}{2}, & \text{if }\quad K(r) \text{ satisfies $($iv$)$\, },\\
n-p\frac \alpha \beta,  & \text {if }\quad  K(r)  \text { satisfies $($v$)$},\\
n-p, &\text {if } K(r) \text{ satisfies $($vi$)$},\\
n - (n-1)\frac B{2\epsilon} -p e^{\frac {A}{2\epsilon}}, &\text{if } K(r) \text{
satisfies $($vii$)$}\, .
\label{10.4}
\end{array}
\right.
\end{equation}

In particular,  every $p$-harmonic map $u: M \to N$ with finite $p-$energy is a constant.
\end{theorem}

\begin{proof} This follows immediately from Theorem \ref{T: 10.1} in which $F(t)=\frac 1p(2t)^{\frac p2}\, $ and $d_F = \frac p2\, .$
\end{proof}

\begin{corollary} Let $M\, ,$ $N\, ,$ $K(r)\, ,$ $\lambda\, $ and the growth
condition $(\ref{10.3})$ be as in Theorem \ref{T: 10.2}, in which $p=2\, .$ Then every harmonic map $u: M \to N$ is a constant.
\end{corollary}

\section{Generalized Yang-Mills-Born-Infeld
fields (with the plus sign) on Manifolds}
In \cite {SiSiYa}, Sibner-Sibner-Yang consider a variational problem which is a generalization of the (scalar valued) Born-Infeld model and at the same time a quasilinear generalization of the Yang-Mills theory. This motivates the study of Yang-Mills-Born-Infeld
fields on $\Bbb{R}^4$, and they prove that a generalized self-dual
equation whose solutions are Yang-Mills-Born-Infeld fields has no
finite-energy solution except the trivial solution on $\Bbb{R}^4$. In \cite {DW}, Dong and Wei introduced the following notions:
\begin{definition11.1}
The \emph{generalized Yang-Mills-Born-Infeld energy functional with the plus sign on a manifold} $M$ is the mapping $\mathcal{YM}_{BI}^{+} : \mathcal{C}\to \mathbb{R}^+\, $ given by
\begin{equation}
\mathcal{YM}_{BI}^{+}(\nabla )=\int_{M} \sqrt{1+|| R ^\nabla
||^2}-1 \quad  dv\, .\label{11.1}
\end{equation}
The generalized Yang-Mills-Born-Infeld energy functional with the negative sign on a manifold $M$ is the mapping $\mathcal{YM}_{BI}^{-} : \mathcal{C}\to \mathbb{R}^+\, $ given by
\[
\mathcal{YM}_{BI}^{-}(\nabla )=\int_{M}(1-\sqrt{1-|| R ^\nabla
||^2}) dv\, .
\]
The associate curvature form $R ^\nabla $ of a critical connection $\nabla $ of
$\mathcal{YM}_{BI}^{+}$ $($resp. $\mathcal{YM}_{BI}^{-})$ is called a generalized
\emph{Yang-Mills-Born-Infeld field with the plus sign} $($resp. \emph {with the
minus sign}$)$ on a manifold.
\end{definition11.1}

\begin{theorem}[$\operatorname{See}$ \cite {DW}]
Every generalized Yang-Mills-Born-Infeld field $($with the plus sign or with the minus sign$)$ on a manifold satisfies an $F$-conservation law.
\label{T: 11.1} \end{theorem}

\begin{theorem} \label{T: 11.2} Let the radial curvature $K(r)$ of $M$ satisfy one of seven conditions in $(\ref{8.12})$ in which $k=2\, $ and $d_F=1\, .$ Let $R ^\nabla $ be a generalized
Yang-Mills-Born-Infeld field with the plus sign on $M$. If
$R ^\nabla $ satisfies the following growth
condition
$$
\int_{B_\rho(x_0)} \sqrt{1+|| R ^\nabla ||^2}-1 \quad dv = o(\rho^{\lambda})\quad \text{as
} \rho\rightarrow \infty\, ,
$$
where

\begin{equation}
\lambda =\left\{
\begin{array}{cc}
1+(n-1)A_1-4A, & \text{if }\quad K(r) \text{ satisfies $($i$)$}, \\
1 + (n-1)\frac{1 + \sqrt {1+4A_1}}{2} -2\sqrt {1+4A}),& \text{if }\quad K(r) \text{ satisfies $($ii$)$}, \\
1 + (n-1)(|B-\frac 12|+\frac 12) -2\big (1 + \sqrt {1+4B_1(1-B_1)}\big ), &\text{if }\quad K(r) \text{ satisfies $($iii$)$}, \\
1+ (n-1)\frac{1 + \sqrt {1-4B}}{2} -2(1 + \sqrt {1+4B_1}), & \text{if}\quad K(r) \text{ satisfies $($iv$)$},\\
n-4\frac \alpha \beta,  & \text {if }\quad  K(r)  \text {satisfies $($v$)$},\\
n-4, &\text {if } K(r) \text{ satisfies $($vi$)$},\\
n - (n-1)\frac B{2\epsilon} - 4 e^{\frac {A}{2\epsilon}}, &\text{if } K(r) \text{
satisfies $($vii$)$},
\end{array}
\right.\label{11.2}
\end{equation}
then its curvature $R ^\nabla \equiv 0$. In particular, if
$R ^\nabla $ has finite $\mathcal{YM}_{BI}^{+}$-energy, then $R ^\nabla \equiv
0$.
\end{theorem}
\begin{proof} By applying
Theorem \ref {T: 11.1} and $F(t)=\sqrt{1+2t}-1$ to Theorem \ref{T: 8.1} in which $d_F=1\, ,$  and $k=2\, ,$ for $R ^\nabla \in A^2(Ad P)$, the result follows immediately.
\end{proof}

\section{Dirichlet Boundary Value Problems}

We recall \emph{$F$-lower degree} $l_F$ is defined to be
\begin{equation}
l_F=\inf_{t\geq 0}\frac{tF^{\prime }(t)}{F(t)}\, .\label{12.1}
\end{equation}
A bounded domain $D\subset M$ with $C^1$ boundary is called
\emph {starlike} ( relative to $x_0\, )$ if there exists an inner point $x_0\in D$ such that
\begin{equation}
\langle\frac \partial {\partial r_{x_0}},\nu\rangle |_{\partial
D}\geq 0\, , \label{12.2}
\end{equation}
where $\nu$ is the unit outer normal to $\partial
D\, ,$ and for any $x\in D \backslash \{x_0\} \cup \partial D\, ,$ $\frac {\partial} {\partial r_{x_0}}(x)$ is the unit vector field tangent to the unique geodesic emanating from $x_0$ to $x$.
\smallskip

It is obvious that any disc or convex domain is starlike.

\begin{theorem} \label{T: 12.1} Let $D$ be a bounded starlike domain $($relative to $x_0)$ with
$C^1$ boundary in a complete Riemannian $n$-manifold $M$. 
Let $\xi :E\rightarrow M$
be a Riemannian vector bundle on $M$ and $ \omega \in A^1(\xi )\, .	$
Assume that the radial curvature $K(r)$ of $M$ satisfies one of the
following seven conditions:
\begin{equation}
\aligned
(i)&\quad \eqref{3.14}\, \operatorname{holds}\,  \operatorname{with}\,  1+(n-1)A_1-2d_FA > 0; \\
(ii)&\quad \eqref{3.17}\,\, \operatorname{holds}\,   \operatorname{with}\,   1 + (n-1)\frac{1 + \sqrt {1+4A_1}}{2} -d_F(1 + \sqrt {1+4A}) > 0;\\
(iii)&\quad \eqref{3.49}\,\, \operatorname{holds}\,   \operatorname{with}\,   1 + (n-1)\big (|B-\frac 12|+\frac 12\big ) - d_F\big (1 + \sqrt {1+4B_1(1-B_1)}\big ) > 0;\\ 
(iv)&\quad \eqref{3.45}\, \operatorname{holds}\,  \operatorname{with}\,  1+ (n-1)\frac{1 + \sqrt {1-4B}}{2} -d_F(1 + \sqrt {1+4B_1}) > 0;\\
 (v)& \quad -\alpha ^2\leq K(r)\leq -\beta ^2\quad  \operatorname{with}\quad \alpha > 0, \beta > 0 \quad \operatorname{and}\quad (n-1)\beta -2\alpha d_F\geq 0;\\
(vi)& \quad K(r) = 0\quad  \operatorname{with}\quad  n-2d_F>0;\\
(vii)& -\frac A{(1+r^2)^{1+\epsilon}}\leq K(r) \leq \frac B{(1+r^2)^{1+\epsilon}}\quad  \operatorname{with}\quad \epsilon > 0\, , A \ge 0\, , 0 < B < 2\epsilon\quad \operatorname{and}\\
&\qquad n - (n-1)\frac B{2\epsilon} -2e^{\frac {A}{2\epsilon}}d_F > 0.
\endaligned\label{12.3}
\end{equation}
Assume that $ l_F\geq \frac 12$. If $\omega \in A^1(\xi
)$ satisfies $F$-conservation law and annihilates any tangent
vector $\eta $ of $\partial D$, then $\omega $ vanishes on $D$.
\end{theorem}

\begin{proof} By the assumption, there exists a point $x_0\in D$
such that the distance function $r_{x_0}$ relative to $x_0\, $ satisfies (\ref{12.2}). Take
$X=r\nabla r$, where $r=r_{x_0}$. From the proof of Theorem 8.1,
we know that
\begin{equation}
\langle S_{F,\omega },\nabla X^{\flat}\rangle\geq
\,\lambda F(\frac{|\omega |^2}2)\tag{8.15}
\end{equation}
in $D\, ,$ where $\lambda\, $ is a positive constant given by 
\begin{equation}
\lambda =\left\{
\begin{array}{cc}
1+(n-1)A_1-2d_FA, & \text{if }\quad K(r) \text{ satisfies $($i$)$}, \\
1 + (n-1)\frac{1 + \sqrt {1+4A_1}}{2} - d_F(1 + \sqrt {1+4A}),& \text{if }\quad K(r) \text{ satisfies $($ii$)$}, \\
1 + (n-1)(|B-\frac 12|+\frac 12) - d_F\big (1 + \sqrt {1+4B_1(1-B_1)}\big ), & \text{if }\quad K(r) \text{ satisfies $($iii$)$}, \\
1+ (n-1)\frac{1 + \sqrt {1-4B}}{2} -d_F ( 1 + \sqrt {1+4B_1} ), & \text{if }\quad K(r) \text{ satisfies $($iv$)$}, \\
n-2\frac \alpha \beta d_F, & \text {if }\quad  K(r)  \text { satisfies $($v$)$},\\
n-2d_F, &\text {if }\quad K(r) \text{ satisfies $($vi$)$},\\
n - (n-1)\frac B{2\epsilon} -2 e^{\frac {A}{2\epsilon}}d_F, &\text{if }\quad K(r) \text{
satisfies $($vii$)$}\, .
\end{array}
\right.
\label{12.4}
\end{equation}
Since $\omega \in A^1(\xi )$
annihilates any tangent vector $\eta$ of $\partial D$, we easily
derive via \eqref{8.3} and \eqref{12.1}, the following inequality on $\partial D$
\begin{equation}
\aligned S_{F,\omega }(X,\nu)&=rS_{F,\omega }(\frac
\partial
{\partial r},\nu) \\
&= r\bigg (F(\frac{|\omega |^2}2)\langle\frac \partial {\partial r},\nu
\rangle-F^{\prime }(\frac{|\omega |^2}2)\langle\omega (\frac
\partial
{\partial r}),\omega (\nu )\rangle\bigg )\\
&=r\langle\frac \partial {\partial r},\nu
\rangle\bigg (F(\frac{|\omega |^2}2)-2F^{\prime }(\frac{|\omega |^2}2)\frac{|\omega |^2}2\bigg )\\
&\leq r\langle\frac \partial {\partial r},\nu \rangle
F(\frac{|\omega |^2}2)(1-2l_F)\leq 0\, .
\endaligned \label{12.5}
\end{equation}
From (\ref{8.8}), (\ref{12.3}) and (\ref{12.5}), we have
$$
0 \le \int_D \lambda F(\frac{|\omega |^2}2) dv \leq 0\, ,
$$
which implies that $\omega \equiv 0$.  \end{proof}

\begin{theorem}$($Dirichlet problems for $F$-harmonic maps$)$ Let $M$, $D$, and  $\xi$ be as in Theorem $\ref{T: 12.1}$. Assume that the radial curvature $K(r)$ of $M$ satisfies one of the
following seven conditions:
\begin{equation}
\aligned
&(i)\quad \eqref{3.14}\, \operatorname{holds}\,  \operatorname{with}\,  1+(n-1)A_1-2d_FA > 0\, ;\\
&(ii)\quad \eqref{3.17}\, \operatorname{holds}\,   \operatorname{with}\,   1 + (n-1)\frac{1 + \sqrt {1+4A_1}}{2} -d_F(1 + \sqrt {1+4A}) > 0;\\
&(iii)\quad  -\alpha ^2\leq K(r)\leq -\beta ^2\, \operatorname{with}\,  \alpha >0\, , \beta
>0\, \operatorname{and}\,  (n-1)\beta -2d_F\alpha \geq 0;\\
&(iv)\quad  K(r) = 0\, \operatorname{with}\,  n-2d_F >0;\\
&(v)\quad  -\frac A{(1+r^2)^{1+\epsilon}} \le K(r) \leq \frac B{(1+r^2)^{1+\epsilon}}\, \operatorname{with}\,  \epsilon > 0 \, , A \ge 0\, , 0 < B < 2\epsilon \,  ,\operatorname{and}\\
& \qquad \qquad  n-(n-1)\frac{B}{2\epsilon} - 2e^{\frac{A}{2\epsilon}} d_F>0\, ;\\
&(vi)\quad \eqref{3.49}\, \operatorname{holds}\,   \operatorname{with}\,   1 + (n-1)(|B-\frac 12|+\frac 12) - d_F\big (1 + \sqrt {1+4B_1(1-B_1)}\big ) > 0\, ;\\ 
&(vii)\quad \eqref{3.45}\, \operatorname{holds}\,  \operatorname{with}\,  1+ (n-1)\frac{1 + \sqrt {1-4B}}{2} -d_F (1 + \sqrt {1+ 4B_1}) > 0\, .\endaligned\label{12.6}
\end{equation}
Let $u:\overline{D}\rightarrow N$ be
an $F$-harmonic map with $l_F \ge \frac{1}{2}$ into an arbitrary Riemannian
manifold $N$. If $u|_{\partial D}$ is constant, then $u|_D$ is
constant. \label{T: 12.2}
\end{theorem}
\begin{proof} Take $\omega =du$. Then $\omega|_{\partial D} = 0$. Hence $\omega$ satisfies an $F$-conservation law and annihilates any tangent
vector $\eta $ of $\partial D\, .$ The assertion  
follows at once from Theorem \ref{T: 12.1} and \cite [Theorem 6.1]{DW}.
\end{proof}
\begin{corollary}\label{C:12.1} Suppose $M$ and $D$ satisfy the same
assumptions of Theorem \ref{T: 12.2}. Let $u:\overline{D}\rightarrow N$ be
a $p$-harmonic map $(p\geq 1)$ into an arbitrary Riemannian
manifold $N$. If $u|_{\partial D}$ is constant, then $u|_D$ is
constant.
\end{corollary}
\begin{proof} For a $p$-harmonic map $u$, we have $F(t)=\frac 1p
(2t)^{\frac{p}{2}}$. Obviously $d_F=l_F=\frac p2$. Take $\omega =du$. This
corollary follows immediately from Theorem \ref{T: 12.1} or Theorem \ref{T: 12.2}.
\end{proof}

\end{document}